\documentclass{amsart}
\pdfoutput=1
\usepackage{a4wide}
\usepackage{comment}
\usepackage{microtype}
\usepackage{stackengine}
\stackMath
\usepackage{amscd}
\usepackage{amsmath,amssymb,bbm,amsfonts, accents}
\usepackage{mathtools}
\usepackage{xstring}
\usepackage{xcolor}
\usepackage{eucal}
\usepackage{url}
\usepackage{latexsym}
\usepackage{amsthm}
\usepackage[latin1]{inputenc}
\usepackage{multicol}
\usepackage[shortlabels]{enumitem}
\usepackage[colorlinks=true, linkcolor={blue!50!black}, pdfhighlight=/O, 
ocgcolorlinks=true]{hyperref}
\hypersetup{bookmarksdepth=2}
\usepackage{graphicx}
\usepackage{color}
\usepackage{tikz-cd}
\usepackage[all]{xy}
\usepackage[normalem]{ulem}

\newcommand{\eg}{e.g.\@}
\newcommand{\ie}{i.e.\@}

\newcommand{\IFF}{if and only if}
\usepackage[capitalise]{cleveref}
\crefformat{equation}{(#2#1#3)}
\newcommand{\lortho}{curved orthofibration}
\newcommand{\lorthos}{curved orthofibrations}

\newcommand{\Lorthos}{Curved orthofibrations}
\newcommand{\Grayop}{\mathrm{OpGray}}
\newcommand{\Ortholax}{\mathrm{CrvOrtho}}
\newcommand{\LOrth}{\Ortholax}

\newcommand{\Cocartlax}{\coCartlax}
\newcommand{\Cocart}{\cocart}
\newcommand{\Cart}{\cart}
\newcommand{\LFib}{\Leftfib}
\newcommand{\RFib}{\Rightfib}
\newcommand{\tFun}[0]{\tcat{Fun}}
\DeclareMathOperator{\LCocart}{LCocart}
\DeclareMathOperator{\RCocart}{RCocart}
\DeclareMathOperator{\RCart}{RCart}
\DeclareMathOperator{\LCart}{LCart}
\newcommand{\Orth}{\ortho}

\makeatletter
\def\@tocline#1#2#3#4#5#6#7{\relax
  \ifnum #1>\c@tocdepth 
  \else
    \par \addpenalty\@secpenalty\addvspace{#2}%
    \begingroup \hyphenpenalty\@M
    \@ifempty{#4}{%
      \@tempdima\csname r@tocindent\number#1\endcsname\relax
    }{%
      \@tempdima#4\relax
    }%
    \parindent\z@ \leftskip#3\relax \advance\leftskip\@tempdima\relax
    \rightskip\@pnumwidth plus4em \parfillskip-\@pnumwidth
    #5\leavevmode\hskip-\@tempdima
      \ifcase #1
       \or \hskip -1em \or \hskip 1em \or \hskip 3em \else \hskip 5em \fi%
      #6\nobreak\relax
    \hfill\hbox to\@pnumwidth{\@tocpagenum{#7}}
      \par
    \nobreak
    \endgroup
  \fi}
\makeatother

\DeclareFontFamily{U}{cbgreek}{}
\DeclareFontShape{U}{cbgreek}{m}{n}{
        <-6>    grmn0500
        <6-7>   grmn0600
        <7-8>   grmn0700
        <8-9>   grmn0800
        <9-10>  grmn0900
        <10-12> grmn1000
        <12-17> grmn1200
        <17->   grmn1728
      }{}
\DeclareFontShape{U}{cbgreek}{bx}{n}{
        <-6>    grxn0500
        <6-7>   grxn0600
        <7-8>   grxn0700
        <8-9>   grxn0800
        <9-10>  grxn0900
        <10-12> grxn1000
        <12-17> grxn1200
        <17->   grxn1728
      }{}

\makeatletter
\newcommand{\normalorbold}{%
  \ifnum\pdf@strcmp{\math@version}{bold}=\z@ bx\else m\fi
}
\makeatother

\newtheoremstyle{introthms}
	{}{}{\itshape}{}{\bfseries }{}{ }
	{\thmname{#1} \thmnumber{#2}. \thmnote{\bfseries{(#3)}}}

\theoremstyle{introthms}
\newtheorem{introthm}{Theorem}
\newtheorem{introprop}[introthm]{Proposition}
\newtheorem{introcor}[introthm]{Corollary}

\swapnumbers

\numberwithin{equation}{subsection}
\theoremstyle{plain}
\newtheorem{theorem}[equation]{Theorem}
\newtheorem{thm}[equation]{Theorem}

\newtheorem{lemma}[equation]{Lemma}

\newtheorem{proposition}[equation]{Proposition}
\newtheorem{propn}[equation]{Proposition}

\newtheorem{corollary}[equation]{Corollary}
\newtheorem{cor}[equation]{Corollary}

\theoremstyle{definition}

\newtheorem{definition}[equation]{Definition}
\newtheorem{defn}[equation]{Definition}

\newtheorem{construction}[equation]{Construction}
\newtheorem{example}[equation]{Example}

\newtheorem{observation}[equation]{Observation}
\newtheorem{remark}[equation]{Remark}

\newtheorem{notation}[equation]{Notation}

\newtheorem*{uremark}{Remark}

\setenumerate{label=(\arabic*),itemsep=3pt,topsep=3pt, leftmargin=*}
\setitemize{itemsep=3pt,topsep=3pt, leftmargin=*}

\newcommand{\cocolon}{\nobreak \mskip6mu plus1mu \mathpunct{}\nonscript\mkern-\thinmuskip {:}\mskip2mu \relax}

\newcommand{\ol}[1]{\overline{#1}}
\newcommand{\CAT}{\textsc{Cat}}
\newcommand{\tCat}{\textbf{\textup{Cat}}}
\newcommand{\C}{\cat{C}}
\newcommand{\D}{\cat{D}}

\newcommand{\cat}[1]{
\StrLen{#1}[\mystrlen]
\ifnum\mystrlen=1 {#1}
\else \mathrm{#1}
\fi}

\newcommand{\sSet}[0]{\cat{sSet}}
\newcommand{\Gpd}{\cat{Gpd}}
\newcommand{\sS}[0]{\Gpd}

\newcommand{\colim}{\operatornamewithlimits{\mathrm{colim}}}

\newcommand{\mm}[1]{\mathrm{#1}}
\newcommand{\Hom}[0]{\mm{Hom}}
\newcommand{\Map}[0]{\mm{Map}}
\newcommand{\rto}[1]{\stackrel{#1}{\rt}}

\newcommand{\rt}[0]{\longrightarrow}

\newcommand{\Fun}[0]{\cat{Fun}}

\newcommand{\op}[0]{\mm{op}}
\newcommand{\id}[0]{\mm{id}}

\newcommand{\pr}{\mathrm{pr}}

\newcommand{\CC}[0]{\cat{C}}
\newcommand{\DD}[0]{\cat{D}}
\newcommand{\Cat}{\cat{Cat}}
\newcommand{\core}{\iota}

\DeclareMathOperator{\ortho}{Ortho}
\DeclareMathOperator{\bifib}{Bifib}
\newcommand{\Bifib}{\bifib}
\DeclareMathOperator{\cocart}{Cocart}
\DeclareMathOperator{\cart}{Cart}
\newcommand{\Strco}{\mathrm{Str}^{\mathrm{cc}}}

\newcommand{\Strcart}{\mathrm{Str}^{\mathrm{ct}}}

\newcommand{\Strsc}{\mathrm{Str}^{\scale}}
\newcommand{\Unco}{\mathrm{Un}^{\mathrm{cc}}}
\newcommand{\Uncart}{\mathrm{Un}^{\mathrm{ct}}}

\newcommand{\Unsc}{\mathrm{Un}^{\scale}}
\newcommand{\Dualco}{\mathrm{D}^\mathrm{cc}}
\newcommand{\Dualcart}{\mathrm{D}^\mathrm{ct}}

\DeclareMathOperator{\Tw}{Tw}
\DeclareMathOperator{\Ar}{Ar}

\newcommand{\TwL}{\Tw^{\ell}}
\newcommand{\TwR}{\Tw^{r}}
\newcommand{\TwLB}{\TwL_{B}}

\newcommand{\Cartlax}{\cart^{\mathrm{opl}}}
\newcommand{\coCartlax}{\cocart^{\mathrm{lax}}}
\newcommand{\Rightfib}{\mathrm{RFib}}
\newcommand{\Leftfib}{\mathrm{LFib}}
\newcommand{\leftfib}{\Leftfib}
\newcommand{\rightfib}{\Rightfib}

\newcommand{\Gray}{\mathrm{Gray}}
\newcommand{\localarrowtail}{\stackengine{0pt}{\rightarrowtail}{\circ}{O}{c}{F}{F}{L}}
\newcommand{\localarrowhead}{\stackengine{0pt}{\twoheadrightarrow}{\circ\phantom{|}}{O}{c}{F}{F}{L}}

\newcommand{\nerve}{\mathrm N}

\newcommand{\lococart}{\mathrm{LocCocart}}
\newcommand{\LocCocart}{\lococart}
\newcommand{\locart}{\mathrm{LocCart}}
\newcommand{\LocCart}{\locart}
\newcommand{\lax}{\mathrm{lax}}
\newcommand{\oplax}{\mathrm{opl}}
\newcommand{\OO}{\cat{O}}
\newcommand{\pset}[1]{\langle #1\rangle}
\newcommand{\MonCat}{\tcat{MonCat}}
\newcommand{\OMonCat}{\MonCat_{\OO}}
\newcommand{\Adj}{\mm{Adj}}
\newcommand{\Gtimes}{\boxtimes}
\newcommand{\tcat}[1]{\mathbf{#1}}
\newcommand{\scale}{\mathrm{sc}}
\newcommand{\markSet}{\mathrm{sSet}^+}
\newcommand{\markCat}{\mathrm{Cat}^+_{\Delta}}
\newcommand{\scSet}{\mathrm{sSet}^{\scale}}

\tikzset{Isom/.style={above,every to/.append style={edge node={node [sloped, allow upside down, auto=false]{$\sim$}}}}}

\newcommand{\icat}{$\infty$-category}
\newcommand{\iopd}{$\infty$-operad}
\newcommand{\icats}{$\infty$-categories}
\newcommand{\itcat}{$(\infty,2)$-category}
\newcommand{\itcats}{$(\infty,2)$-categories}
\newcommand{\igpds}{$\infty$-groupoids}
\newcommand{\igpd}{$\infty$-groupoid}
\newcommand{\Fcoc}[1]{F^{\mathrm{cc}}_{#1}}
\newcommand{\Fcart}[1]{F^{\mathrm{ct}}_{#1}}
\newcommand{\coc}{\mathrm{cc}}
\newcommand{\cac}{\mathrm{ct}}
\newcommand{\xto}[1]{\overset{#1}{\rightarrow}}
\newcommand{\ev}{\mathrm{ev}}

\DeclareMathOperator{\Corr}{Corr}
\DeclareMathOperator{\corr}{corr}
\newcommand{\isoto}{\rto{\sim}}

\newcommand{\csquare}[8]{%
  \begin{tikzcd}%
     #1 \arrow{r}{#5} \arrow{d}[swap]{#6} \pgfmatrixnextcell #2 %
      \arrow{d}{#7} \\%
     #3 \arrow{r}{#8} \pgfmatrixnextcell #4%
   \end{tikzcd}%
   }
   \newcommand{\nlcsquare}[4]{\csquare{#1}{#2}{#3}{#4}{}{}{}{}}

   \newcommand{\pF}{\mathrm{Fin}_{*}}
   \newcommand{\Dop}{\Delta^{\op}}

\title[Lax monoidal adjunctions, two-variable fibrations and the calculus of mates]{Lax monoidal adjunctions, two-variable fibrations\\ and the calculus of mates}

\author[R.~Haugseng]{Rune Haugseng}
\address{Institutt for matematiske fag, NTNU Trondheim, Norway}
\email{rune.haugseng@ntnu.no}
\author[F.~Hebestreit]{Fabian Hebestreit}
\address{Department of Mathematics, University of Aberdeen, UK}
\email{fabian.hebestreit@abdn.ac.uk}
\author[S.~Linskens]{Sil Linskens}\address{Mathematisches Institut, RFWU Bonn, Germany}
\email{linskens@math.uni-bonn.de}
\author[J.~Nuiten]{Joost Nuiten}
\address{IMT, Universit\'e de Toulouse III, France}
\email{joost.nuiten@math.univ-toulouse.fr}

\begin{document}
	
	\begin{abstract}
		We provide a calculus of mates for functors to the $\infty$-category of $\infty$-categories and extend Lurie's unstraightening equivalences to show that (op)lax natural transformations correspond to maps of (co)cartesian fibrations that do not necessarily preserve (co)cartesian edges. As a sample application we obtain an equivalence between lax symmetric monoidal structures on right adjoint functors and oplax symmetric monoidal structures on the left adjoint functors between symmetric monoidal $\infty$-categories that is compatible with both horizontal and vertical composition of such structures.
		
		As the technical heart of the paper we study various new types of fibrations over a product of two $\infty$-categories. In particular, we show how they can be dualised over one of the two factors and how they encode functors out of the Gray tensor product of $(\infty, 2)$-categories.
	\end{abstract}
	\setcounter{tocdepth}{1}
	\maketitle
	\tableofcontents

	\section{Introduction}
		The goal of the present paper is to establish a version of the \emph{calculus of mates} for diagrams of \icats{}, encoded in terms of (co)cartesian fibrations. Recall that if we have an \emph{oplax square} in a 2-category, that is a diagram of the form
	\[
	\begin{tikzcd}
		\bullet \arrow{r}{L} \arrow{d}[swap]{f} & \bullet \arrow{d}{g} \\
		\bullet \arrow{r}[swap]{L'} \arrow[Rightarrow, shorten=1.5ex]{ur}{\alpha} & \bullet
	\end{tikzcd}
	\]
	given by a 2-morphism $\alpha \colon L'f \to gL$, where $L$ and $L'$ have right adjoints $R$ and $R'$, then we can form the \emph{mate square}, which is the lax square
	\[
	\begin{tikzcd}
		\bullet \arrow{r}{R} \arrow{d}[swap]{g} & \bullet \arrow{d}{f} \\
		\bullet \arrow{r}[swap]{R'} \arrow[Leftarrow, shorten=1.5ex]{ur}{\beta} & \bullet,
	\end{tikzcd}
	\]
		where $\beta$ is the \emph{Beck--Chevalley transformation}
	\[ fR \rt R'L'fR \stackrel{\alpha}{\rt} R'gLR \rt R'g \]
	defined using the unit of the adjunction $L' \dashv R'$ and the counit of $L \dashv R$. It is not hard to show that this procedure and its dual give an isomorphism between oplax squares with horizontal left adjoints and lax squares with horizontal right adjoints. Moreover, this \emph{mate correspondence} is compatible with compositions of squares both vertically and horizontally. A useful way to encode some of this functoriality is in terms of \emph{(op)lax natural transformations}. If $C,D$ are 2-categories, a lax transformation $\phi$ between functors $F, G \colon C \to D$ has lax naturality squares
	\[
	\begin{tikzcd}
		F(x) \arrow{r}{\phi_{x}} \arrow{d}[swap]{F(f)} & G(x) \arrow{d}{G(f)} \\
		F(y) \arrow{r}[swap]{\phi_{y}} \arrow[Leftarrow, shorten=1.5ex]{ur} & G(y)
	\end{tikzcd}
	\]
		for each morphism $f \colon x \to y$ in $C$, and similarly for oplax transformations. By taking mates we obtain an equivalence between oplax transformations given pointwise by left adjoints and lax transformations given pointwise by right adjoints.
		
		In this paper we will prove an $\infty$-categorical version of this equivalence in the special case where $C$ is an \icat{} and $D$ is the $(\infty,2)$-category of \icats{}. This arises as a combination of our two main results: on the one hand, we will prove a version of the mate correspondence for (co)cartesian fibrations (\cref{thm:fibadj}), and on the other we will provide a fibrational description of (op)lax transformations (\cref{laxstr}). 
	
	\subsection*{Lax monoidal adjunctions}
	Since many $\infty$-categorical structures are conveniently encoded using fibrations, we believe this fibrational approach to mates is very useful in its own right. Before we discuss our general results in more detail, we will look at a concrete special case that illustrates this point: 
	the left adjoint of a lax symmetric monoidal functor between $\infty$-categories admits a canonical oplax symmetric monoidal structure, and vice versa. To state this more precisely, let us briefly recall the definitions of symmetric monoidal \icats{} and (op)lax monoidal functors among them:
		
		If $X$ is an \icat{} with finite products and $\pF$ denotes the  category of finite pointed
		sets, a \emph{commutative monoid} in $X$ can be defined as a functor 
		$M \colon \pF \to X$
		such that the Segal maps $M_{\langle n \rangle} \to \prod_{i=1}^{n}
		M_{\langle 1 \rangle}$ are equivalences; here $\langle n \rangle = \{0, \dots, n\}$ is pointed by $0$ and
		the map is induced by the projections $\langle n \rangle \to \langle 1 \rangle$ that send all but one element to the base point.
		(The underlying object of the monoid is $M_{\langle 1 \rangle}$.) Such a commutative monoid in the \icat{} $\Cat$ of small \icats{} is then
		a \emph{symmetric monoidal \icat{}}. Moreover, a homomorphism of commutative monoids is simply a natural
		transformation of functors $\pF \to X$; in the case of symmetric monoidal
		\icats{} these correspond to \emph{strong} symmetric monoidal functors.
		
		A functor $B \to \Cat$ can be encoded as either a cocartesian fibration over $B$ or a cartesian fibration over $B^{\op}$. A symmetric monoidal structure on an \icat{} $C$ can therefore be described either by a cocartesian fibration $C^{\otimes} \to \pF$ or a cartesian fibration $C_{\otimes} \to \pF^{\op}$. 
                  In terms of these fibrations, a strong monoidal functor from $C$ to $D$ corresponds to a commutative triangle of either of the forms
	\[
	\begin{tikzcd}
		C^{\otimes} \arrow{rr}{f^{\otimes}} \arrow{dr} & & D^{\otimes} \arrow{dl} \\
		& \pF,
	\end{tikzcd} \quad \quad \quad \begin{tikzcd}
			C_{\otimes} \arrow{rr}{f_{\otimes}} \arrow{dr} & & D_{\otimes} \arrow{dl} \\
			& \pF^{\op},
		\end{tikzcd}
	\]
		where $f^{\otimes}$ preserves cocartesian morphisms and $f_\otimes$ preserves cartesian morphisms.
		
		By weakening the conditions on $f^{\otimes}$ and $f_\otimes$ we obtain good definitions of
		lax and oplax symmetric monoidal functors. To this end, recall that a morphism
		$\phi \colon \langle n \rangle \to \langle m \rangle$ in $\pF$ is
		called \emph{inert} if it restricts to an isomorphism
		$\langle n \rangle \setminus \{\phi^{-1}(0)\} \isoto \langle m \rangle
		\setminus \{0\}$. In \cite{HA} Lurie defines a lax symmetric
		monoidal functor to be a commutative triangle as on the left where $f^{\otimes}$ only preserves the cocartesian morphisms that cover inert morphisms in $\pF$. Informally, we can think of an object of $C^{\otimes}$ over $\langle n \rangle$ as a list $(c_{1},\ldots,c_{n})$ of objects in $C$, and this condition should be thought of as requiring the image of this object in $D^{\otimes}$ to be the list $(f(c_{1}),\ldots,f(c_{n}))$. For objects $x,y \in C$ we have a cocartesian morphism starting at $(x,y)  \in C^{\otimes}_{\langle 2 \rangle}$ covering the map $\langle 2 \rangle  \rightarrow \langle 1\rangle$ that takes both $1$ and $2$ to $1$. Its target encodes (essentially by definition) the tensor product $x \otimes y$. This morphism gets sent to a morphism $(fx,fy) \to f(x \otimes y)$ in $D^{\otimes}$, which may no longer be cocartesian; taking a cocartesian factorisation we then obtain the lax monoidal structure map $fx \otimes fy \to f(x \otimes y)$ in $D$.
                
		 If we instead decide to similarly relax the conditions on $f_\otimes$ and require that it only preserve the cartesian morphisms covering inerts, we obtain Lurie's definition of an oplax symmetric monoidal functor. Namely now the tensor product of $x$ and $y$ in $C$ is encoded by a cartesian morphism $x \otimes y \to (x,y)$ in $C_\otimes$ covering the same map as before. Its image in $D_{\otimes}$ factors into a cartesian edge preceded by a map $f(x \otimes y) \to f(x) \otimes f(y)$, which by definition is the oplax structure map of $f$. We remark that the opposite functor
		$(C_{\otimes})^{\op} \to \pF$ is actually the cocartesian fibration
		which encodes the natural symmetric monoidal structure on $C^{\op}$ and so an oplax
		monoidal functor can also be described as a lax monoidal functor
		$C^{\op} \to D^{\op}$, as the name suggest.
		
In this fibrational context, we can easily deduce the following from our general results:
		\begin{introprop}\label{propa}
			Given two symmetric monoidal $\infty$-categories $\cat{C}$ and
			$\cat{D}$, the extraction of adjoints gives inverse equivalences between
			the $\infty$-category of lax symmetric monoidal right adjoints
			$\cat{C}^{\otimes} \rightarrow \cat{D}^{\otimes}$ and the opposite of the
			$\infty$-category of oplax monoidal left adjoints
			$\cat{D}_{\otimes} \rightarrow \cat{C}_{\otimes}$.
		\end{introprop}
		
		In fact, we will produce an equivalence of $(\infty,2)$-categories
		that also encodes the compatibility of taking mates with compositions
		of (op)lax monoidal functors. We note that since the first version of
		this paper appeared, another proof of \cref{propa} has been given by
		Torii in \cite{Torii2,Torii3}, based on the two universal properties
		of Day convolution on presheaves. Let us also say immediately that in
		\cite{HA} Lurie already proved that the right adjoint of a strong
		monoidal functor admits a lax monoidal structure, which suffices for a
		great many applications. Moreover, Torii has more generally produced a
		lax monoidal structure (but none of the accompanying coherences) on
		the right adjoint of an oplax monoidal functor in \cite{Torii1}, by
		means of a span category construction. (We compare his construction to
		ours in \cite{part2}.)

	\subsection*{Parametrised adjunctions}
	With the example of lax monoidal adjunctions in mind it hopefully seams reasonable that the fibrational calculus of mates should relate fibrewise right adjoint functors between cocartesian fibrations and fibrewise left adjoint functors between the corresponding cartesian fibrations. Our first main result shows that this holds generally:
	
	\begin{introthm}\label{thm:fibadj}
		Let $B$ be an $\infty$-category. Then there is an equivalence of $(\infty,2)$-categories
		\[
		\tcat{Cocart}^{\lax, \mm{R}}(B)\simeq \Big(\tcat{Cart}^{\oplax, \mm{L}}(B^{\op})\Big)^{(1, 2)\mm{-}\op}
		\]
		extracting adjoints fibrewise; here the left-hand side denotes the
		$(\infty,2)$-category with cocartesian fibrations over $B$ as objects,
		fibrewise right adjoint functors (that need not preserve cocartesian
		lifts) as $1$-morphisms, and natural transformations between these as
		$2$-morphisms. The right hand side is defined dually, using
		cartesian fibrations and fibrewise right adjoints, with the directions
		of $1$- and $2$-morphisms reversed by the superscript. Furthermore, these equivalences are natural in pulling back along the base.
	\end{introthm}
	Taking $B$ to be an $\infty$-operad and taking appropriate subcategories cut out by the Segal conditions, this specialises to give: 
	\begin{introcor}\label{cord}
		For any $\infty$-operad $\OO$, the extraction of adjoints gives a canonical equivalence of $(\infty, 2)$-categories
		\[
		\OMonCat^{\lax, \mm{R}}\simeq \Big(\OMonCat^{\oplax, \mm{L}}\Big)^{(1, 2)\mm{-op}},
		\]
		where the left-hand side denotes the $(\infty,2)$-category of $\OO$-monoidal \icats{}, lax $\OO$-monoidal functors that admit (objectwise) left adjoints, and $\OO$-monoidal transformations; the right-hand side is defined dually using oplax $\OO$-monoidal functors that admit right adjoints. Furthermore these equivalences are natural in pulling back along operad maps in the base.
	\end{introcor}
	Proposition \ref{propa} is contained in this statement by taking
	$\OO = \mathbb E_{\infty}$ and examining the morphism \icats{} between
	two symmetric monoidal \icats{} $\C$ and $\D$. We also use Corollary
	\ref{cord} to extend a result of Hinich: We show that the internal
	mapping functor
	\[[-,-] \colon \C^\op \times \C \longrightarrow \C\]
	in a closed $\mathbb E_{n+1}$-monoidal \icat{} $\C$ admits a
	canonical lax $\mathbb E_n$-monoidal structure, where $1 \leq n \leq
	\infty$; in \cite{Hinich} Hinich established the case $n=\infty$ by
	different means. \\
	
	In order to prove Theorem~\ref{thm:fibadj}, we need to describe
	functors to the full subcategories $\Cocartlax(B)$ and $\Cartlax(B)$
	of cocartesian and cartesian fibration in $\Cat/B$ in terms of
	fibrations.
	
	We show that functors $A \to \Cocartlax(B)$ correspond under covariant
	unstraightening to functors
	$p=(p_{1},p_{2}) \colon E \to A \times B$ such that
	\begin{enumerate}[(1)]
		\item $p_{1}$ is a cocartesian fibration,
		\item $p_{1}$-cocartesian morphisms map to equivalences under $p_{2}$,
		\item for every $a \in A$ the functor $(p_{2})_{a} \colon E_{a}\to B$ on
		fibres over $a$ is a cocartesian fibration.
	\end{enumerate} 
	We call such a functor a \emph{Gray fibration}, for reasons that will
	become clear in a moment. Dually, functors $A^{\op} \to \Cocartlax(B)$ correspond under
	contravariant unstraightening to functors 
	$p = (p_{1},p_{2}) \colon E \to A \times B$ such that
	\begin{enumerate}[(1)]
		\item $p_{1}$ is a cartesian fibration,
		\item $p_{1}$-cartesian morphisms map to equivalences under $p_{2}$,
		\item for every $a \in A$ the functor $(p_{2})_a \colon E_{a}\to B$ on
		fibres over $a$ is a cocartesian fibration.
	\end{enumerate}
	We call these functors \emph{curved orthofibrations}. While the notion of Gray fibrations admits a cartesian dual, \emph{op-Gray fibrations}, which encode functors $A^\op \rightarrow \Cart^\oplax(B)$, the key point for our proofs is that curved orthofibrations are self-dual in the sense that they can also be characterised  by $p_2$ being a cocartesian fibration, $p_2$-cocartesian morphisms mapping to equivalences under $p_1$ and the functors $(p_{1})_b \colon E_b \rightarrow A$ being cartesian fibrations. They can therefore also be straightened covariantly in the second variable, and hence
	correspond to functors $B \to \Cartlax(A)$.

	Combining these one-variable straightenings, we obtain the following ``dualisation'' equivalences:
	\begin{introthm}\label{thm:grayorthdual}
		There is a natural equivalence of \icats{}
		\[ \Gray(A,B) \simeq \LOrth(A^{\op},B) \quad \text{and} \quad \Grayop(A,B) \simeq \LOrth(A,B^\op);\]
		here $\Gray(A,B)$ and $\Grayop(A,B)$ are the \icat{} of Gray fibrations and op-Gray fibrations over $A \times
		B$, respectively, while $\LOrth(A^{\op},B)$ is the \icat{} of \lorthos{} over $A^{\op}
		\times B$ and in both cases the morphisms are required to
		preserve the defining (co)cartesian morphisms.
	\end{introthm}
	
	Special cases of this duality were already known. For example,
	bifibrations are precisely those curved orthofibration whose fibres
	are $\infty$-groupoids, and under the equivalences above they
	correspond precisely to the left and right fibrations,
	respectively. An equivalence of this kind was first established by
	Stevenson in \cite{Stevenson} by different means.
	
	To see the relation of these results on two-variable fibrations to
	Theorem~\ref{thm:fibadj}, let us explain how the equivalence we build
	acts on a morphism $f \colon D \to E$ of cartesian fibrations over $B$
	which is given fibrewise by left adjoints:
	\begin{enumerate}[(1)]
		\item First, $f$ can be covariantly unstraightened to a \lortho{} over $B
		\times [1]$.
		\item Using Theorem~\ref{thm:grayorthdual}, this corresponds to a
		Gray fibration over $B^{\op} \times [1]$. Furthermore, because $f$ is given fibrewise by left adjoints, this Gray fibration is also a \lortho{} over $[1] \times
		B^{\op}$.
		\item Therefore, this new curved orthofibration can be covariantly straightened to a functor $[1]^{\op} \to
		\Cocartlax(B^{\op})$, corresponding to a morphism $E^{\vee} \to
		D^{\vee}$ over $B^{\op}$ between the cocartesian fibrations dual to
		the cartesian fibrations we started with.
	\end{enumerate}
	Here the second half of the second step is the parametrised analogue of the statement that adjunctions among $\infty$-categories can be encoded by functors to $[1]$ that are both cocartesian and cartesian fibrations, with the left and right adjoint obtained by cocartesian and cartesian straightening, respectively. At the level of spaces of objects, a very similar argument appears in work of Ayala, Mazel-Gee and Rozenblyum, where the base $B$ is allowed to be an $(\infty, 2)$-category \cite[Lemma B.5.7]{AyalaM}. Their work also contains \Cref{cord} at the level of objects, as \cite[Remark 4.1.7]{AyalaM}.  \\

	\subsection*{Unstraightening lax natural transformations}
	So far we have explained the fibrational perspective on lax natural transformation and how in this context we produce an analogue of the calculus of mates. However, there is a second perspective on lax natural transformations which is more intrinsic to $(\infty,2)$-category theory. 
	To explain it, recall that in terms of diagrams of \icats{}, a lax natural transformation $f$ between two functors $F, G \colon B \rightarrow \Cat$ should be encoded by maps $F(b)\to G(b)$ for each $b$, together with lax commuting squares
	\[\begin{tikzcd}
		F(b)\arrow[d, "\beta_!"{swap}]\arrow[r, "f_{b}"] & F(b)\arrow[d, "\beta_!"]\arrow[ld, Rightarrow, shorten=1.5ex,, "f_\beta"]\\
		F(b')\arrow[r, "f_{b'}"{swap}] & F(b')
	\end{tikzcd}\]
	for each map $\beta \colon b \rightarrow b'$ in $B$ (together with coherence data for compositions). For oplax transformations the direction of the natural transformation $f_\beta$ is reversed.  
	The $(\infty,2)$-categorical approach to defining such (op)lax natural transformations is to use the Gray tensor product $\Gtimes$, for which a good model has been introduced by Gagna, Harpaz and Lanari in \cite{GHL-Gray}. Indeed, one can define lax and oplax natural transformations between functors of $(\infty,2)$-categories $F,G \colon \tcat{X} \to \tcat{Y}$ in general as functors $[1] \Gtimes \tcat{X} \to \tcat{Y}$ and
	$\tcat{X}\Gtimes [1] \to \tcat{Y}$ restricting to $F$ and $G$ under the two embedding $[0] \rightarrow [1]$, respectively. To see why this is reasonable we observe that the Gray tensor product $[1]\Gtimes [1]$ of the 1-simplex with itself is exactly the lax square
	\[
	\begin{tikzcd}
		\bullet \arrow{r}\arrow{d}& \bullet \arrow{d}\\
		\bullet \arrow{r} \arrow[Leftarrow, shorten=1.5ex]{ur} & \bullet.
	\end{tikzcd}
	\]
	Therefore a functor $f\colon  [1]\Gtimes \tcat{X}\rightarrow \tcat{Y}$ in particular encodes the data of a lax square 
	\[
	\begin{tikzcd}
	F(b)\arrow[d, "\beta_!"{swap}]\arrow[r, "f_{b}"] & F(b)\arrow[d, "\beta_!"]\arrow[ld, Rightarrow, shorten=1.5ex,, "f_\beta"]\\
	F(b')\arrow[r, "f_{b'}"{swap}] & F(b')
	\end{tikzcd}
	\] 
	for every 1-morphism $\beta$ in $\tcat{X}$. The Gray tensor product $[1]\Gtimes \tcat{X}$ further encodes all of the higher coherences that these lax squares should satisfy.
	
	Our second main result relates these definitions to the fibrational approach discussed above. To explain the statement, let us introduce two $(\infty,2)$-categories $\tFun^\lax(\tcat{X},\tcat{Y})$ and $\tFun^\oplax(\tcat{X},\tcat{Y})$ universally defined via
	\[\Hom_{\Cat_2}(-,\tFun^\lax(\tcat{X},\tcat{Y})) \simeq \Hom_{\Cat_2}(-\Gtimes \tcat{X}, \tcat{Y})\] and \[\Hom_{\Cat_2}(-,\tFun^\oplax(\tcat{X},\tcat{Y})) \simeq \Hom_{\Cat_2}(\tcat{X} \Gtimes -, \tcat{Y})\]
	as functors $\Cat_2 \rightarrow \Gpd$. By definition the 1-morphisms in the $(\infty,2)$-category $\tFun^\lax(\tcat{X},\tcat{Y})$ correspond to lax natural transformations and analogously for oplax transformations. 
By means of Lurie's locally cocartesian unstraightening equivalence \cite{LurieGoo} we show:
	\begin{introthm}\label{laxstr}
	There are natural equivalences of $\infty$-categories 
	\[\Gray(A,B) \simeq \Fun(A \Gtimes B,\tcat{Cat}),\]
	and consequently natural equivalences of
	$(\infty,2)$-categories 
	\[\tcat{Cocart}^{\lax}(B) \simeq \tFun^\lax(B,\tCat), \qquad
	\tcat{Cart}^{\oplax}(B) \simeq \tFun^{\oplax}(B^{\op},\tCat),\] given on
	objects by straightening of (co)cartesian fibrations; here the targets
	are defined as above, and so have functors as objects, (op)lax natural
	transformations as morphisms, and modifications between these as
	$2$-morphisms.
	\end{introthm}
	In particular, this implies that the cocartesian unstraightening of a functor $B\rt \Cat$ has the universal property of the \emph{lax colimit}: it is given by the left adjoint of the constant diagram functor $\tCat\rt \tFun^\lax(B,\tCat)$ (see \cref{obs:lax colim}). Furthermore, using \cref{laxstr} we obtain the following reformulation of
	\cref{thm:fibadj}:
	\begin{introcor}\label{core}
		Extracting adjoints gives an equivalence of $(\infty,2)$-categories
		\[\tFun^{\lax,\mm R}(B,\tCat) \simeq \Big(\tFun^{\oplax,\mm L}(B,\tCat)\Big)^{(1, 2)\mm{-}\op}\]
		for every $\infty$-category $B$, where the superscript $\mm R$ denotes
		the locally full (or 1-full) sub-$2$-category of $\tFun^{\lax}(B,\tCat)$ spanned by
		those lax natural transformations that admit pointwise left adjoints,
		and dually for the right hand side. Furthermore, these equivalences are natural for restriction in $B$.
	\end{introcor}

	We will show that
        the equivalence of \itcats{} in \cref{core} is given by
        a higher-categorical form of the
	calculus of mates, in the following sense:
        on 1-morphisms it takes a lax natural transformation $\theta^{R} \colon [1] \Gtimes B \to \tCat$ such that
	$\theta^{R}_{b}$ is a right adjoint for every $b \in B$ to
	an oplax transformation $\theta^{L} \colon B \Gtimes [1] \to \tCat$
	such that the lax naturality squares for $\theta^{R}$ are the mates
	of the oplax naturality squares for $\theta^{L}$. This procedure of taking mates should make sense in any
	\itcat{}, which suggests that a version of \cref{core} should hold for any two $(\infty,2)$-categories in lieu of $B$ and $\tCat$. This generality does not seem to be directly within reach of our
	methods.
	
	Furthermore, we provide two elaborations on the equivalences
	of \cref{thm:fibadj} and \cref{core}. First, we describe the (co)unit
	of a fibrewise adjunction fibrationally, and consequently also the passage to adjoint morphisms in families. Secondly, we provide a
	characterisation of fibrewise adjoints in terms of mapping
	functors, to identify these in practice.

	Let us finally point out that the book of Gaitsgory and Rozenblyum
	contains \cref{laxstr} as a consequence of their general
	$2$-categorical straightening procedure \cite[Corollary
	11.1.2.6]{GaitsgoryR}. It also outlines a version of \cref{core} with
	$B$ and $\tCat$ replaced by arbitrary $(\infty,2)$-categories
	\cite[Section 12.3]{GaitsgoryR}, using constructions somewhat similar
	to the ones presented here; in particular, our notions of curved
	orthofibrations and op-Gray fibrations are also introduced in Sections
	12.2.1 and 12.2.3 of \cite{GaitsgoryR}. However, as far as we can tell
	they deduce both results using arguments different from ours, which at
	some points seem to rely on some of the unproven statements from
	\cite[Chapter 10]{GaitsgoryR}, notably about models for
	$(\infty, 2)$-categories in terms of bisimplicial spaces of lax
	squares.

	\begin{uremark}
		This article is one part of a recombination of our earlier preprints \cite{Haugseng} and
		\cite{HLN}, which contain many of the results we present here, most of
		them twice; the other part is \cite{part2}. For the reader interested in archaeology we mention that
		Theorems 1.1, 1.2 and 1.3 from \cite{Haugseng} are now contained in
		Corollary \ref{core}, Proposition \ref{propa} and Proposition
		\ref{prop:monoidalinternalhom}, whereas Theorems A, B and C from \cite{HLN} are now part of Proposition \ref{propa}, Theorem
		\ref{thm:grayorthdual} and Corollary \ref{cord},
		respectively. 
	\end{uremark}

\subsection*{Organisation}
Section \ref{sec:ortho} introduces curved
orthofibrations and Gray fibrations in more detail and establishes
their basic properties. In particular, we deduce
\cref{thm:grayorthdual} as \cref{thm:dual lortho-gray}. In Section
\ref{sec:adjunctions} we then introduce and study parametrised adjunctions in fibrational form. We prove \cref{thm:fibadj}  as \cref{thm:dualizing lax
  natural transformations}, and deduce \cref{propa} and \cref{cord} as  \cref{propaintext} and 
\cref{thm:monoidal adjunctions}, respectively. In between, this section also contains the identification of the functor in \cref{thm:fibadj} on morphisms with the Beck-Chevalley construction and the characterisation of parametrised adjoints in terms of mapping \igpds{}. Section~\ref{sec:paraunit} then discusses
units and counits for parametrised adjunctions and derives the
functoriality of the passage to adjoint morphisms in the
parametrised  context. Finally, in Section \ref{sec:laxgray} we establish
the connection to lax natural transformations, prove \cref{laxstr} as a combination of \cref{cor:gray fib is loc cocart of gray} and \cref{thm:cocart vs functor lax transformations}, and lastly deduce \cref{core} as \cref{thm:mates}.

\subsection*{Conventions}
As mentioned above, in order to declutter notation we will write $\Gpd, \Cat$ and $\Cat_2$ for the \icats{} of $\infty$-groupoids (or spaces), $\infty$-categories and $(\infty, 2)$-categories, respectively. By default we use complete two-fold Segal spaces as the definition of the latter, but we will also need to discuss other models in \cref{sec:laxgray}.

The letter $\iota$ will denote the core of an
$\infty$-category, \ie{} the $\infty$-groupoid spanned by its
equivalences, and $|\cdot|$ its realisation. By a subcategory of an $\infty$-category we
mean a functor such that the induced morphisms on mapping \igpds{} and
cores are inclusions of path components. A subcategory is \emph{full} if the functor furthermore induces
equivalences on mapping \igpds{}, while it is \emph{wide} if the functor induces an
equivalence on cores. Similarly, a sub-$2$-category of an
$(\infty, 2)$-category is a functor inducing subcategory inclusions
on mapping \icats{} and an inclusion of path components on underlying
\igpds{}; we say such a sub-$2$-category is \emph{1-full} if it is locally full, \ie{} given by full subcategory inclusions on mapping \icats{}.

Throughout, we shall use small caps such as $\CAT$ to indicate the large
variants of $\infty$-categories and boldface such as
$\tCat$ to indicate the $(\infty,2)$-categorical variants. We have
also reserved sub- and superscripts on category names to refer to
changes on morphisms, \eg{} $\cart(A) \subseteq \Cartlax(A)$.

We will write $\Ar(C)$ for the arrow \icat{} $\Fun([1],C)$ of an \icat{} $C$, and $\TwL(C)$ and $\TwR(C)$ for the two versions of the twisted arrow category, geared so that the combined source-target map defines a left fibraton in the former, and a right fibration in the latter case, see \ref{not:Tw}.

\subsection*{Acknowledgments}
It is a pleasure to thank Dustin Clausen, Yonatan Harpaz, Gijsbert Scheltus Karel Sebastiaan Heuts, Corina Keller, Achim Krause, Markus Land, Denis Nardin, Thomas Nikolaus, Stefan Schwede, Wolfgang Steimle and Lior Yanovski for several very useful discussions.

During the preparation of this manuscript FH and SL were members of the Hausdorff Center for Mathematics at the University of Bonn funded by the German Research Foundation (DFG), grant no. EXC 2047. FH and JN were further supported by the European Research Council (ERC) through the grants `Moduli spaces, Manifolds and Arithmetic', grant no. 682922, and `Derived Symplectic Geometry and Applications', grant no. 768679, respectively.

\section{Two-variable fibrations}\label{sec:ortho}

Our main goal in the present section is to introduce two new classes of
fibrations over a product of two \icats{}, namely \lorthos{} and Gray
fibrations, and describe how they can be unstraightened over one of
the two factors, and consequently dualised. We first recall some basic material about
(co)cartesian fibrations in \S\ref{subsec:fibbackgr}. Then in \S\ref{subsec:Rfib} we discuss functors to a product of two \icats{} that behave like a (co)cartesian fibration in one of the two variables; both \lorthos{} and Gray fibrations are special cases of such functors. In \S\ref{subsec:lorth} we then introduce (curved) orthofibrations and study their partial unstraightenings. 
We consider Gray fibrations in \S\ref{subsec:Gray} and characterise cocartesian and left fibrations among these. Finally, in \S\ref{subsec:duallll} we record the various ways in which these fibrations can be dualised. In particular, we prove \cref{thm:grayorthdual} there.

\subsection{Background}\label{subsec:fibbackgr}
For the reader's convenience, we begin by briefly reviewing some basic
material on (co)cartesian morphisms and fibrations.

\begin{definition}
  Let $p \colon X\rightarrow S$ be a functor of
  $\infty$-categories. Then a morphism $\alpha \colon y\rightarrow z$ of $X$ is \emph{$p$-cartesian} if the square
  \[
    \begin{tikzcd}
      \Map_X(x,y) \arrow[r, "\alpha_*"] \arrow[d] & \Map_X(x,z) \arrow[d] \\ 
      \Map_S(p(x),p(y)) \arrow[r, "p(\alpha)_*"] & \Map_S(p(x),p(z))
    \end{tikzcd}
  \]
  is a pullback square in $\Gpd$ for every $x \in X$. Dually, the
  morphism $\alpha$ is \emph{$p$-cocartesian} if it is $p^{\op}$-cartesian
  when regarded as a morphism in $X^{\op}$, or in other words if for every $x \in X$ the square
  \[
    \begin{tikzcd}
      \Map_X(z,x) \arrow[r, "\alpha^*"] \arrow[d] & \Map_X(y,x) \arrow[d] \\ 
      \Map_S(p(z),p(x)) \arrow[r, "p(\alpha)^*"] & \Map_S(p(y),p(x))
    \end{tikzcd}
  \]
  is cartesian.
\end{definition}

\begin{notation}\label{not:cartesian arrows}
  To make diagrams more readable, for $p \colon X \to B$ we will
  sometimes indicate a $p$-cocartesian morphism of $X$ by $x\rightarrowtail y$ and a $p$-cartesian morphism of $X$ by $x\twoheadrightarrow y$.
\end{notation}

\begin{definition}
  Let $p \colon X \to S$ be a functor of \icats{}. If $T$ is a
  subcategory of $S$,
  we say that \emph{$X$ has all $p$-cartesian
      lifts over $T$} if for every morphism $f \colon a \to b$ in $T$
    and every object $x$ such that $p(x) \simeq b$, there exists a
    filler in the commutative square
  \[
    \begin{tikzcd}
      {[0]} \arrow{r}{x} \arrow{d}{d_{0}} & X \arrow{d}{p} \\
      {[1]} \arrow{r}{f} \arrow[dashed]{ur} & S
    \end{tikzcd}
  \]
  which is a $p$-cartesian morphism. Dually, \emph{$X$ has all $p$-cocartesian
	lifts over $T$} if $X^{\op}$ has all $p^{\op}$-cartesian lifts
      over $T^{\op}$.
The functor $p \colon X \to S$ is a \emph{cartesian fibration} if
$X$ has all $p$-cartesian lifts over $S$, and a \emph{cocartesian
  fibration} if $X$ has all $p$-cocartesian lifts over $S$.
\end{definition}

\begin{notation}
  We will write $\coCartlax(S)$ and $\Cartlax(S)$ for
  the full subcategories of $\Cat/S$ spanned by the cocartesian and
  cartesian fibrations, respectively, and $\cocart(S)$ and $\cart(S)$
  for the wide subcategories thereof in which morphisms are required to
  preserve (co)cartesian edges. 
\end{notation}

\begin{remark}\label{internalexternal}
The definition above is an invariant version of the
    definition for quasicategories given by Lurie in \cite[Definition
    2.4.2.1]{HTT}. More precisely, a map $p$ between quasicategories corresponds to a (co)cartesian fibration in our sense if and only if for some (and then any) factorisation of $p$ into a categorical equivalence followed by a categorical fibration the latter is a (co)cartesian fibration in Lurie's sense.
\end{remark}

\begin{definition}
  Let $p \colon X \to S$ be a functor of \icats{}. A morphism $\alpha
  \colon y \to z$ in $X$ is \emph{locally $p$-(co)cartesian} if it is
  a (co)cartesian morphism for the pullback $X \times_{S} [1] \to [1]$
  of $p$ along $p(\alpha) \colon [1] \to S$. The functor $p$ is a
  \emph{locally (co)cartesian fibration} if the pullback $X \times_{S}
  [1] \to [1]$ is a (co)cartesian fibration for every map $[1]
  \to S$.
\end{definition}

\begin{notation}
  We write $\LocCocart^{\lax}(S)$ and $\LocCart^{\oplax}(S)$ for the
  full subcategories of $\Cat/S$ spanned by the locally cocartesian
  and locally cartesian fibrations, respectively. We also denote by
  $\LocCocart(S)$ and $\LocCart(S)$ the wide subcategories of these
  where morphisms are required to preserve locally (co)cartesian
  morphisms.
\end{notation}

\begin{defn}
	We call a functor $p \colon X \to S$ that is both a cartesian and a
	cocartesian fibration a \emph{bicartesian fibration}. We write
	$\mathrm{Bicart}^{\mathrm{(op)lax}}(S)$ for the full subcategory of
	$\Cat/S$ spanned by the bicartesian fibrations.
\end{defn}

\begin{remark}
	In the category theory literature our ``bicartesian fibrations'' are
	often called ``bifibrations''; we will instead use the latter term
	as in \cite{HTT}, see \cref{defn:bifib}.
\end{remark}

We recall the following characterisation from \cite[Lemma 2.4.2.7]{HTT} of cartesian morphisms in a locally cartesian fibration, which will be used repeatedly below.

\begin{propn}\label{lem:cartinloccartfib}
  Suppose $p \colon E \to B$ is a locally cartesian fibration. Then
  the following are equivalent for a locally $p$-cartesian morphism $f
  \colon x\to y$ in $E$:
  \begin{enumerate}[(1)]
  \item $f$ is a $p$-cartesian morphism.
  \item For every locally $p$-cartesian morphism $g \colon z \to x$,
    the composite $fg \colon z \to y$ is also locally $p$-cartesian. \qed
  \end{enumerate} 
\end{propn}

\begin{cor}\label{cor:loccartiscartcond}
  Suppose $p \colon E \to B$ is a locally cartesian fibration. Then
  $p$ is a cartesian fibration \IFF{} any composite of locally
  $p$-cartesian morphisms is locally $p$-cartesian. \qed
\end{cor}

The following is \cite[Proposition 2.4.2.4]{HTT}:

\begin{lemma}
The following conditions on a cartesian fibration $p \colon X \rightarrow S$ are equivalent:
  \begin{enumerate}
  \item the fibres $X_{s}$ are \igpds{} for all $s$ in $S$,
  \item all morphisms in $X$ are $p$-cartesian,
  \item $p$ is conservative.\qed
  \end{enumerate} 
\end{lemma}

\begin{definition}
A \emph{right fibration} is a cartesian fibration satisfying the equivalent conditions of the previous lemma. A \emph{left fibration} is a functor whose opposite is a right fibration.
\end{definition}

Finally, let us briefly discuss functoriality. Consider the functor
$t \colon \Ar(\Cat) \rightarrow \Cat$ extracting the target of a
morphism. Its cartesian edges are precisely the
pullback squares, so since $\Cat$ is complete we obtain a
functor
\[\Cat^\op \longrightarrow \CAT, \quad S \longmapsto \Cat/S\]
by cartesian unstraightening, where $\CAT$ denotes the \icat{} of
\emph{large} \icats{}. By \cite[Proposition 2.4.1.3]{HTT}
the pullback of a (co)cartesian fibration is a (co)cartesian fibration and the structure map in a pullback preserves cocartesian edges. Therefore one obtains subfunctors
\[\leftfib, \rightfib, \cocart, \cart \colon \Cat^\op \longrightarrow \CAT\]
via the construction above. Combining Lurie's unstraightening equivalence with \cite[Appendix A]{GHN} one finds inverse equivalences
\[\begin{tikzcd}\Strcart \colon \Cart  \arrow[r, yshift=0.5ex] & \arrow[l, yshift=-0.5ex]  \Fun(-^\op,\Cat) \cocolon \Uncart\end{tikzcd} \quad \text{and}\quad\begin{tikzcd}\Strco \colon \Cocart  \arrow[r, yshift=0.5ex] & \arrow[l, yshift=-0.5ex]  \Fun(-,\Cat) \cocolon \Unco\end{tikzcd}\]
which restrict to equivalences
\[\begin{tikzcd}\rightfib  \arrow[r, yshift=0.5ex] & \arrow[l, yshift=-0.5ex]  \Fun(-^\op,\Gpd) \end{tikzcd} \quad \text{and}\quad\begin{tikzcd}\leftfib  \arrow[r, yshift=0.5ex] & \arrow[l, yshift=-0.5ex]  \Fun(-,\Gpd).\end{tikzcd}\]
The resulting equivalence between cartesian and cocartesian fibrations we shall denote
\[\begin{tikzcd}\Dualco \colon \Cart(-^\op)  \arrow[r, yshift=0.5ex] & \arrow[l, yshift=-0.5ex]  \Cocart \cocolon \Dualcart.\end{tikzcd}\]
Its restriction to left and right fibrations is simply given by taking opposites, but this is not true in general, since $\Dualcart$ and $\Dualco$ are given by the identity on $\Cart(*) \simeq \Cat \simeq \Cocart(*)$; an explicit description of the equivalence in the general case is the main result of \cite{BGN}.

\subsection{Straightening in one variable}\label{subsec:Rfib}
Before we introduce the main classes of fibrations we are interested
in, here we will consider the most general kinds of functors to a
product of \icats{} that can be straightened over one of the two
factors.  The basic observation we need for this is the following:
\begin{propn}\label{propn:projtoeqcocart}
  Given a functor of \icats{} $p = (p_1,p_2) \colon X \to A \times B$, a
  morphism $\alpha \colon x \to y$ in $X$ such that $p_{1}(\alpha)$ is an
  equivalence is $p$-cocartesian \IFF{} it is $p_{2}$-cocartesian.
\end{propn}
\begin{proof}
  Since equivalences are always cocartesian, this is an immediate consequence of \cite[Proposition
  2.4.1.3(3)]{HTT}, which says that given $X \xto{q} Y \xto{r} Z$,
    a morphism in $X$ whose image in $Y$ is $r$-cocartesian is
    $q$-cocartesian \IFF{} it is $rq$-cocartesian.
\end{proof}

\begin{corollary}\label{onevarcocart}
  The following are equivalent for a functor $p = (p_1,p_2) \colon X \to
  A \times B$:
  \begin{enumerate}[(1)]
  \item $X$ has all $p$-cocartesian lifts over $A \times \core{B}$.
  \item $p_1$ is a cocartesian fibration and all $p_1$-cocartesian morphisms
    lie over equivalences in $B$.
  \item In the commutative triangle
    \[
      \begin{tikzcd}
        X \arrow{rr}{p}\arrow[dr, "p_1"{swap}] & & A \times B
        \arrow{dl}{\pr_{1}} \\
        & A
      \end{tikzcd}
    \]
    the map $p_1$ is a cocartesian fibration, and $p$ takes
    $p_1$-cocartesian morphisms to $\pr_{1}$-cocartesian morphisms.
  \end{enumerate}
\end{corollary}
\begin{proof}
  The equivalence of (1) and (2) is immediate from
  \cref{propn:projtoeqcocart}, while that of (2) and (3) amounts to
  the observation that the cocartesian morphisms for $\pr_{1} \colon A
  \times B \to A$ are precisely those morphisms that project to equivalences in $B$.
\end{proof}

\begin{defn}
  We say that a functor $p \colon X \to A \times B$ is
  \emph{cocartesian over the left factor}, or simply \emph{cocartesian over $A$} when no confusion can arise, if it satisfies
  the equivalent conditions of \cref{onevarcocart}. Dually, we say
  that $p$ is \emph{cartesian over $A$} if $p^{\op}$ is cocartesian
  over $A^{\op}$. We write $\LCocart(A,B)$ and $\LCart(A,B)$ for the
  subcategories of $\Cat/(A \times B)$ whose objects are (co)cartesian
  over $A$, with the morphisms required to preserve the (co)cartesian
  morphisms over $A \times \core{B}$. Similarly, we write
  $\RCocart(A,B)$ and $\RCart(A,B)$ for the subcategories of $\Cat/(A \times B)$ whose objects are (co)cartesian
  over the right factor $B$, with the morphisms required to preserve the (co)cartesian
  morphisms over $\core{A} \times B$. 
\end{defn}

Of course, we obtain equivalences
  \[ \RCocart(A,B) \simeq \LCocart(B,A), \qquad \RCart(A,B) \simeq
    \LCart(B,A)\]
by restricting the obvious equivalence $\Cat/(A \times B)
  \simeq \Cat/(B \times A)$. 
  
  From the third condition in \cref{onevarcocart} we immediately see:
\begin{corollary}\label{rmk:lcocstr}
	We write $\pr_1:A\times B\rightarrow A$ for the projection to $A$. The equivalence $\Cat/(A \times B) \simeq (\Cat/A)/\pr_1$ restricts to equivalences of subcategories
  \begin{equation}
    \label{eq:lcocslice}
   \LCocart(A,B) \simeq \Cocart(A)/\pr_1, \qquad \LCart(A,B)
   \simeq \Cart(A)/\pr_1. 
  \end{equation}
Combining these equivalences with straightening over $A$, we get
  natural equivalences
  \begin{equation}
    \label{eq:lcocstr}
    \LCocart(A,B) \simeq \Fun(A, \Cat/B), \qquad \LCart(A,B) \simeq
    \Fun(A^{\op}, \Cat/B),
  \end{equation}
  since $\pr_1:A \times B \rightarrow A$ straightens over $A$ to the constant functor with
  value $B$. \qed
\end{corollary}

For later use we also note the following consequence of
\cref{propn:projtoeqcocart} here:
\begin{corollary}\label{cor:cocartprgpd}
  Let $I$ be an \igpd{} and $C$ an \icat{}. The following are
  equivalent for a functor $p \colon X \to I \times C$:
  \begin{enumerate}[(1)]
  \item $p$ is a cocartesian fibration.
  \item For every $i \in I$, the morphism on fibres $p_{i} \colon
    X_{i} \to C$ is a cocartesian fibration.
  \item The composite $X \xto{p} I \times C \to C$ is a cocartesian
    fibration.
  \end{enumerate}
\end{corollary}
\begin{proof}
  The equivalence of (1) and (3) follows from
  \cref{propn:projtoeqcocart}, while (1) implies (2) since cocartesian
  fibrations are closed under base change. Finally, applying the criterion of \cite[Lemma A.1.8]{cois} to the commutative
  triangle
  \[
    \begin{tikzcd}
      X \arrow{rr}{p} \arrow{dr} & & I \times C \arrow[dl,"\pr_1"] \\
       & I,
    \end{tikzcd}
  \]
 shows that (2) implies (1).
\end{proof}

\begin{notation}
  Given a functor $p \colon X \to A \times B$, we define $p_{\ell}$ 
  and $p_{r}$ by the cartesian squares 
  \[
    \begin{tikzcd}
      X_{\ell} \arrow[d, "p_{\ell}"{swap}] \arrow{r} & X \arrow{d}{p} & X_{r}
      \arrow{l} \arrow{d}{p_{r}} \\
      A \times \core{B} \arrow{r} & A \times B & \core{A} \times B. \arrow{l}
    \end{tikzcd}
  \]
  To make diagrams more readable, we will sometimes indicate a $p_{\ell}$-cartesian edge of $X$ by $x\localarrowhead y$ and a $p_r$-cocartesian edge of $X$ by $x\localarrowtail y$.      
\end{notation}

\begin{remark}\label{rmk:pbtocore}
  From \cref{cor:cocartprgpd} it follows immediately that for a
  functor $p \colon X \to A \times B$ the pullback $p_{\ell}$ is a
  (co)cartesian fibration \IFF{} for every $b \in B$ the map on fibres
  $p_{b} \colon X_{b} \to A$ is a (co)cartesian fibration, and
  similarly for $p_{r}$.
\end{remark}

\subsection{\Lorthos{}}\label{subsec:lorth}
If we combine our conditions from the previous subsection for a
functor to $A \times B$ to straighten contravariantly over $A$ and
covariantly over $B$, we obtain the following definition:
\begin{definition}
  A \emph{\lortho{}} is a functor of \icats{} $p \colon X \to A \times
  B$ such that $p$ is cartesian over $A$ and cocartesian over $B$,
  \ie{} $X$ has all $p$-cartesian lifts over $A \times \core{B}$ and all
  $p$-cocartesian lifts over $\core{A} \times B$.  We write $\LOrth(A,B)$ for the subcategory of $\Cat/(A\times B)$
  whose objects are the \lorthos{}, with the morphisms
  required to preserve both cartesian morphisms over $A$ and
  cocartesian morphisms over $B$.
\end{definition}
Let us record two alternative characterisations:

\begin{observation}
  Using \cref{onevarcocart} we can reformulate the definition of a
  \lortho{} as a functor $p=(p_1,p_2) \colon X \to A \times B$ such that
  \begin{enumerate}[(1)]
  \item $p_1$ is a cartesian fibration,
  \item $p_2$ is a cocartesian fibration,
  \item every $p_1$-cartesian morphism in $X$ lies over an equivalence
    in $B$,
  \item every $p_2$-cocartesian morphism in $X$ lies over an equivalence
    in $A$.
  \end{enumerate}
\end{observation}

\begin{propn}\label{prop:lortho reformulations}
  The following are equivalent for a functor $p=(p_1,p_2) \colon X \to A \times B$:
  \begin{enumerate}[(1)]
  \item\label{condlortho} $p$ is a \lortho{}.
  \item\label{lorthocartfw} In the commutative triangle
    \[
      \begin{tikzcd}
        X \arrow{rr}{p} \arrow[dr, "p_1"{swap}] & & A \times B \arrow{dl}{\pr_{1}} \\
         & A,
      \end{tikzcd}
    \]
    $p_1$ is a cartesian fibration, $p$ takes $p_1$-cartesian
    morphisms to $\pr_{1}$-cartesian morphisms, and for every $a \in
    A$ the map on fibres $X_{a} \to B$ is a cocartesian fibration.
  \item\label{lorthocartr} $p$ is cartesian over $A$ and
    $p_{r} \colon X_{r} \to \core{A} \times B$ is a cocartesian fibration.
  \item\label{lorthococfw} In the commutative triangle
    \[
      \begin{tikzcd}
        X \arrow{rr}{p} \arrow[dr, "p_2"{swap}] & & A \times B \arrow{dl}{\pr_{2}} \\
         & B,
      \end{tikzcd}
    \]
    $p_2$ is a cocartesian fibration, $p$ takes $p_2$-cocartesian
    morphisms to $\pr_{2}$-cocartesian morphisms, and for every $b \in
    B$ the map on fibres $X_{b} \to A$ is a cartesian fibration.
  \item\label{lorthococl} $p$ is cocartesian over $B$ and $p_{\ell} \colon X_{\ell} \to
    A \times \core{B}$ is a cartesian fibration.
  \end{enumerate}
\end{propn}
\begin{proof}
  The equivalence of \ref{lorthocartfw} and \ref{lorthocartr}, as well
  as of \ref{lorthococfw} and \ref{lorthococl}, follows from
  \cref{rmk:pbtocore} and \cref{onevarcocart}. It thus remains to show that
  \ref{condlortho} is equivalent to one of these pairs, since they
  correspond to each other under taking opposites.

  If $p$ is a \lortho{} then it is immediate from the definition that
  $p_{\ell}$ is a cartesian and $p_{r}$ a cocartesian fibration, \ie{}
  \ref{condlortho} implies \ref{lorthocartr} and
  \ref{lorthococl}. Conversely, the implication \ref{lorthocartfw}
  $\Rightarrow$ \ref{condlortho} is
  \cite[Lemma A.1.10]{cois}. For completeness, we also include a
  brief argument that \ref{lorthococl} implies \ref{condlortho}:
  We need to show that a $p_{\ell}$-cartesian lift $\lambda \colon x \rightarrow y$ in $X$ of an arrow $(\alpha, \beta)\colon (a,b) \rightarrow (a',b')$, for which $\beta\colon b \rightarrow b'$ is an equivalence, is automatically $p$-cartesian.

Consider thus the black part of the diagram
	
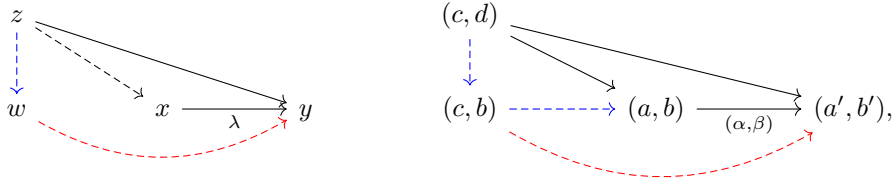
\[\begin{tikzcd}[column sep=4em]
	z\arrow[rrd]\arrow[d, dashed, blue]\arrow[rd, dashed] & & & (c, d)\arrow[rd]\arrow[rrd]\arrow[d, dashed, blue]\\
	w\arrow[rr, dashed, red, bend right] & x\arrow[r, "\lambda" below] & y & (c,b)\arrow[r, dashed, blue]\arrow[rr, dashed, red, bend right]& (a, b)\arrow[r, "{(\alpha, \beta)}" below] & (a', b'),
      \end{tikzcd}\] which is a lifting problem in which one has to
    find a black dashed arrow in an essentially unique manner. First
    take an (essentially unique) $p$-cocartesian lift $z\rt w$ of
    $(c,d) \rightarrow (c,b)$. Since this arrow is cocartesian in all
    of $X$, there is an essentially unique dotted red arrow lifting
    the outer triangle on the right. Since the lower horizontal part
    of the diagram lives over $A \times \core{B}$ there now exists an
    essentially unique map $w\rt x$ (not drawn) lifting the lower
    triangle. The composition with $z\rt x$ is the desired black
    dotted map, and using that $z\rt w$ is $p$-cocartesian one can
    then complete the diagram in an essentially unique way. The
    essential uniqueness of the map $z\rt x$ is seen by reading the
    argument in reverse.
\end{proof}

By definition a functor to $A \times B$ is a \lortho{} \IFF{} it lies in both $\cat{LCart}(A,B)$ and $\cat{RCocart}(A,B)$ and thus by
\cref{rmk:lcocstr} a \lortho{} can be straightened in either of the
two variables. The previous lemma allows us to see precisely \emph{what} a curved orthofibration straightens to:

\begin{corollary}\label{cor:locorthstr}
  Straightening over $A$ and $B$ give natural equivalences
  \[ \Fun(A^{\op}, \Cocartlax(B))^{\mathrm{cc}} \simeq \LOrth(A,B) \simeq
    \Fun(B, \Cartlax(A))^{\mathrm{ct}},\]
  respectively,
  where $\Fun(B, \Cartlax(A))^{\mathrm{ct}}$ denotes the wide subcategory
  of $\Fun(B, \Cartlax(A))$ in which the morphisms are natural
  transformations whose components all preserve cartesian morphisms
  over $A$, and similarly for $\Fun(A^{\op},
  \Cocartlax(B))^{\mathrm{cc}}$.
\end{corollary}

\begin{proof}
\cref{lorthocartfw} and \cref{lorthococfw} of \cref{prop:lortho reformulations} immediately imply the result at the level of objects. We will exhibit the claim on morphisms for the left equivalence, the right being dual. Suppose that \[\begin{tikzcd}
X\arrow[rd, "p"']\arrow[rr, "f"] & 	&	Y\arrow[ld, "q"] \\
&	A\times B 	& 
\end{tikzcd}\] is the unstraightening of a map $\eta\in \Fun(A^\op,\Cocartlax(B))$. Because it is given by cartesian unstraightening, it will necessarily preserve cartesian edges over $A$. Naturality of unstraightening implies that the components of $\eta$ preserving cocartesian edges is equivalent to the maps $f_{a}$ on fibres preserving cocartesian edges. \Cref{propn:projtoeqcocart} then implies that $f_{r}$ is a map of cocartesian fibrations. However, we have seen in \cref{prop:lortho reformulations} that for a \lortho, the $p_r-$cocartesian edges agree with the $p$-cocartesian edges over $B$. Therefore, $\eta$ preserves cocartesian edges pointwise if and only if $f$ preserves cocartesian edges over $B$, and is thus a map of curved orthofibrations.
\end{proof}

We saw above that \lorthos{} over $A \times B$ can be straightened to
functors $A^{\op} \to \Cocartlax(B)$ or $B \to \Cartlax(A)$. Our next
goal is to introduce a further condition that will ensure these
functors actually land in the subcategories $\Cocart(B)$ and
$\Cart(A)$ (and thus encode functors
$A^\op \times B \rightarrow \Cat$), giving the notion of
\emph{orthofibrations}. We also specialise further to
\emph{bifibrations} as considered by Lurie in
\cite[Section 2.4.7]{HTT}, and studied in detail for example in
Stevenson \cite{Stevenson} or \cite[Appendix A]{cois}; these
straighten to functors $A^{\op}\to \LFib(B)$ and $B \to \RFib(A)$.

\begin{construction}\label{ex:interpol}
Suppose that $A=[1]^{\op}$ and $B=[1]$, and let us write $\alpha, \beta$ for the unique nondegenerate simplices in $A$ and $B$. Consider the diagram $\rho\colon [1]\rt \cocart^{\lax}([1])$ corresponding to the map of cocartesian fibrations (between ordinary categories)
\begin{equation}\label{eq:corep interpol}\begin{tikzcd}
{[1]}\arrow[rd, equal]\arrow[rr, "\partial_1"] & & {[2]}\arrow[ld, "\sigma_1"]\\
& {[1].}
\end{tikzcd}\end{equation}
Note that this diagram is characterised by a universal property in $\Ar(\Cocartlax([1]))^{\mathrm{cc}},$ which is our slightly shorter notation for $\Fun([1],\cocart^{\lax}([1]))^{\mathrm{cc}}$: for each object $g\colon X\rt Y$ in $\Ar(\Cocartlax([1]))^{\mathrm{cc}}$, evaluation at $\{0\}\in [1]=\rho(0)$ yields a natural equivalence to the fibre of $X$ over $0$
$$
\Map_{\Ar(\Cocartlax([1]))^{\mathrm{cc}}}(\rho, g)\simeq \core X_0.
$$
Indeed, unravelling the definitions shows that a natural transformation
$\rho\Rightarrow g$ whose components preserve cocartesian arrows is
given by a cocartesian arrow $\tilde{\beta}\colon x\rt \beta_!x$ in
$X$ over $\beta$, together with a factorisation of $g(\tilde{\beta})$ into a cocartesian morphism followed by a fibrewise one,
\begin{equation}\label{eq:interpol str}\begin{tikzcd}
g(\tilde{\beta})\colon g(x)\arrow[r, tail] & \beta_!g(x)\arrow[r, "\rho_\beta(x)"] & g(\beta_!x).
\end{tikzcd}\end{equation}
Now the cartesian unstraightening of $\rho$ over $A=[1]$ is the \lortho{}
$$
q\colon Q\rt [1]^{\op}\times [1]
$$
where $Q$ is the poset given by
\begin{center}
\begin{tikzcd}
                                        & 11'\arrow[rd]\ar[dd,dotted] &                          \\
10 \arrow[d,two heads] \arrow[ru,tail] \ar[rd, dotted] \ar[rr, dotted]&    & 11 \arrow[ld,two heads] \\
00 \arrow[r,tail]                       & 01  &                         
\end{tikzcd}  
\end{center} 
and the projection is the evident one, sending $11'\rt 11$ to the identity. Then $Q\rt [1]^{\op}\times [1]$ has the following universal property: for every \lortho{} $p\colon X\rt A\times B$ and every $\alpha\colon a'\rt a$ and $\beta\colon b\rt b'$, there is a natural equivalence between the \igpd{} $\core(X_{(a, b)})$ of objects in the fibre over $(a, b)$ and the \igpd{} of maps of curved orthofibrations
$$\begin{tikzcd}
Q\arrow[r]\arrow[d] & X\arrow[d, "p"]\\
{[1]^{\op}\times[1]}\arrow[r, "\alpha\times\beta"] & A\times B.
\end{tikzcd}$$
Let us refer to such diagrams as \emph{$p$-interpolating diagrams}. Explicitly, the $p$-interpolating diagram associated to $x\in X_{(a, b)}$ is given by
\[
\begin{tikzcd}
                                        &
                                        (\id,\beta)_{!}(\alpha,\id)^{*}x\arrow[rd,
                                        dashed] \arrow[dd,dotted] &                          \\
(\alpha,\id)^{*}x \arrow[d, two heads] \ar[rd,dotted] \arrow[ru, tail] \arrow[rr,dotted] &    & (\alpha,\id)^{*}(\id,\beta)_{!}x \arrow[ld, two heads] \\
x \arrow[r, tail]                       & (\id,\beta)_{!}x  &                         
\end{tikzcd}  
\]
where we choose $p$-(co)cartesian morphism and (dotted) compositions as indicated, and finally the dashed arrow is given by either factoring the horizontal dotted morphisms through the cocartesian morphism over $(\id,\beta)$, or equivalently by factoring the vertical dotted morphism through the cartesian morphism over $(\alpha,\id)$.
\end{construction}

\begin{definition}\label{def:interpol lorth}
Let $p\colon X\rt A\times B$ be a \lortho{}. We will refer to a morphism in $X$ as \emph{$p$-interpolating} if it arises as the evaluation at $11' \rightarrow 11$ of a $p$-interpolating diagram $Q \rightarrow X$.
\end{definition}
\begin{remark}
The $p$-interpolating edges in $X$ are precisely the edges that arise under unstraightening from the morphisms $\rho_\beta(x)$ described in \eqref{eq:interpol str}.
\end{remark}

\begin{defn}
  A functor $p=(p_1,p_2) \colon X \to A \times B$ is an
  \emph{orthofibration} if it is a \lortho{} and all $p$-interpolating
  morphisms in $X$ are invertible, \ie{} for every
  pair of morphisms $\alpha \colon a' \to a$ in $A$ and $\beta \colon b \to b'$
  in $B$ and every object $x$ in $X$ over $(a,b)$, the interpolating
  morphism
  \[ (\id,\beta)_{!}(\alpha,\id)^{*}x \longrightarrow (\alpha,\id)^{*}(\id,\beta)_{!}x \]
  is an equivalence. We write $\Orth(A,B)$ for the full subcategory of
  $\LOrth(A,B)$ spanned by the orthofibrations. (Note that our orthofibrations are the same as the \emph{two-sided fibrations} defined in \cite[Section 7.1]{RiehlVerity}.)
\end{defn}

\begin{propn}
  The following are equivalent for a curved orthofibration $p=(p_1,p_2) \colon X \to A \times B$:
  \begin{enumerate}[(1)]
  \item\label{orth} $p$ is an orthofibration.
  \item\label{orth ct} For every morphism $\alpha \colon a' \to a$ in $A$ the cartesian
    transport functor $\alpha^{*} \colon X_{a} \to X_{a'}$ preserves
    $p_2$-cocartesian morphisms. 
  \item\label{orth cot} For every morphism $\beta \colon b' \to b$ in $B$ the cocartesian
    transport functor $\beta_{!} \colon X_{b'} \to X_{b}$ preserves
    $p_1$-cartesian morphisms. 
  \end{enumerate}
\end{propn}
\begin{proof}
Let us consider a \lortho{} $p$,
  morphisms $\alpha\colon a'\rt a$ and $\beta\colon b\rt b'$, and
  $x\in X_{(a, b)}$. By \cref{ex:interpol}, the associated interpolating morphism fits into commuting triangles
\[\begin{tikzcd}
(\alpha, \id)^*x\arrow[r, tail]\arrow[rd, "\alpha^*(\tilde{\beta})"{swap}] & (\id, \beta)_!(\alpha, \id)^*x\arrow[d] & & (\id, \beta)_!(\alpha, \id)^*x\arrow[r]\arrow[rd, "\beta_!(\tilde{\alpha})"{swap}] & (\alpha, \id)^*(\id, \beta)_!x\arrow[d, two heads]\\
& (\alpha, \id)^*(\id, \beta)_!x & & & (\id, \beta)_!x.
\end{tikzcd}\]
  Here the diagonal morphisms are the image of the cocartesian
  morphism $\tilde{\beta}\colon x\rt (\id, \beta)_!x$ under
  $\alpha^*\colon X_{a}\rt X_{a'}$ and the image of the cartesian
  morphism $(\alpha, \id)^*\colon (\id, \alpha)^*x\rt x$ under
  $\beta_!\colon X_b\rt X_{b'}$. It follows that the interpolating
  morphism is an equivalence if and only if these images remain
  cocartesian and cartesian, respectively. This shows that condition
  \ref{orth} is equivalent to both \ref{orth ct} and \ref{orth cot}.
 \end{proof}

From this we see that restricting the equivalence of
\cref{cor:locorthstr} to
orthofibrations gives:
\begin{corollary}\label{cor:orthostr}
  Straightening over $A$ and $B$ give natural equivalences
 \begin{equation*}
 	\pushQED{\qed} 
    \label{eq:orthstr}
    \Fun(A^{\op}, \Cocart(B)) \simeq \Orth(A,B) \simeq
    \Fun(B, \Cart(A)). \qedhere
    \popQED
  \end{equation*}
\end{corollary}

Of course one can now apply another instance of the straightening functor on both outer terms. We will discuss the result in \S \ref{subsec:duallll} below.
For now, let us instead specialise the discussion further. Since interpolating edges always lie over equivalences in
$A \times B$, we find:
\begin{propn}\label{bifibinortho}
For a functor $p \colon X\rightarrow A\times B$ the following are
equivalent:
\begin{enumerate}
    \item $p$ is a conservative \lortho{}.
    \item $p$ is a \lortho{} whose fibres are \igpds{}.
    \item $p$ is a \lortho{} and $p_{\ell}$ is a left fibration.
    \item $p$ is a \lortho{} and $p_r$ is a right fibration.
    \item $p_1$ is a cartesian fibration and a morphism in $X$ is $p_{1}$-cartesian \IFF{} it is sent to an
    equivalence by $p_{2}$, and $p_2$ is a cocartesian fibration and a morphism in $X$ is $p_{2}$-cocartesian \IFF{} it is sent to an
    equivalence by $p_{1}$.
\end{enumerate}
If these conditions are satisfied, then $p$ is in particular an
orthofibration. \qed
\end{propn}

\begin{defn}\label{defn:bifib}
  A \emph{bifibration} is a functor $p = (p_{1},p_{2}) \colon X \to A
  \times B$ satisfying the equivalent conditions of the previous proposition.
  \end{defn}

Restricting the equivalence of \cref{cor:orthostr} to bifibrations
gives:
\begin{cor}\label{cor:bistr}
  Straightening over $A$ and $B$ gives natural equivalences
  \begin{equation*}
  	\pushQED{\qed}
   \Fun(A^{\op}, \LFib(B)) \simeq \Bifib(A,B) \simeq
    \Fun(B, \RFib(A)).   \qedhere
    \popQED
  \end{equation*}
\end{cor}

\subsection{Gray fibrations}\label{subsec:Gray}
We saw above that \lorthos{} over $A^{\op} \times B$ could be
unstraightened to functors $A \to \Cocartlax(B)$. We can also consider
the functors to $A \times B$ that correspond to such functors under
\emph{cocartesian} unstraightening over $A$, which leads to the
following definition:

\begin{defn}
  A \emph{Gray fibration} over $(A,B)$ is a functor $p \colon X \to A \times B$
  such that $p$ is cocartesian over $A$ and $p_{r} \colon X_{r}
  \to \core{A} \times B$ is a cocartesian fibration.  
  We write $\Gray(A,B)$ for the subcategory of $\Cat/(A\times B)$
  whose objects are the Gray fibrations, with morphisms required to
  preserve both types of cocartesian morphisms.

  Dually, we say $p \colon X \rightarrow A\times B$ is an \emph{op-Gray
    fibration} if $p^\op$ is a Gray fibration over $(A^\op,B^\op)$,
  and denote the \icat{} they span by $\Grayop(A,B)$.
\end{defn}
We will see in \cref{cor:gray fib is loc cocart of gray}
below that Gray fibrations over $(A,B)$ encode functors of
$(\infty,2)$-categories $A \Gtimes B \to \tCat$, where $\Gtimes$
denotes the Gray tensor product, which is the reason for the name. In particular, just as the Gray tensor product is not symmetric, let us point out that a Gray fibration $(p_1, p_2)\colon X\rt A\times B$ typically does \emph{not} determine a Gray fibration $(p_2, p_1)\colon X\rt B\times A$.

\begin{observation}
  From \cref{onevarcocart} and \cref{rmk:pbtocore} we see that a
  functor $p = (p_1,p_2) \colon X \to A \times B$ is a Gray fibration
  \IFF{} in the commutative triangle
  \[
    \begin{tikzcd}
      X \arrow{rr}{p} \arrow[dr, "p_1"{swap}] & & A \times B \arrow{dl}{\pr_{1}} \\
      & A,
    \end{tikzcd}
  \]
  $p_1$ is a cocartesian fibration, $p$ takes $p_1$-cocartesian
  morphisms to $\pr_{1}$-cocartesian morphisms, and for every $a \in
  A$ the map on fibres $X_{a} \to B$ is a cocartesian fibration.
\end{observation}

 Combining this observation with \cref{rmk:lcocstr}, and the same analysis as in \ref{cor:locorthstr}, we see:

\begin{corollary}\label{cor:graycocartstr}
  Straightening over $A$ gives a natural equivalence
  \[\pushQED{\qed}
  \Gray(A,B) \simeq \Fun(A, \Cocartlax(B))^{\mathrm{cc}}.\qedhere
  \popQED\]
\end{corollary}

Our next goal is to give an alternative characterisation of Gray
  fibrations, namely as those locally cocartesian fibrations that are cocartesian over certain triangles in the base.
This characterisation will be the key to
relating them to Gray tensor products below in
\S\ref{subsec:scunstrGray}. We first observe that Gray fibrations are
in particular locally cocartesian fibrations:

\begin{lemma}
  Let $p \colon X \to A \times B$ be a Gray fibration. Then every
  morphism in $X$ over $(\alpha,\beta)$ of the form
  $x \to (\id,\beta)_{!}x \to (\alpha,\id)_{!}(\id,\beta)_{!}x$, where the first
  morphism is $p$-cocartesian over $(\id,\beta)$ and the second is
  $p_{r}$-cocartesian over $(\alpha,\id)$, is locally $p$-cocartesian. In
  particular, $p$ is a locally cocartesian fibration where all locally
  $p$-cocartesian morphisms are of this form, and we have a
  fully faithful inclusion
  \[\Gray(A,B) \subseteq \lococart(A \times B).\]
\end{lemma}
\begin{proof}
  It follows from \cref{lem:cartinloccartfib} that the morphisms of the given
  form are locally $p$-cocartesian, since any $p_{r}$-cocartesian
  morphism in $X_{r}$ is in particular locally $p$-cocartesian.
  Thus $X$ has all locally $p$-cocartesian lifts, \ie{} $p$ is a locally
  cocartesian fibration. Moreover, all locally $p$-cocartesian
  morphisms are of the given form by uniqueness.
  
  It remains to show that a morphism $f
  \colon X \to Y$ between Gray fibrations over $A \times B$ preserves
  locally cocartesian morphisms \IFF{} it lies in $\Gray(A,B)$, which
  is immediate from the description of locally $p$-cocartesian
  morphisms in terms of the two types of cocartesian morphisms for a
  Gray fibration.
\end{proof}

\begin{remark}
  It is immediate from the definition that any cocartesian fibration
  over $A \times B$
  is a Gray fibration. Since $\Cocart(A \times B)$ is also a full
  subcategory of $\lococart(A \times B)$ it follows that we have a
  fully faithful inclusion
  \[\cocart(A \times B) \subseteq \Gray(A,B).\]
  The following characterisation pins down exactly how Gray fibrations
  fit in between cocartesian and locally cocartesian fibrations:
\end{remark}

\begin{lemma}\label{lem:gray reform}
  A locally cocartesian fibration $p\colon X\rt A\times B$ is a Gray fibration if and only if it restricts to a cocartesian fibration over each triangle $\sigma \colon [2]\rt A\times B$ of one of the following forms:
\[\begin{tikzcd}
{(a,
  b)}\arrow[r, "{(\id, \beta)}"]\arrow[rd, "{(\id,
  \beta'\beta)}"{swap}] & {(a, b')}\arrow[d, "{(\id, \beta')}"] & {(a, b)}\arrow[r, "{(\alpha, \id)}"]\arrow[rd, "{(\alpha'\alpha,
  \id)}"{swap}] & {(a', b)}\arrow[d, "{(\alpha', \id)}"] &  {(a,
  b)}\arrow[r, "{(\alpha, \id)}"]\arrow[rd, "{(\alpha, \beta)}"{swap}]
& {(a', b)}\arrow[d, "{(\id, \beta)}"]\\ 
& {(a, b'')} & & {(a'', b)} &  & {(a', b')}
\end{tikzcd}\]
\end{lemma}
\begin{proof}
  For a locally cocartesian
  fibration $p$, being cocartesian over the first type of triangle is
  equivalent to $p_r$ being a cocartesian fibration by \cref{cor:loccartiscartcond}. Using \cref{lem:cartinloccartfib},
  being cocartesian over the second and third types of triangles is
  equivalent to $p$ being a locally cocartesian fibration such that
  for any two locally cocartesian arrows $x\rt x'$ and $x'\rt x''$
  covering $(\alpha, \id)$ and $(\alpha', \beta)$ respectively, their composition is
  locally cocartesian as well. By \cref{lem:cartinloccartfib},
  this means precisely that $p$ admits cocartesian lifts over $A
  \times \core{B}$.
\end{proof}

We see that the difference between Gray and cocartesian fibrations lies in
the fact that in a Gray fibration the locally cocartesian lifts of the three edges in a
  diagram of the form
\[\begin{tikzcd} (a,b) \ar[r,"{(\id,\beta)}"] \ar[rd,"{(\alpha,\beta)}"'] & (a,b') \ar[d,"{(\alpha,\id)}"] \\
    & (a',b')\end{tikzcd}\] 
need not form a commutative diagram. We now analyze the
relationship between Gray fibrations and cocartesian fibrations over a
product more closely.

\begin{construction}\label{rmk:interpol Gray}
  Let $p \colon X \to A \times B$ be a Gray fibration. Consider an
  edge $(\alpha,\beta) \colon (a,b) \rightarrow (a',b')$ in
  $A \times B$ as above. Given a lift $x$ of the source of this edge we can choose $p_r$- and $p$-cocartesian lifts as in the solid part of
\begin{center}
\begin{tikzcd}
                                   & (\id,\beta)_!(\alpha,\id)_!x \arrow[rd, dashed] &    \\
(\alpha,\id)_!x \ar[rr,dotted]\arrow[ru, tail, "\circ" marking]  &               & (\alpha,\id)_!(\id,\beta)_!x \\
x \ar[ruu,dotted] \ar[rru,dotted]\arrow[r, tail, "\circ" marking] \arrow[u, tail]              & (\id,\beta)_!x. \arrow[ru, tail]  &   
\end{tikzcd}
\end{center}
(Recall that tailed arrows denote $p$-cocartesian edges and tailed
arrows marked by a circle denote $p_r$-cocartesian edges.) Now by
\cref{lem:cartinloccartfib} the composition $x \rightarrow (\alpha,\id)_!x \rightarrow (\id,\beta)_!(\alpha,\id)_!x$ along the top is still locally $p$-cocartesian, whence there
exists an essentially unique dashed arrow as indicated making the
diagram commute.  More formally, consider the functor
$\rho\colon [1]\rt \cocart^{\lax}([1])$ from \cref{ex:interpol}. Then
the cocartesian unstraightening of $\rho$ over
$[1]$ can be identified with the Gray fibration $Q'\rt [1]\times [1]$,
where $Q'$ is the poset 
\begin{center}
\begin{tikzcd}
	& 11' \arrow[rd] &    \\
	10 \arrow[ru,tail,"\circ" marking]  \ar[rr,dotted] &               & 11. \\
	00 \arrow[r,tail,"\circ" marking] \arrow[u,tail] \arrow[ruu,dotted] \ar[rru,dotted]              & 01 \arrow[ru,tail]  &   
\end{tikzcd}
\end{center}
and the projection is the evident one, sending $11'\to 11$ to the
identity. Then just as in \cref{ex:interpol} evaluation at $00$ induces an equivalence between the $\infty$-groupoid of maps of Gray fibrations
$$\begin{tikzcd}
Q'\arrow[r]\arrow[d] & X\arrow[d, "p"]\\
{[1]\times[1]}\arrow[r, "\alpha\times\beta"] & A\times B
\end{tikzcd}$$
and $\iota X_{(a,b)}$.
\end{construction}

\begin{definition}\label{def:interpol gray}
If $p\colon X\rt A\times B$ is a Gray fibration, then a morphism $Q'\rt X$ of Gray fibrations is said to be a \emph{$p$-interpolating diagram}.
A morphism in $X$ is said to be \emph{$p$-interpolating} if it arises as the restriction of a $p$-interpolating diagram to $11'\rt 11$.
\end{definition}

Note that we do not include in the notation whether an edge in
  $X$ is regarded as an interpolating edge for a Gray or curved
  orthofibration, assuming $p$ is both (a situation we will have to
  explicitly consider later). We will be more explicit when the need
  arises.

\begin{proposition}\label{cocartgray}
  Let $p \colon X\rightarrow A\times B$ be a Gray fibration. Then the
  following are equivalent:
  \begin{enumerate}[(1)]
  \item $p$ is a cocartesian fibration,
  \item $p$ restricts to a cocartesian fibration over each triangle
    $\sigma \colon [2] \rightarrow A \times B$ of the form
    \[\begin{tikzcd} (a,b) \ar[r,"{(\id,\beta)}"] \ar[rd,"{(\alpha,\beta)}"'] & (a,b') \ar[d,"{(\alpha,\id)}"] \\
        & (a',b'),\end{tikzcd}\]
  \item\label{cocartinterpolating} Every $p$-interpolating edge in
    $X$ is an equivalence.
  \item For every morphism $\alpha \colon a \to a'$ in $A$, the cocartesian transport functor $\alpha_{!} \colon X_{a} \to X_{a'}$ preserves
    cocartesian morphisms over $B$.
  \item The functor $A \to \Cocartlax(B)$ obtained by straightening
    $p$ over $A$ factors through the wide subcategory $\Cocart(B)$.
  
  \end{enumerate}
\end{proposition}
\begin{proof}
  Condition (1) immediately implies (2). Conversely, since $p$ is a
  Gray fibration it follows from (2) that locally $p$-cocartesian
  morphisms in $X$ are closed under composition:
  Given morphisms $(\alpha,\beta) \colon (a,b) \to (a',b')$ and
  $(\alpha',\beta') \colon (a',b') \to (a'',b'')$ in $A \times B$, we
  must show that for $x_{00} \in E_{(a,b)}$ the composite
  \[ x_{00} \to x_{11} \to x_{22} \]
  of locally $p$-cocartesian morphisms over $(\alpha,\beta)$ and
  $(\alpha',\beta')$ is again locally
  $p$-cocartesian over $(\alpha'\alpha,\beta'\beta)$. We can expand this
  to a composite
  \[ x_{00} \to x_{10} \to x_{11} \to x_{12} \to x_{22} \] of locally
  $p$-cocartesian morphisms over $(\alpha,\id)$, $(\id,\beta)$,
  $(\id,\beta')$, and $(\alpha',\id)$. Then the composite
  $x_{10} \to x_{12}$ is locally $p$-cocartesian over
  $(\id,\beta'\beta)$, and so the composite morphism
  $x_{00} \to x_{12}$ can alternatively be factored as
  $x_{00} \to x_{02} \to x_{12}$ where these morphisms are locally
  $p$-cocartesian over $(\id,\beta'\beta)$ and $(\alpha,\id)$. Then
  $x_{02} \to x_{12} \to x_{22}$ is locally $p$-cocartesian over
  $(\alpha'\alpha,\id)$, and so finally $x_{00} \to x_{02} \to x_{22}$
  is locally cocartesian over $(\alpha'\alpha,\beta'\beta)$ as
  required. Hence $p$ is a cocartesian fibration by
  \cref{cor:loccartiscartcond}.

  From \cref{rmk:interpol Gray} we see that (2) implies (3), since the
  $p$-interpolating morphisms are now obtained by factoring a morphism
  that is already locally $p$-cocartesian. Moreover, if all
  $p$-interpolating morphisms are invertible we can also conclude that
  the composite of locally $p$-cocartesian morphisms over a triangle
  as in (2) factors as a locally $p$-cocartesian morphism followed by
  an equivalence, and hence is again locally $p$-cocartesian.

  Since the cocartesian edges in $X_{a}$ over $B$ are precisely the
  locally $p$-cocartesian edges in $X$ that lie over $\id_{a}$, the
  equivalence of (3) and (4) is immediate from the definition of
  $p$-interpolating edges, while (5) is just a rephrasing of (4).
\end{proof}

The interpolating edges of a Gray fibration $p:X\rightarrow A\times B$ map to $\core(A \times B)$ by construction. From the analogous assertion for cocartesian fibrations we therefore immediately obtain:
\begin{corollary}\label{leftgray}
A Gray fibration $p:X\rightarrow A\times B$ is a left fibration if and
only if it is conservative, or equivalently if its fibres are
\igpds{}. \qed
\end{corollary}

\subsection{Dualisation of curved orthofibrations and Gray fibrations}\label{subsec:duallll}

In the present section we put together the pieces and analyse the dualisation equivalence promised in \cref{thm:grayorthdual}. 

\begin{thm}\label{thm:dual lortho-gray}
Cocartesian straightening followed by cartesian unstraightening over $A$ provides a natural equivalence
	\begin{equation}
	\label{eq:grayeqlorth}
	\begin{tikzcd}\Dualcart \colon\Gray(A,B)  \arrow[r, yshift=0.5ex] & \arrow[l, yshift=-0.5ex]  \LOrth(A^{\op},B) \cocolon \Dualco\end{tikzcd}
	\end{equation}
	between Gray fibrations over $A \times B$ and \lorthos{} over
	$A^{\op}\times B$. It restricts to the identity if
          $A=\ast$ and is the usual dualisation equivalence between
          cartesian and cocartesian fibrations if $B=\ast$. In
          particular, for all $(a,b) \in A \times B$ there are canonical equivalences
	\begin{equation}\Dualcart(p)_{(a,b)} \simeq X_{(a,b)} \quad\text{and}\quad \Dualco(q)_{(a,b)} \simeq Y_{(a,b)}\end{equation}
        for every  Gray fibration $p \colon X \rightarrow A \times B$ and curved orthofibration $q \colon Y \rightarrow A^\op \times B$.

Dually, there is an equivalence
\begin{equation}
\label{eq:opgrayeqlorth}
\begin{tikzcd}\Dualcart \colon\LOrth(A,B)  \arrow[r, yshift=0.5ex] & \arrow[l, yshift=-0.5ex]  \Grayop(A,B^\op) \cocolon \Dualco\end{tikzcd}
\end{equation}
with the analogous properties.
\end{thm}

\begin{proof}
	Combine the straightening equivalence of \cref{cor:graycocartstr}
	with the first equivalence of \cref{cor:locorthstr} to obtain 
	\[\Gray(A,B) \simeq \Fun(A,\Cocartlax(B))^\coc \simeq \LOrth(A^\op,B).\]
	The addenda are all immediate from the construction, and the dual case is obtained by using the the equivalence from \cref{cor:locorthstr} combined with the dual of \cref{cor:graycocartstr}.
\end{proof}

	\begin{propn}\label{lem:interpol dual}
		Let $p\colon X\rt A\times B$ be a Gray fibration and let $q\colon Y\rt A^{\op}\times B$ be the dual \lortho{}. For each $\alpha\colon a\rt a'$, $\beta\colon b\rt b'$ and $x\in X_{a', b}\simeq Y_{a', b}$, the canonical equivalence $X_{a, b'}\simeq Y_{a, b'}$ identifies the associated $p$-interpolating morphism from Definition \ref{def:interpol gray} with the associated $q$-interpolating morphism from Definition \ref{def:interpol lorth}.
	\end{propn}

\begin{proof}
	The statement immediately reduces to the case where $A=B=[1]$. Now by construction the Gray fibration	$Q'\rt [1]\times [1]$ from \cref{def:interpol gray} is dual to the
	\lortho{} $Q\rt [1]^{\op}\times [1]$ from \cref{ex:interpol}, so it follows that the \igpd{} of interpolating diagrams $Q'\rt X$ and $Q\rt Y$ are equivalent. Since dualisation identifies the morphism $11'\rt 11$ in $Q$ with $11'\rt 11$ in $Q'$, dualisation preserves interpolating morphisms as well.
\end{proof}

\begin{corollary}\label{cor:strothbifibb}
	The equivalences from \cref{thm:dual lortho-gray} restrict to equivalences
	\[\begin{tikzcd}\Cocart(A \times B)  \arrow[r, yshift=0.5ex] & \arrow[l, yshift=-0.5ex]  \Orth(A^{\op},B) & and & \Leftfib (A \times B)  \arrow[r, yshift=0.5ex] & \arrow[l, yshift=-0.5ex]  \Bifib(A^{\op},B)\end{tikzcd}\]
	and dually
	\[\begin{tikzcd}\Orth(A, B)  \arrow[r, yshift=0.5ex] & \arrow[l, yshift=-0.5ex]  \Cart(A \times B^{\op}) & and & \Bifib (A, B)  \arrow[r, yshift=0.5ex] & \arrow[l, yshift=-0.5ex]  \Rightfib(A \times B^{\op}).\end{tikzcd}\]
\end{corollary}

\begin{proof}
	The left hand equivalences follow by replacing the use of \cref{cor:locorthstr} and \cref{cor:graycocartstr} in \cref{thm:dual lortho-gray} with \cref{cor:orthostr} and straightening for (co)cartesian fibrations. Alternatively, using the previous proposition they follow from characterisation \ref{cocartinterpolating} of \cref{cocartgray}. The statement about left and bifibrations follows by inspecting fibres.
\end{proof}

\begin{remark}\label{remark:annoyingsquare} \begin{enumerate}
	\item Equivalences as on the right were first constructed by Stevenson, by comparing both \(\bifib(A,B)\) and \(\leftfib(A^{\op} \times B)\) to an \icat{} of correspondences \cite[Theorems C \& D]{Stevenson}.  In the companion paper \cite{part2} we will prove a uniqueness
		result for the equivalences above that in particular shows that our equivalences restrict to those of Stevenson.%
	
\item\label{item:annoyingsquareisannoying} From \cref{cor:strothbifibb} we obtain a diagram of equivalences
\[\begin{tikzcd} & \ar[rd, "\sim"]\ar[ld] \Orth(A,B) & \\
\Cocart(A^\op \times B) \ar[ru, "\sim"]\ar[rr, "\sim"]& &\ar[ll] \ar[lu]\Cart(A \times B^\op)\end{tikzcd}\]
where the lower maps are given by dualisation in a single
  variable (i.e. over $A^\op \times B$). It is not a priori clear that
  this diagram commutes, but this will also be a consequence of the
  results in \cite{part2}. Combined with the usual
	straightening equivalences for (co)cartesian fibrations, we similarly obtain two a priori different equivalences
	\[ \Orth(A,B) \simeq \Fun(A^{\op} \times B, \Cat), \]
	given by straightening first over $A$ and then over $B$, or vice versa. Both restrict to equivalences
	\[ \Bifib(A,B) \simeq \Fun(A^{\op} \times B, \Gpd)\] 
	 and their agreement seems to be new even in this latter case.
\item By restricting to one of the two legs in the previous point, the dualisation of bifibrations is also discussed in detail in \cite[Section 5]{HLAS}, \cite[Section A.1]{cois} and \cite[Section 7.1]{CDH1}
\item In \cite{part2} we also supply a more explicit description of the equivalences in \cref{thm:dual lortho-gray} based on span \icats{}, generalising the work of Barwick, Glasman and Nardin \cite{BGN} in the single-variable case.
\end{enumerate}
\end{remark}

As a typical example of the dualisation procedure above, consider the bifibration $(s,t) \colon \Ar(C) \rightarrow C \times C$. Its duals are the twisted arrow categories of $C$; let us briefly recall these to fix conventions.

\begin{notation}\label{not:Tw}
For an $\infty$-category $\cat{C}$, we write $\TwL(\cat{C})$ and
  $\TwR(\cat{C})$ for the left and
  right \emph{twisted arrow $\infty$-category} of $\cat{C}$. These are
  characterised by the natural equivalences
  \[ \Map\big([n], \TwR(\cat{C})\big) \simeq \Map\big([n] \star [n]^{\op},
    \cat{C}\big), \qquad \Map\big([n], \TwL(\cat{C})\big) \simeq \Map\big([n]^{\op} \star [n],
    \cat{C}\big),\]
    so that $\TwR(\cat{C}) = \TwL(\cat{C})^{\op}$. The natural inclusions of $[n]$ and $[n]^{\op}$ correspond to
  functors
  \[(s,t) \colon \TwL(\cat{C}) \longrightarrow \cat{C}^{\op} \times \cat{C}, \qquad\qquad
    (s,t) \colon \TwR(\cat{C}) \longrightarrow \cat{C} \times \cat{C}^{\op},\]
  which are a left fibration and a right fibration, respectively, both
  straightening to the mapping functor \[\Map_{\cat{C}}
    \colon \cat{C}^{\op} \times \cat{C} \to \sS.\]
\end{notation}

\begin{remark}
  Informally, the objects of $\TwR(\cat{C})$ are the morphisms in $\cat{C}$. For morphisms $f
  \colon x \to y$, $f' \colon x' \to y'$ in $\cat{C}$, a morphism from $f$ to $f'$ in $\TwR(\cat{C})$ is a
  commutative diagram
    \[
    \begin{tikzcd}
      x \arrow{d}[swap]{f}\arrow{r}  & x' \arrow{d}{f'}  \\
      x' & y'. \arrow{l}
    \end{tikzcd}
    \]
\end{remark}

\begin{example}\label{ex:twardualar}
There are canonical equivalences
\[\Dualco(\Ar(C) \rightarrow C \times C) \simeq \TwL(C) \rightarrow C^\op \times C\] and
\[\Dualcart(\Ar(C) \rightarrow C \times C) \simeq \TwR(C) \rightarrow C \times C^\op.\]
This is proved for example in \cite[Corollary A.2.5]{cois}, based on Lurie's recognition criterion for twisted arrow categories \cite[Corollary 5.2.1.2]{HTT}. We supply another proof in \cite{part2} and also extend the statement to the (op)lax arrow and twisted arrow categories of an $(\infty,2)$-category.
\end{example}

\section{Parametrised and monoidal adjunctions}\label{sec:adjunctions}

In the present section we will use the results of Section
\ref{sec:ortho} to study the operation of taking adjoint functors in
families. The statements we prove in \S\ref{subsec:paraadj} boil down to the fact that for any lax natural transformation $g\colon \cat{F}\Rightarrow \cat{G}$ between two diagrams of $\infty$-categories, such that each component of $g$ is a right adjoint, the pointwise left adjoints assemble into an oplax natural transformation $f\colon \cat{G}\Rightarrow \cat{F}$. For the moment, we shall, however, stay in the
fibrational picture and instead consider maps between the associated
(co)cartesian fibrations. In particular, we will prove Theorem
\ref{thm:fibadj} from the introduction. The translation of this
statement into the mate correspondence for (op)lax transformations
will be delayed till \S\ref{subsec:laxnat}. In
\S\ref{subsec:mateident} and \S\ref{subsec:identifications} we carry
out two consistency checks: In the former we show that for each
morphism our functorial
passage to fibrewise adjoints is given by the Beck-Chevalley
construction on morphisms, and in the latter we prove that the fibrewise adjoints we produce in the fibrational picture are characterised by the expected relation on morphism \igpds{} from the left to the right adjoint.

In \S\ref{subsec:laxmonadj} we then finally specialise the discussion to maps of $\infty$-operads and produce the correspondence between lax $\cat O$-monoidal structures on a right adjoint functor and oplax $\cat O$-monoidal structures on its left adjoint. In particular, we will prove Proposition \ref{propa} and Corollary \ref{cord} here. 

\subsection{Parametrised adjunctions}\label{subsec:paraadj}

We start by considering adjunctions in families over a base \icat{} $B$:

\begin{definition}\label{def:param adj}
  A map $g\colon C\rt D$ in $\cocart^{\lax}(B)$ is said to be a
  \emph{$B$-parametrised right adjoint} if it induces right adjoint
  functors between the fibres over each $b\in B$. Dually, a map
  $f\colon D\rt C$ in $\cart^{\oplax}(B)$ is said to be a
  \emph{$B$-parametrised left adjoint} if it induces left adjoint
  functors between the fibres.

Let us write $\cocart^{\lax, \mm{R}}(B)$ and
$\cart^{\oplax,\mm{L}}(B)$ for the wide subcategories of
$\cocart^{\lax}(B)$ and $\cart^\oplax(B)$ whose maps are
$B$-parametrised right and left adjoints, respectively.
\end{definition}

As defined the categories $\cocart^{\lax, \mm{R}}(B)$ and
 $\cart^{\oplax,\mm{L}}(B)$ are oblivious to the fact that there are
  non-invertible transformations between $B$-parametrised adjoints. As
  it is often important not to forget these when passing to adjoints
  (in the specialisation to symmetric monoidal categories in
  \S\ref{subsec:laxmonadj} they correspond to symmetric monoidal
  natural transformations for example), we first enhance the \icats{}
  from \cref{def:param adj} to \itcats{} that encode natural
  transformations as their $2$-morphisms.

For this we will use the description of \itcats{} as complete
  $2$-fold Segal \igpds{}:
\begin{defn}
  A \emph{complete 2-fold Segal \igpd{}} is a functor $X \colon \Dop
  \times \Dop \to \Gpd$ such that
  \begin{enumerate}[(1)]
  \item the simplicial \igpds{} $X_{n,\bullet}$ and $X_{\bullet,m}$
    satisfy the Segal condition for all $n,m$,
  \item the simplicial \igpd{} $X_{0,\bullet}$ is constant,
  \item the Segal \igpds{} $X_{\bullet,0}$ and $X_{1,\bullet}$ (and
    hence $X_{n,\bullet}$ for all $n$) are complete.
  \end{enumerate}
  Note that by \cite[Lemma 2.8]{JohnsonFreydScheimbauer} these
  conditions imply that $X_{\bullet,m}$  is also complete for all
  $m$.
\end{defn}

We use the following general construction to enhance our \icats{} to
\itcats{}:
\begin{propn}\label{lem:extract2cat}
  Suppose $F \colon \Cat^{\op} \to \CAT$ is a limit-preserving functor
  such that for every \icat{} $B$ the functor $F(|B|) \to F(B)$
  arising from the canonical map $B \to |B|$ induces a monomorphism
  $\iota F(|B|) \to \iota F(B)$ on underlying \igpds{}.  If we define
  $F_{\mathrm{lc}}(B) \subseteq F(B)$ to be the full subcategory
  spanned by the image of $F(|B|)$ under this functor, then
  $F_{\mathrm{lc}}$ is also a limit-preserving functor
  $\Cat^{\op} \to \CAT$, and the bisimplicial \igpd{}
  \[ ([n],[m]) \mapsto \Map_{\CAT}([n], F_{\mathrm{lc}}([m])) \]
  is a complete 2-fold Segal space.
\end{propn}
\begin{proof}
  Since the map $B \to |B|$ is a natural transformation in $B$, we see
  that $F_{\mathrm{lc}}$ is a subfunctor of $F$. Note that the
  condition that $\iota F(|B|) \to \iota F(B)$ is a monomorphism
  implies that $\iota F(|B|) \to \iota F_{\mathrm{lc}}(B)$ is an
  equivalence. For any colimit of \icats{} $B
  \simeq \colim_{i}B_{i}$ we then have a commutative diagram
  \[
    \begin{tikzcd}
      F(|B|) \arrow{r}{\sim} \arrow{d} & \lim_{i} F(|B_{i}|) \arrow{d} \\
      F_{\mathrm{lc}}(B) \arrow{r}
      \arrow[hookrightarrow]{d} & \lim_{i} F_{\mathrm{lc}}(B_{i})
      \arrow[hookrightarrow]{d} \\
      F(B) \arrow{r}{\sim} & \lim_{i} F(B_{i}),
    \end{tikzcd}
  \]
  where the top and bottom horizontal maps are equivalences since $F$
  preserves limits (and $|-| \colon \Cat \to \Gpd \hookrightarrow
  \Cat$ preserves colimits), and the bottom right vertical map is
  fully faithful since fully faithful maps are closed under
  limits. Hence the middle horizontal functor is also fully faithful,
  by the 2-of-3 property for equivalences applied to mapping \igpds{}.
  In the top square the vertical morphisms are both given by
  equivalences on underlying \igpds{}, since this condition is also
  closed under limits. By the 2-of-3 property it follows that the
  middle horizontal functor is also an equivalence on underlying
  \igpds{}, and hence it is an equivalence. Thus $F_{\mathrm{lc}}$
  preserves limits.

  It follows that the functor
  \[ ([n],[m]) \mapsto \Map_{\CAT}([n], F_{\mathrm{lc}}([m])) \]
  satisfies the Segal and completeness conditions levelwise in each
  variable, since these can be expressed as taking certain colimits in
  $\Cat$ to limits. It remains only to observe that for $n = 0$ the
  simplicial space $\iota F_{\mathrm{lc}}([m])$ is indeed constant:
  the unique map $[m] \to [0]$ is the localisation $[m] \to |[m]|
  \simeq *$ and so we know that the map $\iota F_{\mathrm{lc}}([0]) \to \iota
  F_{\mathrm{lc}}([m])$ is an equivalence; since $[0]$ is terminal in
  $\Delta$ the diagram is then necessarily constant.
\end{proof}

\begin{remark}
  If we regard a simplicial \icat{} $X \colon \Dop \to \Cat$ that
  satisfies the Segal condition as a \emph{double \icat{}} whose
  objects are the objects of $X_{0}$, horizontal morphisms are the
  morphisms in $X_{0}$, vertical morphisms are the objects of $X_{1}$,
  and squares are the morphisms in $X_{1}$, then the construction of
  \cref{lem:extract2cat} can be interpreted as extracting an \itcat{} from the
  double \icat{} $[n] \mapsto F([n])$ by forgetting the non-invertible
  vertical morphisms. Such a construction can be performed more
  generally, but the conditions in \ref{lem:extract2cat} seem required
  to ensure the resulting 2-fold Segal space is complete.
\end{remark}

Returning to our specific situation, for any $\infty$-category $A$ we have natural equivalences
\[
\Map_{\CAT}(A, \cocart^{\lax}(B\times S)\big)\simeq \core\Ortholax(A^{\op}, B\times S)\simeq \Map_{\CAT}(B\times S, \cart^{\oplax}(A)\big)
\]
by \cref{cor:locorthstr}. By the Yoneda lemma, this implies that for
all $B$, the functor
\[  \Cat^{\op} \to \CAT, \quad S \mapsto
  \cocart^{\lax}\big(B\times S\big)\]
preserves limits. Moreover, on underlying \igpds{} we have equivalences
\[ \iota \Cocartlax\big(B\times S\big) \simeq \iota \Cocart\big(B
  \times S\big) \simeq \Map(B \times S, \Cat) \simeq \Map(S,
  \Fun(B,\Cat)),\] so that the functor
$\iota \Cocartlax(B \times |S|) \to \iota \Cocartlax(B \times S)$ corresponds to
the functor \[\Map(|S|, \Fun(B, \Cat)) \to \Map(S, \Fun(B, \Cat))\]
given by composition with $S \to |S|$; this is therefore a
monomorphism by the universal property of the localisation $|S|$,
which says that $\Fun(|S|, X) \to \Fun(S, X)$ is fully faithful with
image those functors that take all morphisms in $S$ to equivalences. 

Let us denote by 
\[\cocart^{\lax}_{S}\big(B\times S\big)\subseteq \cocart^{\lax}\big(B\times S\big)\] 
the full subcategory of cocartesian fibrations which are locally
constant on $S$, \ie{} those obtained by pulling back a cocartesian
fibration over $B \times |S|$, or equivalently those
whose straightening to a functor
$B \times S \rightarrow \Cat$ factors through the localisation to $B \times |S|$. Applying
\cref{lem:extract2cat}  we then have:
\begin{cor}
  The functor
  $$\begin{tikzcd}
    \Cat^{\op}\arrow[r] & \CAT; & S\arrow[r, mapsto] & \cocart^{\lax}_{S}\big(B\times S\big)
  \end{tikzcd}
  $$
  preserves limits, and the bisimplicial space
  \begin{equation}\label{eq:nerve of 2-cat of lax functors}
    \begin{tikzcd}
      \big([m], [n]\big)\arrow[r, mapsto] & \Map_{\CAT}\Big([m], \cocart^{\lax}_{[n]}\big(B\times [n]\big)\Big)
    \end{tikzcd}
  \end{equation}
  is a complete 2-fold Segal \igpd{}.\qed
\end{cor}
The same assertion holds if we instead take
$\cocart^{\lax, \mm{R}}_{S}(B\times S)$,
$\cart^{\oplax}_{S}(B\times S)$ or
$\cart^{\oplax, \mm{L}}_{S}(B\times S)$, which are all defined
analogously.

\begin{definition}\label{def:2-cat of cart oplax}
Let $B$ be a small $\infty$-category. We define $\mathbf{Cocart}^{\lax}(B)$ to be the $(\infty, 2)$-category associated to the complete $2$-fold Segal \igpd{} \eqref{eq:nerve of 2-cat of lax functors}. Likewise, we define the $(\infty, 2)$-category $\mathbf{Cart}^{\oplax}(B)$ to be the $(\infty, 2)$-category associated to the $2$-fold complete Segal \igpd{}
\begin{equation}\label{eq:nerve of 2-cat of oplax functors}\begin{tikzcd}
\big([m], [n]\big)\arrow[r, mapsto] & \Map_{\CAT}\Big([m], \cart^{\oplax}_{[n]}\big(B\times [n]\big)\Big)
\end{tikzcd}\end{equation}
We define the $(\infty, 2)$-categories $\mathbf{Cocart}^{\lax, \mm{R}}(B)$ and $\mathbf{Cart}^{\oplax, \mm{L}}(B)$ similarly.
\end{definition}

In the special case where $B=\ast$, the equivalent $2$-fold complete Segal spaces
\[\mathbf{Cocart}^{\lax}(*) \simeq \mathbf{Cart}^{\oplax}(*)\]
provide a model for the $(\infty, 2)$-category $\tcat{Cat}$ of  $\infty$-categories (this is proved more precisely in Section \ref{subsec:laxnat}). Consequently, we can identify
\[\mathbf{Cocart}^{\lax,\mm R}(*) \simeq \tcat{Cat}^{\mm R} \quad \text{and} \quad \mathbf{Cart}^{\oplax,\mm L}(*) \simeq \tcat{Cat}^{\mm L}.\]

\begin{observation}\label{rmk:morincocartS}
  If the \icat{} $S$ has contractible realisation (\ie{} $|S| \simeq
  *$)
  then the objects of $\cocart^{\lax}_{S}\big(B\times S\big)$ are by definition
  the cocartesian fibrations over $B \times S$ that are pulled back
  along the projection $B \times S \to B$, \ie{} those of the
  form $E
  \times S \to B \times S$ for a cocartesian fibration $E \to B$. A
  morphism between two such objects can then be identified with a
  commutative triangle
  \[
    \begin{tikzcd}
      E \times S \arrow{dr} \arrow{rr} & & E' \arrow{dl} \\
       & B
    \end{tikzcd}
  \]
  for cocartesian fibrations $E,E' \to B$. Note that this applies in
  particular for $S = [n]$. In particular, a $2$-morphism in
  $\mathbf{Cocart}^{\lax}(B)$ is simply a natural transformation $\mu$
  over $B$ of maps $g,g'$ between cocartesian fibrations $X \rightarrow B$ and $Y \rightarrow B$.

  This we can view as a family of natural
  transformations $\mu_b\colon g_b\rt g'_b$ that
  commutes with the lax structure maps, in the sense that for each
  $b\rt b'$, there is a commuting diagram
$$\begin{tikzcd}[column sep=3.4pc, row sep=1.9pc]
& X_b\arrow[ldd, ""{right, name=s3}]\ar[r, bend left, "g_b", ""{name=s1, below}]\ar[r, bend right, "{g'_b}"{below}, ""{above, name=t1}] & Y_b\arrow[ldd, ""{left, name=t3}]\arrow[Rightarrow, from=s1, to=t1, "\mu_b"]\\
\\
X_{b'}\ar[r, bend left, "g_{b'}", ""{name=s2, below}]\ar[r, bend right, "{g'_{b'}}"{below}, ""{above, name=t2}] & Y_{b'}.\arrow[Rightarrow, from=s2, to=t2, "\mu_{b'}"]\arrow[Leftarrow, bend left=40, from=s3, to=t3, start anchor={[xshift=3ex, yshift=1ex]}, end anchor={[xshift=-2ex, yshift=-1ex]}]\arrow[Leftarrow, bend right=30, from=s3, to=t3, start anchor={[xshift=2ex, yshift=0ex]}, end anchor={[xshift=-3ex, yshift=-2ex]}]
\end{tikzcd}$$
Depicting this diagram cubically, it can also be viewed as a lax
natural transformation between two functors $B\times [1]\rt \Cat$ that
are constant along the interval. Note that $\mathbf{Cocart}^{\lax, \mm{R}}(B)\subseteq \mathbf{Cocart}^{\lax}(B)$ is the $1$-full sub-$2$-category whose morphisms are lax natural transformations consisting of right adjoints.
\end{observation}

\begin{remark}\label{rem:opposite of lax is oplax}
Note that for any two $\infty$-categories $B$ and $S$, taking opposite $\infty$-categories defines an equivalence
$$
(-)^{\op}\colon \cocart_S^{\lax}(B\times S)\isoto \cart_{S^{\op}}^{\oplax}\big(B^{\op}\times S^{\op}\big).
$$
Using this, one deduces that taking opposite $\infty$-categories defines an equivalence of $(\infty, 2)$-categories, where in the target the $2$-morphisms are reversed
$$
(-)^{\op}\colon \mathbf{Cocart}^{\lax}(B)\isoto \mathbf{Cart}^{\oplax}\big(B^{\op}\big)^{ 2\mm{-op}}.
$$
\end{remark}
We now come to our main technical result, \cref{thm:fibadj}:
\begin{theorem}\label{thm:dualizing lax natural transformations}
Let $B$ be an $\infty$-category. Then there is a natural equivalence of $(\infty, 2)$-categories
\begin{equation}\label{eq:taking adjoint 2-functorially}
\Adj\colon \mathbf{Cocart}^{\lax, \mm{R}}(B)\rto{\sim} \mathbf{Cart}^{\oplax, \mm{L}}(B^{\op})^{(1, 2)\mm{-}\op}
\end{equation}
sending each cocartesian fibration to the cartesian fibration classifying the same functor $B\rt \Cat$. Here in the target the directions of $1$- and $2$-morphisms are changed, as indicated.
\end{theorem}

In particular, for $B=\ast$ this produces an equivalence
\[\mathbf{Cat}^{\mm{R}}\simeq (\mathbf{Cat}^{\mm{L}})^{(1, 2)\mm{-}\op}.\]

For the proof recall first that a functor $g\colon C\rt D$ between \icats{} is a right
adjoint if the corresponding cartesian fibration $p\colon X\rt [1]$ is
a cocartesian fibration as well, see \cite[Section 5.2.2]{HTT}. Dually, a
functor $f\colon D\rt C$ is a left adjoint if the corresponding
cocartesian fibration is a cartesian fibration as well. In other
words, one can encode adjunctions by functors $p\colon X\rt [1]$ that
are simultaneously cartesian and cocartesian fibrations; the two
adjoint functors can be extracted from this by (co)cartesian
straightening. 

We now extend this statement by showing that a functor between two $B$-parametrised categories (in the form of cocartesian fibrations over $B$) is a parametrised right adjoint if and only if the corresponding curved orthofibration over $[1] \times B$ is also a Gray fibration, when considered over $B \times [1]$, and similarly in the dual situation. More generally:

\begin{lemma}\label{lem:lortho with adj}
Let $p=(p_1, p_2)\colon X\rt A\times B$ be a functor. Then the following conditions are equivalent:
\begin{enumerate}
\item\label{it:straightened adjoint fib} $p$ is a \lortho{} and the functor $A^{\op}\rt \Cocart^{\lax}(B)$ classifying $p$ via Corollary \ref{cor:locorthstr} takes values in the wide subcategory $\Cocart^{\lax, \mm{R}}(B)$,
\item\label{it:straightened adjoint fib2} $p$ is a \lortho{} whose restriction $p_l$ to $A\times \core(B)$ is a cocartesian fibration as well,
\item\label{it:straightened adjoint fib3} $\ol{p}=(p_2,
    p_1)\colon X\rt B\times A$ is a Gray fibration and the functor $B
    \rightarrow \Cocart^{\lax}(A)$ classifying $B$ takes values in the
    full subcategory $\mathrm{Bicart}^{\mathrm{(op)lax}}(A)$,
\item\label{it:straightened adjoint fib4} $\ol{p}=(p_2, p_1)\colon X\rt B\times A$ is a Gray fibration whose restriction to $\core(B)\times A$ is a cartesian fibration as well.
\end{enumerate}
Dually, a \lortho{} $q=(q_1, q_2)\colon Y\rt B\times A$ classifies a
functor $A\rt \Cart^{\oplax, \mm{L}}(B)$ via \cref{cor:locorthstr} if
and only if $q_r$ is a cartesian fibration as well, or equivalently if
and only if $(q_2, q_1)\colon Y\rt A\times B$ is an op-Gray fibration, which is then automatically classified by a functor $A^\op \rightarrow \mathrm{Bicart}^{\mathrm{(op)lax}}(B)$.
\end{lemma}
\begin{proof}
  For the equivalence between \ref{it:straightened adjoint fib} and \ref{it:straightened adjoint fib2}, we claim that one can check both conditions fibrewise in $B$. Namely for \ref{it:straightened adjoint fib} this follows by the naturality of the straightening equivalence (\cref{cor:locorthstr}) in $B$, and for \ref{it:straightened adjoint fib2} this is an immediate consequence of Corollary \ref{cor:cocartprgpd}. So it suffices to prove the two conditions are equivalent for $B = \{\ast\}$, where the assertion becomes that a cartesian fibration classifies a diagram of \icats{} and
  right adjoints if and only it is also a cocartesian fibration, which
  is \cite[Proposition 4.7.4.17]{HA}. The equivalence between \ref{it:straightened adjoint fib2} and
  \ref{it:straightened adjoint fib4} follows from characterisation \ref{lorthocartr} of \lortho{s} in
  \cref{prop:lortho reformulations}, and finally the equivalence between \ref{it:straightened adjoint fib3} and \ref{it:straightened adjoint fib4} is part of Corollary \ref{cor:cocartprgpd}.
\end{proof}

Let us write $\cat{M}(A, B)$ for the \igpd{} of functors $p\colon X\rt A\times B$ satisfying the equivalent conditions of Lemma \ref{lem:lortho with adj}, so that there are natural inclusions of path components
$$\begin{tikzcd}
\core\Gray(B, A)\arrow[r, hookleftarrow] & \cat{M}(A, B)\arrow[r, hookrightarrow] & \core\Ortholax(A, B)
\end{tikzcd}$$
where the left inclusion sends $p=(p_1, p_2)$ to $(p_2, p_1)$.
Likewise, let us write $\cat{N}(A, B)$ for the \igpd{} of functors $q\colon Y\rt B\times A$ satisfying the equivalent opposite conditions of Lemma \ref{lem:lortho with adj}, so that there are natural inclusions of path components
$$\begin{tikzcd}
\core\Grayop(A, B)\arrow[r, hookleftarrow] & \cat{N}(A, B)\arrow[r, hookrightarrow] & \core\Ortholax(B, A).
\end{tikzcd}$$
More generally, let us write $M_S(A, B\times S)\subseteq M(A, B\times S)$ and $N_S(A, B\times S)\subseteq N(A, B\times S)$ for the natural subspaces spanned by fibrations $p\colon X\rt A \times (B\times S)$ such that each $X_{a, b}\rt S$ is locally constant, i.e.\ the associated functor factors through $|S|$.

\begin{corollary}\label{cor:fibration model}
For any \icat{} $B$, unstraightening over $[m]$ provides natural equivalence of $2$-fold complete Segal spaces
$$\begin{tikzcd}[row sep=0.2pc]
\Uncart\colon \Map_{\CAT}\Big([m], \cocart^{\lax, \mm{R}}_{[n]}\big(B\times [n]\big)\Big)\arrow[r, "\simeq"] & M_{[n]}\big([m]^{\op}, B\times [n]\big)\\
\Unco\colon \Map_{\CAT}\Big([m], \cart^{\oplax, \mm{L}}_{[n]}\big(B\times [n]\big)\Big)\arrow[r, "\simeq"] & N_{[n]}\big([m], B\times [n]\big).
\end{tikzcd}$$
\end{corollary}
\begin{proof}
Apply \cref{lem:lortho with adj} and use that local constancy along $[n]$ can be checked when $[m]=\ast$, in which case the unstraightening functors are equivalent to the identity.
\end{proof}
\begin{lemma}\label{lem:dualisation of bicart}
  The dualisation functor from \cref{thm:dual lortho-gray} with respect to $B\times S$
$$
\Dualcart\colon \Gray\big(B\times S, A\big)\rto{\sim} \Ortholax\big((B\times S)^{\op}, A\big)
$$
restricts to an equivalence of \igpds{}
$$
\Dualcart\colon \cat{M}_S(A, B\times S)\rto{\sim} \cat{N}_{S^{\op}}(A, (B\times S)^{\op}).
$$ 
\end{lemma}
\begin{proof}
When $A=\ast$, dualisation over $B\times S$ simply sends cocartesian fibrations to their dual cartesian fibrations. This preserves local constancy in $S$ and by naturality in $A$, one sees that dualisation preserves those objects that restrict to locally constant fibrations over $\{b\}\times S\times A$.
 
  By the addenda of \cref{thm:dual lortho-gray}, for $B\times S=*$ the dualisation
  equivalence restricts
  to a natural self-equivalence of the \icat{} of cocartesian
  fibrations over $A$ that is equivalent to the identity. By
  naturality in $B\times S$, one therefore sees that the dualisation preserves
  those objects that restrict for each $x\in B\times S$ to a bicartesian
  fibration over $\{x\}\times A$, as required.
\end{proof}
\begin{proof}[Proof of \cref{thm:dualizing lax natural transformations}]
\cref{cor:fibration model} and \cref{lem:dualisation of bicart} yield a natural equivalence
\begin{equation}\label{eq:mate correspondence}\begin{tikzcd}[column sep=2.2pc]
\Map_{\scriptscriptstyle\CAT}\big(A^{\op}, \cocart^{\lax, \mm{R}}_S(B\times S)\big)\arrow[d, "\sim"{swap}]\arrow[r, "\Uncart"{above}, ] & M_S(A, B\times S)\arrow[d, "{\Dualcart}"{left}, ] \\
\Map_{\scriptscriptstyle\CAT}\big(A, \cart^{\oplax, \mm{L}}_{S^{\op}}(B^{\op}\times S^{\op})\big) & N_{S^{\op}}(A^\op, B^{\op}\times S^{\op})\arrow[l, "\Strco"{above}]
\end{tikzcd}\end{equation}
Taking $A$ and $S$ to be simplices, one obtains the desired equivalence between $2$-fold Segal spaces $\mathbf{Cocart}^{\lax, \mm{R}}(B)\simeq \mathbf{Cart}^{\oplax, \mm{L}}(B^{\op})^{(1, 2)\mm{-}\op}$.
\end{proof}

%

\begin{example}\label{ex:monoidalclosed}
  A \emph{two-variable adjunction} consists of functors $F \colon B
  \times C \rightarrow D$ and $G \colon B^{\op} \times D \to C$,
  together with a natural equivalence
  \[ \Map_{D}(F(b,c), d) \simeq \Map_{C}(c, G(b,d));\]
  the prototypical example is the tensor-hom adjunction in a (left-)closed
  monoidal \icat{}. This is a special case of our parametrised
  adjunctions: It follows from the Yoneda lemma that given $F$, the functor $G$ is
  uniquely determined and exists \IFF{} $F(b,-)$ is a left
  adjoint for all $b \in B$. We can then view $F$ as a parametrised
  left adjoint
  \[
    \begin{tikzcd}
      B \times C \arrow{rr}{(\pr_1,F)} \arrow[dr,"\pr_{1}"'] & & B \times D
      \arrow{dl}{\pr_{1}} \\
       & B.
    \end{tikzcd}
  \]
  Since the dual cocartesian fibration to $\pr_{1}
  \colon B \times C \to B$ is the projection $B^{\op} \times C \to
  B^{\op}$,
  \cref{thm:dualizing lax natural transformations} produces a
  parametrised right adjoint in the form
  \[
    \begin{tikzcd}
      B^\op \times C \arrow[dr,"\pr_{1}"'] & & B^\op \times D \arrow[ll,"{(\pr_1,G)}"']
      \arrow{dl}{\pr_{1}} \\
       & B^\op.
    \end{tikzcd}
  \]
  At the moment we only know that $G(b,-)$ is right adjoint to
  $F(b,-)$ for each $b$, but we will verify in
  \cref{cor:identifyingadjoints} below that $G$ indeed gives the
  expected natural equivalence on mapping \igpds{}. We will apply this fibrational
  description of two-variable adjunctions to analyse the monoidal properties of the internal mapping functor in
  \cref{cor::monoidalinternalhom}.
\end{example}

\subsection{Identifying mates}\label{subsec:mateident}

Our goal in this subsection is to describe the effect on morphisms of the equivalence from \cref{thm:dualizing lax natural transformations} in terms of mates or Beck--Chevalley transformations, see \cref{beck-chevalley} below. 

In order to do this, let us first recollect how one can
obtain the unit and counit of the adjunction from a bicartesian fibration $p\colon X\rt [1]$, using the following general
construction:
\begin{construction}\label{con:cocartesian push}
Let $p\colon X\rt [1]$ be a cocartesian fibration and $I$ any \icat{}. By \cite[Proposition 3.1.2.1]{HTT}, post-composition with $p$ determines a cocartesian fibration \[ p^{I} \colon \Fun(I, X) \to \Fun(I,[1]), \] 
with cocartesian morphisms those natural transformations that are given by $p$-cocartesian morphisms at each object of $I$. Given a functor $\phi \colon I \to X$, its \emph{cocartesian transport functor} $\phi_{\coc} \colon I \times [1] \to X$ is the diagram corresponding to the essentially unique $p^{I}$-cocartesian morphism with domain $\phi$ covering the map $p\phi\Rightarrow \mm{const}_1$ in $\Fun(I, [1])$. Alternatively, it is the unique diagram whose restriction to $I\times \{0\}$ is given by $\phi$ such that each $\phi_{\coc}(i)\colon [1]\rt X$ is $p$-cocartesian over $p(\phi(i)\leq 1)$.

Dually, for a cartesian fibration $p\colon X\rt [1]$ and a functor $\psi\colon I\rt X$ one can form the \emph{cartesian transport functor} $\psi_{\cac} \colon I \times [1] \to X$ of $\psi$.
\end{construction}
\begin{example}\label{ex:cocartesian push}
Let $p\colon X\rt [1]$ be a cocartesian fibration classifying a functor $f\colon D\rt C$. Taking $\phi$ to be the fibre inclusion $i_{0} \colon D\simeq X_{0} \hookrightarrow X$, one obtains a diagram $i_{0,\coc} \colon D \times [1] \to X$. The restriction to $D\times \{1\}$ gives a functor $D\rt X_1\simeq C$ naturally equivalent to $f$ (as a consequence of \cite[Lemma 5.2.1.4]{HTT}) and for each $d\in D$, the arrow $i_{0,\coc}(d) \colon d \rightarrowtail f(d)$ is $p$-cocartesian.
\end{example}
\begin{construction}\label{con:unit-counit}
Let $p\colon X\rt [1]$ be a cartesian and cocartesian fibration classifying an adjoint pair $f\colon D\leftrightarrows C\cocolon g$. Applying Example \ref{ex:cocartesian push} and taking the cartesian transport functor of the resulting diagram $i_{0,\coc} \colon D \times [1] \to X$ yields a functor
 \[ (i_{0,\coc})_{\cac} \colon D \times [1]
    \times [1] \to X,\]
  which takes $y \in D$ to the square
 \[
    \begin{tikzcd}
      y \arrow[r, "\eta_d"] \arrow[equals]{d} & gf(y) \arrow[d, two heads] \\
      y \arrow[r, tail] & f(y).
    \end{tikzcd}
  \]
  The functor $(i_{0,\coc})_{\cac}(-,-,0)$ factors through the fibre
  $D\simeq X_{0}$, and encodes the unit transformation $\eta\colon
  \id_D\Rightarrow gf$ of the adjunction classified by $p$. The above
  square shows that for a fixed object $y$, the unit $\eta_y\colon
  y\rt gf(y)$ is obtained by taking a cocartesian arrow $y\rt f(y)$ and factoring it as a fibrewise map followed by a cartesian map.
  
Dually, starting with the cartesian transport of the fibre inclusion $C\hookrightarrow X$ and then taking the cocartesian transport gives $(i_{1,\cac})_{\coc} \colon C \times [1]
    \times [1] \to X$ whose restriction to $C\times [1]\times \{1\}$ encodes the counit transformation $\epsilon\colon fg\Rightarrow \id_C$ of the adjunction.
\end{construction}

To understand the behaviour of a $B$-parametrised right adjoint, let
us start by showing that a map in $\cocart^{\lax}(B)$ can roughly be
viewed as a lax natural transformation; this picture will be made more
precise in Section \ref{sec:laxgray}. 
\begin{construction}\label{con:lax natural}
Let $g\colon C\rt D$ be a morphism in $\cocart^{\lax}(B)$ and
$\beta\colon b\rt b'$ one in $B$. Then $g$ determines a natural transformation of the form
\begin{equation}\label{eq:lax natural transformation}\begin{tikzcd}[column sep=3pc]
C(b)\arrow[d, "\beta_!"{left}]\arrow[r, "g"] & D(b)\arrow[d, "\beta_!"]\arrow[ld, Rightarrow, shorten=1ex, "\rho_\beta"]\\
C(b')\arrow[r, "g"{below}] & D(b').
\end{tikzcd}\end{equation}
We will refer to this as the \emph{$\beta$-component} of $g$.
To see this, note that for each object $x\in C(b)$, the image of the cocartesian lift $\tilde{\beta}\colon x\rt \beta_!x$ under $g$ factors uniquely as
$$\begin{tikzcd}
g(\tilde{\beta})\colon g(x)\arrow[r, rightarrowtail] & \beta_!g(x)\arrow[r, "\rho_\beta(x)"] & g(\beta_!x).
\end{tikzcd}$$
Alternatively, Example \ref{ex:interpol} shows that $\rho_\beta(x)$ can also be obtained as the interpolating edge associated to $x$ in the \lortho{} $p\colon X\rt [1]^{\op}\times B$ classifying $g$. 

To organise these interpolating morphisms $\rho_\beta(x)$ into a natural transformation, one can use a similar maneuver as in Construction \ref{con:cocartesian push} and consider the diagram
\[\begin{tikzcd} \Fun(C(b),C) \ar[rr,"g_*"] \ar[rd]&& \Fun(C(b),D)\ar[ld] \\
& \Fun(C(b),B).\end{tikzcd}\]
whose vertical maps are cocartesian fibrations. Applying the previous
construction to the map $\beta\colon
\mm{const}_b\Rightarrow\mm{const}_{b'}$ in the base and the fibre
inclusion $C(b)\hookrightarrow C$ covering its domain $\mm{const}_b$,
we obtain the desired natural transformation $\rho_\beta\colon
\beta_!g \Rightarrow g\beta_!$. This restricts to the interpolating
maps $\rho_\beta(x)$ defined above because $g_*$-cocartesian arrows are given pointwise by $g$-cocartesian arrows (see \cite[Proposition 3.1.2.1]{HTT}).
\end{construction}

When $g\colon C\rt D$ is a $B$-parametrised right adjoint, the lax commuting square \eqref{eq:lax natural transformation} gives rise to a natural transformation between the fibrewise left adjoints and the change-of-fibre functors $\beta_!$:
\begin{definition}\label{def:BC}
Consider a lax commuting square of the form \eqref{eq:lax natural transformation} such that the horizontal functors are part of adjunctions $f\colon D(b)\leftrightarrows C(b)\colon g$ and $f\colon D(b')\leftrightarrows C(b')\colon g$. Then the \emph{Beck--Chevalley transformation} associated to $\rho_\beta$ is the composition
$$\begin{tikzcd}
f\beta_!\arrow[r, Rightarrow, shorten=-1ex, "f\beta_!\eta"] & f\beta_!gf\arrow[r, Rightarrow, shorten=-1ex, "f\rho_\beta f"] & fg\beta_!f\arrow[r, Rightarrow, shorten=-1ex, "\epsilon \beta_!f"] & \beta_!f.
\end{tikzcd}$$
\end{definition}

We are now ready to describe the effect of the equivalence $\Adj$ from \cref{thm:dualizing lax natural transformations} on morphisms. To this end, let $g\colon C\rt D$ be a morphism in $\cocart^{\lax,\mm{R}}(B)$ and let $f=\Adj(g)$ be the induced morphism in $\cart^{\oplax, \mm{L}}(B^{\op})^\op$. Construction \ref{con:lax natural} and the dual analysis for maps in $\cart^{\oplax}(B)$ show that for each $\beta\colon b\rt b'$ in $B$, the maps $g$ and $f$ give rise to lax commuting squares of the form
$$\begin{tikzcd}[column sep=3pc]
C(b)\arrow[d, "\beta_!"{left}]\arrow[r, "g"] & D(b)\arrow[d, "\beta_!"]\arrow[ld, Rightarrow, shorten=1ex, "\rho_\beta"] & & C(b)\arrow[d, "{(\beta^{\op})^*}"{left}] & D(b)\arrow[d, "{(\beta^{\op})^*}"]\arrow[l, "f"{above}]\\
C(b')\arrow[r, "g"{below}] & D(b') & & C(b') & D(b').\arrow[l, "f"{below}]\arrow[lu, Rightarrow, shorten=1ex, "\lambda_\beta"{swap}]
\end{tikzcd}$$
Note that in these diagrams, the vertical change-of-fibre functors are equivalent. The transformation $\lambda_\beta$ is given by the Beck--Chevalley transformation associated to $\rho_\beta$, more precisely:

\begin{propn}\label{beck-chevalley}
Let $g\colon C\rt D$ be a map in $\cocart^{\lax, R}(B)$, and $\beta\colon b \rightarrow b'$ a morphism in $B$, so that the $\beta$-component of $g$ is given by \eqref{eq:lax natural transformation}. Then regarding $\beta$ as a morphism in $B^\op$, the $\beta$-component of $f = \Adj(g)$ is given by the Beck--Chevalley transformation associated to $\rho_\beta$, i.e.\ $\lambda_\beta$ is equivalent to the composition
$$\begin{tikzcd}
f\beta_!\arrow[r, Rightarrow, shorten=-1ex, "f\beta_!\eta"] & f\beta_!gf\arrow[r, Rightarrow, shorten=-1ex, "f\rho_\beta f"] & fg\beta_!f\arrow[r, Rightarrow, shorten=-1ex, "\epsilon \beta_!f"] & \beta_!f.
\end{tikzcd}$$
\end{propn}

\begin{proof}
The equivalence $\Adj\colon \cocart^{\lax, \mm{R}}(B)\rt \cart^{\oplax, \mm{L}}(B)$ is given at the level of morphisms by \eqref{eq:mate correspondence} for $A=[1]^{\op}$. In other words, consider a map $g\colon C\rt D$ in $\cocart^{\lax, \mm{R}}(B)$ and let $p=(p_1, p_2)\colon X\rt [1]\times B$ be the corresponding \lortho{}, as in Lemma \ref{lem:lortho with adj}. Then the map $f=\Adj(g)\colon D\rt C$ in $\cart(B^{\op})$ is the straightening of the \lortho{} which is \emph{dual} (relative to $B$) to the Gray fibration $\ol{p}=(p_2, p_1)\colon X\rt B\times [1]$. 

Let us now fix $\beta\colon b\rt b'$ in $B$. For $x\in C(b)\simeq X_{1, b}$, Construction \ref{con:lax natural} and Example \ref{ex:interpol} identify $\rho_\beta(x)$ with the corresponding $p$-interpolating morphism $\beta_!g(x)\rt g\beta_!(x)$ in $X$, \emph{where $p$ is considered as a \lortho{}}. On the other hand, for $y\in D(b)\simeq X_{0, b}$, (the dual of) Construction \ref{con:lax natural}, \cref{def:interpol gray} and \cref{lem:interpol dual} show that $\lambda_\beta(y)\colon f\beta_!y\rt \beta_!fy$ is given by the associated $\ol{p}$-interpolating morphism, \emph{where $\ol{p}$ is considered as a Gray fibration}.

To relate $\lambda_\beta$ and $\rho_\beta$, take $y\in D(b)$ and consider the following diagram in $X$
\begin{equation}\label{diag:unit}\begin{tikzcd}[row sep=1.2pc]
y\arrow[dd, rightarrowtail]\arrow[r, "\eta"] & gf(y)\arrow[d, rightarrowtail]\arrow[r, two heads] &  f(y)\arrow[dd, rightarrowtail]\\
& \beta_!gf(y)\arrow[d, "\rho_\beta f"] & \\
\beta_!(y)\arrow[r]\arrow[ru, "\beta_!\eta"] & g\beta_!f(y)\arrow[r,two heads] & \beta_!f(y),
\end{tikzcd}\end{equation}
which we build in steps as follows; first we obtain the outer square as the essentially unique one in which the two vertical arrows are
$p$-cocartesian (as always denoted $\rightarrowtail$) and the map
$y\rt f(y)$ is locally $p$-cocartesian. We then factor the two
horizontal maps into a fibrewise map, followed by a $p$-cartesian
morphism (denoted $\twoheadrightarrow$). Finally, we factor the
induced map $gf(y)\rt g\beta_!f(y)$ into a cocartesian map followed by
a fibrewise one. Note that the right rectangle is then precisely the
$Q$-diagram exhibiting
$\rho_\beta(f(y))\colon \beta_!gf(y)\rt g\beta_!f(y)$ as a
$p$-interpolating edge (\cref{ex:interpol}).

The resulting morphism $y\rt gf(y)$ in the top row is the unit of the adjoint pair $(f, g)$ (at the fibre over $b$) and the map $\beta_!\eta$ exists since the left vertical map is $p$-cocartesian.
Now note that the maps $\beta_!\eta$ and $\rho_\beta f$ are both contained in the fibre $X_{0, b'}\simeq D(b')$, so that choosing locally $p$-cocartesian lifts over $(0, b')\rightarrow (1, b')$ yields a commuting diagram
$$\begin{tikzcd}
\beta_!(y)\arrow[r, "\beta_!\eta"]\arrow[d, rightarrowtail, "\circ" marking] & \beta_!gf(y)\arrow[r, "\rho_\beta f"] \arrow[d, rightarrowtail, "\circ" marking] & g\beta_!f(y) \arrow[d, rightarrowtail, "\circ" marking]\arrow[rd, two heads]\\
f\beta_!(y)\arrow[r, "f\beta_!\eta"{below}] & f\beta_!gf(y)\arrow[r, "f\rho_\beta f"{below}] & fg\beta_!f(y)\arrow[r, "\epsilon"{below}] & \beta_!f(y).
\end{tikzcd}$$
Here the top lives in $X_{0, b'}$ and the bottom in $X_{1, b'}$. Pasting this diagram below \eqref{diag:unit}, the resulting outer diagram determines a $Q'$-diagram in $X$ (\cref{def:interpol gray}) that exhibits the bottom composite $f\beta_!(y)\rt \beta_!f(y)$ as the $\ol{p}$-interpolating arrow associated to $y$. In particular, the bottom composite is equivalent to $\lambda_\beta(y)$.

To obtain the identification as natural transformations, we proceed as
in \cref{con:lax natural}, replacing $g\colon C\rt D$ by $g\colon
\Fun(D(b), C)\rt \Fun(D(b), D)$ and applying the above argument to the
case where $y$ is replaced by the fibre inclusion $\iota\colon
D(b)\hookrightarrow D$, viewed as an object of $\Fun(D(b), D)$. Since the equivalences in \cref{cor:fibration model} and \cref{lem:dualisation of bicart} commute with
taking functor categories (since by adjunction they commute with
products), it follows that $\lambda_\beta$ is naturally equivalent to the Beck--Chevalley transformation associated to $\rho_\beta$, as asserted.
\end{proof}

\subsection{Parametrised correspondences}\label{subsec:identifications}

Our goal in this section is to derive a characterisation of
parametrised adjoints, which we defined in terms of their associated fibrations in \S\ref{subsec:paraadj}, by means of a natural equivalence analogous to the usual equivalence 
\[ \Map_{C}(f(d),c) \simeq \Map_{D}(d, g(c))\]
on mapping $\infty$-groupoids associated to an adjunction $f\dashv g$.

To motivate the form this will take, let us first observe that we can phrase the preceding condition in terms of
left fibrations: $f$ is left adjoint to $g$ if there is an equivalence
\begin{equation}
  \label{eq:adjlfibeq}
(f^{\op} \times \id_{D})^{*}\TwL(D) \simeq (\id_{C} \times g)^{*}\TwL(C)
\end{equation}
of left fibrations over $C^{\op} \times D$, since the twisted arrow
\icats{} are the left fibrations for the mapping \igpd{} functors, and
precomposition corresponds to pullback of left fibrations. We will prove
a parametrised analogue of \cref{eq:adjlfibeq}; to state this we first
need some notation:
\begin{notation}
To simplify a number of formulae we use $(-)^\vee$ to denote the
cocartesian fibration dual to a cartesian fibration in this
subsection.

Now suppose $p \colon E \to B$ is a cocartesian fibration, corresponding
  to a functor $F \colon B \to \Cat$. The natural transformation
  \[ \TwL(F(-)) \to F(-)^{\op} \times F(-) \]
  then corresponds to a commutative triangle
  \[
    \begin{tikzcd}
      \TwLB(E) \arrow{rr} \arrow{dr} & & (E^{\op})^\vee \times_{B} E
      \arrow{dl} \\
       & B,
    \end{tikzcd}
  \]
  since $(E^{\op})^\vee \to B$ is the cocartesian fibration classified by
  $F(-)^{\op}$. Here $\TwL_{B}(E) \to (E^{\op})^\vee \times_{B} E$ is a
  left fibration by the dual of \cite[Proposition 2.4.2.11]{HTT} and the observation that a locally cocartesian fibration with $\infty$-groupoid fibres is automatically a left fibration.
\end{notation}

\begin{thm}\label{thm:paradjlfibeq}
  If $g \colon C \to D$ is a $B$-parametrised right adjoint,
  with parametrised left adjoint $f \colon D^{\vee} \to C^{\vee}$, then there is an
  equivalence
  \[ (f^{\op} \times_{B} \id)^{*} \TwLB(C) \simeq (\id \times_{B} g)^{*}
    \TwLB(D) \]
  of left fibrations over $ (D^{\op})^\vee \times_{B} C$.
\end{thm}

Before we embark upon the proof of \cref{thm:paradjlfibeq}, left us
first observe that if $E \to [1]$
is the bicartesian fibration corresponding to an
adjunction, then we can also phrase \cref{eq:adjlfibeq} in terms of the
\emph{correspondence} associated to this functor, in the following
sense:

\begin{defn}
  A \emph{correspondence} is a left fibration $X \to A^{\op} \times
  B$. We define the \icat{} $\Corr$ of correspondences by the pullback
  \[
    \begin{tikzcd}[column sep=3.5pc]
      \Corr \arrow{r} \arrow{d} & \LFib \arrow{d}{\ev_{1}} \\
      \Cat \times \Cat \arrow{r}{(-)^{\op} \times (-)} & \Cat,
    \end{tikzcd}
  \]
  where here $\LFib$ denotes the full subcategory of $\Ar(\Cat)$
  spanned by the left fibrations.
\end{defn}

We use the following result from \cite{Stevenson}, see also \cite{AyalaF}: 

\begin{thm}[Stevenson]\label{thm:stevensoncorr}
  There is an equivalence
  \[ \corr \colon \Cat/[1] \isoto \Corr,\]
  over $\Cat \times \Cat$, where the functor $\Cat/[1] \to \Cat \times
  \Cat$ is given by taking fibres over $0$ and $1$.
  The value of the functor $\corr$ on $f\colon E\rightarrow [1]$ is defined by the natural pullback square
  \[
    \begin{tikzcd}
      \corr(E)  \arrow{d} \arrow{r} & \TwL(E) \arrow{d} \\
      E_{0}^{\op} \times E_{1} \arrow{r} & E^{\op}\times E,
    \end{tikzcd}
  \]
  where $E_0$ and $E_1$ are the fibres of $f$ over $0$ and $1$ respectively. \qed
\end{thm}
It is easy to check that if $E \to [1]$ is in fact a cocartesian fibration,
corresponding to a functor $f \colon E_{0}\to E_{1}$, then there is a
pullback square
  \[
    \begin{tikzcd}[column sep=3pc]
      \corr(E) \arrow{r} \arrow{d}  & \TwL(E_{1}) \arrow{d} \\
      E_{0}^{\op} \times E_{1} \arrow{r}{f^{\op} \times \id} & E_{1}^{\op}\times E_{1},
    \end{tikzcd}
  \]
while if it is a cartesian fibration, corresponding to $g \colon E_{1}
\to E_{0}$, then we have a pullback
\[
  \begin{tikzcd}
    \corr(E) \arrow{d}  \arrow{r} & \TwL(E_{0}) \arrow{d} \\
    E_{0}^{\op} \times E_{1} \arrow{r}{\id \times g} & E_{0}^{\op}\times E_{0}.
  \end{tikzcd}
\]
Combining these squares we get the equivalence \cref{eq:adjlfibeq}
when $E \to [1]$ corresponds to an adjunction. We now want to develop a
parametrised version of this story.

\begin{defn}
  A
  \emph{$B$-parametrised correspondence} is a left fibration $X \to
  (E_0^{\op})^\vee \times_{B} E_{1}$ for cocartesian fibrations
  $E_{0},E_{1} \to B$. We define the \icat{} $\Corr(B)$ thereof by
  the pullback square
  \[
    \begin{tikzcd}[column sep=2cm]
      \Corr(B) \arrow{r} \arrow{d} & \LFib \arrow{d}{\ev_{1}} \\
      \Cocart(B) \times \Cocart(B) \arrow{r}{(-^{\op})^\vee \times_{B}
        (-)} & \Cat.
    \end{tikzcd}
  \]

  Using \cite[Proposition 2.4.2.11]{HTT} once again, we find that a $B$-parametrised correspondence is equivalently given by the data of two cocartesian fibrations $E_0\rightarrow B$ , $E_1\rightarrow B$ and a commutative triangle
  \[
  \begin{tikzcd}
  X \arrow[rr, "f"]\arrow[rd, "p"'] &	&	(E_0^{\op})^\vee \times_{B} E_{1} \arrow[ld,"q"] \\
  	&	B	& 
  \end{tikzcd}
  \]between cocartesian fibrations, such that $f$ preserves cocartesian edges and $f_a$ is a left fibration for every $b\in B$. Straightening this data in the base $B$, we get that $\Corr(B)$ is equivalently given by the following pullback:
\begin{center}
	\begin{tikzcd}[column sep=2cm]
		\Corr(B) \arrow{r} \arrow{d} & \Fun(B,\LFib) \arrow[d, "\ev_{1}^*"] \\
		\Fun(B,\Cat) \times \Fun(B,\Cat) \arrow[r, "(-)^\op \times (-)"] & \Fun(B,\Cat).
	\end{tikzcd}
\end{center}
Because $\Fun(B,-)$ preserves pullbacks, this implies that $\Corr(B)$ is equivalent to $\Fun(B,\Corr)$.
\end{defn}

\begin{cor}
  There is an equivalence
  \[ \corr_{B} \colon \RCocart([1],B) \isoto \Corr(B) \]
  over $\Cocart(B) \times \Cocart(B)$, where the functor
  $\RCocart([1],B) \to \Cocart(B) \times \Cocart(B)$ is given by
  taking fibres over $0$ and $1$.
  Given $f\colon E\rightarrow [1]\times B$ in $\RCocart([1],B)$,
  its value $\corr_{B}(E)$ is defined by the natural pullback square
  \[
    \begin{tikzcd}
      \corr_{B}(E)  \arrow{d} \arrow{r} & \TwLB(E) \arrow{d} \\
     (E_0^{\op})^\vee \times_{B} E_{1} \arrow{r} & (E^{\op})^\vee \times_{B} E,
    \end{tikzcd}
  \]
  where $(E^\op)^{\vee} \in \RCocart([1],B)$ is obtained by dualising in the second variable.
\end{cor}

\begin{proof}
Unstraightening over $B$ we have the square
  \[
    \begin{tikzcd}[column sep=2cm]
      \RCocart([1],B) \arrow{r}{\corr_{B}} \arrow{d}{\sim} & \Corr(B)
      \arrow{d}{\sim} \\
      \Fun(B, \Cat/[1]) \arrow{r}{\Fun(B,\corr)} & \Fun(B, \Corr),
    \end{tikzcd}
    \]
so the claim follows from \cref{thm:stevensoncorr}.
\end{proof}

\begin{propn}\label{propn:crvorthpblfib}
  Suppose $p \colon E \to [1] \times B$ is a curved orthofibration,
  corresponding to a functor $g \colon E_{1} \to E_{0}$ over $B$. Then
  there is a pullback square
  \[
    \begin{tikzcd}[column sep=3pc]
      \corr_{B}(E) \arrow{d} \arrow{r} & \TwLB(E_{0}) \arrow{d} \\
     (E_0^{\op})^\vee \times_{B}E_{1} \arrow{r}{\id \times_{B} g} & (E_0^{\op})^\vee
      \times_{B}E_{0}.
    \end{tikzcd}
  \]
\end{propn}

In order to prove this we first make some fibrational observations:
\begin{propn}\label{prop:fibrecart}
  Consider a commutative square of \icats{}
  \[
    \begin{tikzcd}
      E \arrow{r}{g} \arrow[d, "p"{swap}] & F \arrow{d}{q} \\
      X \arrow{r}{f} & Y,
    \end{tikzcd}
  \]
  where $p$ and $q$ are cocartesian fibrations and $g$ takes
  $p$-cocartesian morphisms to $q$-cocartesian ones. For $x \in X$,
  let $g_{x} \colon E_{x} \to F_{fx}$ be the restriction of $g$ to the
  fibres over $x$. If a morphism $\phi \colon e' \to e$ in $E_{x}$ is
  $g_{x}$-cartesian, then $\phi$ is also $g$-cartesian.
\end{propn}
\begin{proof}
  For $e'' \in E$ over $x'' \in X$, we have the commutative diagram
  \[
    \begin{tikzcd}
      \Map_{E}(e'', e') \arrow{rr} \arrow{dr} \arrow{dd} & &
      \Map_{E}(e'',e) \arrow{dl}
      \arrow{dd} \\
      & \Map_{X}(x'',x)  \\
      \Map_{F}(ge'',ge') \arrow{rr} \arrow{dr} & & \Map_{F}(ge'',ge) \arrow{dl} \\
       & \Map_{Y}(fx'',fx), \arrow[uu, crossing over, leftarrow]
    \end{tikzcd}
  \]
  where we want to show that the back square is cartesian. It suffices
  to check we have a cartesian square on the fibres over any $\xi \in
  \Map_{X}(x'',x)$, but since $p$ and $q$ are cocartesian fibrations
  and $g$ preserves cocartesian morphisms, we can identify this square
  as
  \[
    \begin{tikzcd}
      \Map_{E_{x}}(\xi_{!}e'', e') \arrow{r} \arrow{d} & \Map_{E_{x}}(\xi_{!}e'',e) \arrow{d} \\
      \Map_{F_{fx}}(f(\xi)_{!}ge'',ge')  \arrow{r} &  \Map_{F_{fx}}(f(\xi)_{!}ge'',ge),
    \end{tikzcd}
  \]
  which is cartesian by the assumption that $\phi$ is $g_{x}$-cartesian.
\end{proof}

\begin{cor}\label{cor:TwLcart}
  For any functor $p \colon E \to B$, a morphism from $x' \rightarrow y'$ to $x \rightarrow y$ in $\TwL(E)$ of the form
      \[
    \begin{tikzcd}
      x' \arrow{d} & x \arrow{d} \arrow{l} \\
      y' \arrow{r} & y
    \end{tikzcd}
  \]
is $\TwL(p)$-cartesian if $x \rightarrow x'$ is $p$-cocartesian and $y' \rightarrow y$ is $p$-cartesian.
%
%
In particular, if $p \colon E \to B$ is a
  cartesian fibration, then $\TwL(E) \to
  \TwL(B)$ has cartesian lifts of morphisms of the form
  \[
    \begin{tikzcd}
      a \arrow{d} \arrow[equals]{r} & a \arrow{d} \\
      b \arrow{r}{g} & b'
    \end{tikzcd}
  \]
  in $\TwL(B)$.
\end{cor}
\begin{proof}
To find that a morphism in $\TwL(E)$, in which $x \rightarrow x'$ is an equivalence and $y' \rightarrow y$ is $p$-cartesian, is $\TwL(p)$-cartesian, apply \cref{prop:fibrecart} to the square
  \[
    \begin{tikzcd}
      \TwL(E) \arrow{r}{\TwL(p)} \arrow{d} & \TwL(B) \arrow{d} \\
      E^{\op} \arrow{r}{p^{\op}} & B^{\op},
    \end{tikzcd}
  \]
noting that on fibres over $x \in E$ with $b = p(x)$ we have the
  functor $E_{x/} \to B_{b/}$, where a morphism of the specified form
  is cartesian. Dually, applying \cref{prop:fibrecart} to the square
  \[
    \begin{tikzcd}
      \TwL(E) \arrow{r}{\TwL(p)} \arrow{d} & \TwL(B) \arrow{d} \\
      E \arrow{r}{p} & B,
    \end{tikzcd}
  \]
we find that a morphism in $\TwL(E)$ with $y' \rightarrow y$ an equivalence and $x \rightarrow x'$ $p$-cocartesian is $\TwL(p)$-cartesian. Since cartesian morphisms are closed under composition, the general case follows.
\end{proof}


Let $p\colon E\rightarrow A\times B$ be a curved orthofibration;
recall from \cref{prop:lortho reformulations} that $p$ can then be
interpreted as a map of cocartesian fibrations over $B$. Applying
$\TwLB(-)$ to it gives the diagram
\[
\begin{tikzcd}
	\TwLB(E) \arrow[d] \arrow[r]               & \TwL(A)\times B \arrow[d]       \\
	(E^\mathrm{op})^\vee\times_B E \arrow[r] & A^\mathrm{op}\times A \times B,
\end{tikzcd}
\]
where we use that $\TwLB(A \times B) \simeq \TwL(A)\times B$. Applying
\cref{cor:TwLcart} fibrewise and appealing to \cref{prop:fibrecart}
again, we get:
\begin{cor}
  Suppose $p \colon E\to A \times B$ is a curved orthofibration. Then
  \[ \TwLB(E) \to \TwL(A) \]
  has cartesian lifts of morphisms in $\TwL(A)$ of the form
  \[
    \begin{tikzcd}
      a \arrow{d} \arrow[equals]{r} & a \arrow{d} \\
      a'' \arrow{r}{g} & a',
    \end{tikzcd}
  \]
  given by the cartesian morphisms for $\TwL(E_{b}) \to \TwL(A)$
  described above.\qed
\end{cor}
\begin{observation}\label{obs:twlortho}
  In particular, for any $a \in A$ the projection
  \[  \TwLB(E) \times_{\TwL(A)} A_{a/} \to A_{a/} \]
  is a cartesian fibration, and $\TwLB(E) \times_{\TwL(A)} A_{a/} \to A_{a/} \times B$ is
  a curved orthofibration.
\end{observation}

\begin{notation}
  Given $E \to [1] \times B$ in $\RCocart([1],B)$, we have pullback
  squares
  \[
    \begin{tikzcd}
    \TwLB(E)|_{0} \arrow{r}\arrow{d} & \TwLB(E) \arrow{d} \\
   (E_0^{\op})^\vee\times_{B}E \arrow{r} & (E^{\op})^\vee \times_{B} E,
  \end{tikzcd} \quad
    \begin{tikzcd}
    \TwLB(E)|_{1} \arrow{r}\arrow{d} & \TwLB(E) \arrow{d} \\
    (E^{\op})^\vee\times_{B}E_{1} \arrow{r} & (E^{\op})^\vee \times_{B} E.
  \end{tikzcd}
  \]
\end{notation}

\begin{proof}[Proof of \cref{propn:crvorthpblfib}]
  Applying \cref{obs:twlortho} to $E \to [1] \times B$ and $0 \in
  [1]$, we see that $\TwLB(E)|_{0} \to [1]_{0/} \cong [1]$ is a cartesian
  fibration. Moreover, from the description of the cartesian morphisms
  we see that
  \[ \TwLB(E)|_{0}  \to (E_0^{\op})^\vee \times_{B} E \]
  is a morphism between cartesian fibrations to $[1]$ that preserves
  cartesian morphisms. 
  
  Taking fibres for $\TwLB(E)|_{0} \to  (E_0^{\op})^\vee \times_{B} E$
  over $0$ and $1$, we get the following diagram where both squares
  are cartesian
  \[
    \begin{tikzcd}
      \TwLB(E_{0}) \arrow{r} \arrow{d} &\TwLB(E)|_{0} \arrow{d}
      & \corr_{B}(E) \arrow{d}  \arrow{l}\\
      (E_0^{\op})^\vee\times_{B}E_{0} \arrow{r} &
      (E_0^{\op})^\vee\times_{B}E & (E_0^{\op})^\vee\times_{B}E_{1} \arrow{l}.
    \end{tikzcd}
  \]
  Taking the cartesian transport over $[0] \to [1]$ we therefore get a
  commutative square
  \[
    \begin{tikzcd}[column sep=3pc]
      \corr_{B}(E) \arrow{r}  \arrow{d}  & \TwLB(E_{0}) \arrow{d} \\
      (E_0^{\op})^\vee\times_{B}E_{1} \arrow{r}{\id \times_{B} g} &
     (E_0^{\op})^\vee\times_{B}E_{0}.
    \end{tikzcd}
  \]
  It remains to show that this square is cartesian, which we can check
  on fibres since the vertical maps are both left fibrations. To do
  this we can first restrict to fibres over $b \in B$, where we have
  the square
  \[
    \begin{tikzcd}[column sep=3pc]
      \corr(E_{b}) \arrow{r}  \arrow{d}  & \TwL(E_{0,b}) \arrow{d} \\
      E_{0,b}^{\op}\times E_{1,b} \arrow{r}{\id \times g_{b}} &
      E_{0,b}^{\op}\times E_{0,b},
    \end{tikzcd}
  \]
  which is cartesian because $g_b$ is a right adjoint.
\end{proof}

\begin{observation}
  If $p \colon E \to [1] \times B$ is in $\RCocart([1],B)$, then we
  also have $(p^{\op})^\vee \colon (E^{\op})^\vee \to [1]^{\op} \times B$
  in $\RCocart([1]^{\op},B)$. Since there is a natural equivalence
  $\TwL(C) \simeq \TwL(C^{\op})$ over the permutation $C^{\op} \times
  C \simeq C \times C^{\op}$, we get for a cocartesian fibration $X
  \to B$ a natural equivalence
  \[ \TwLB(X) \simeq \TwLB( (X^{\op})^\vee) \]
  over $(X^{\op})^\vee \times_{B} X \simeq X \times_{B}(X^{\op})^\vee$,
  and hence also
  \[ \corr_{B}(E) \simeq \corr_{B}( (E^{\op})^\vee) \]
  over $(E_0^{\op})^\vee \times_{B}E_{1} \simeq
  E_{1}\times_{B}(E_0^{\op})^\vee$.
\end{observation}

 Combining this with
  \cref{propn:crvorthpblfib}, we get the following dual version thereof:

\begin{cor}\label{cor:dualcrvorthopblfib}
  Suppose $p \colon E \to [1] \times B$ is in
  $\RCocart([1],B)$. If $ (p^{\op})^\vee \colon (E^{\op})^\vee \to
  [1]^{\op} \times B$ is a curved orthofibration, corresponding to a
  functor $f \colon E_{0}^{\vee}\to E_{1}^{\vee}$ over $B^{\op}$, then
    there is a pullback square
  \[
    \begin{tikzcd}[column sep=3pc]
      \corr_{B}(E) \arrow{d} \arrow{r} & \TwLB(E_{1}) \arrow{d} \\
      (E_0^{\op})^\vee \times_{B}E_{1} \arrow{r}{f^{\op} \times_{B} \id} & (E_1^{\op})^\vee
      \times_{B}E_{1}.
    \end{tikzcd}
  \]
\end{cor}

\begin{proof}[Proof of \cref{thm:paradjlfibeq}]
  Suppose $g$ corresponds to the curved orthofibration $p \colon E \to
  [1] \times B$. Then $p^{\vee} \colon E^{\vee} \to [1] \times
  B^{\op}$ is a curved orthofibration over $B^{\op} \times [1]$ whose
  cocartesian unstraightening over $[1]$ gives $f$. Hence
  $ (p^{\op})^\vee \colon (E^{\op})^\vee \to [1]^{\op} \times B$ is the
  curved orthofibration for $f^{\op}$. Combining
  \cref{propn:crvorthpblfib} and \cref{cor:dualcrvorthopblfib}, we
  thus get equivalences
  \[ \begin{tikzcd} (\id \times_{B} g)^{*} \TwLB(D) \arrow[r, "\sim"] & \corr_{B}(E)  & (f^{\op} \times_{B} \id)^{*} \TwLB(C)\arrow[l, "\sim"{swap}]
  \end{tikzcd}\]
 of left fibrations over $ (D^{\op})^\vee \times_{B} C$.
\end{proof}

Specialising this to the case of projections, we get:
\begin{cor}\label{cor:identifyingadjoints}
  Let $F \colon X \times B \to Y$ be a functor such that $F(-,b)$ is a
  left adjoint for all $b \in B$, and let $G \colon Y \times B^{\op}
  \to X$ be the functor corresponding to the $B^{\op}$-parametrised
  right adjoint of $F$, regarded as a functor $X \times B \to Y \times
  B$ over $B$. Then there is an equivalence
  \[ (\id \times G)^{*} \TwL(X) \simeq (F^{\op} \times \id)^{*} \TwL(Y)\]
  of left fibrations over $X^{\op} \times Y \times B$, and hence a
  natural equivalence of mapping spaces
  \[\pushQED{\qed} \Map_{X}(x, G(y,b)) \simeq \Map_{Y}(F(x,b), y).\qedhere
  \popQED\]
\end{cor}

\subsection{Lax monoidal adjunctions}\label{subsec:laxmonadj} 
Recall that an $\infty$-operad $\OO$ is a map of $\infty$-categories
$p\colon \OO\rt \pF$ to the 1-category of pointed finite
sets, satisfying the following conditions: 
\begin{enumerate}
\item $\OO$ has all $p$-cocartesian lifts for inert morphisms in $\pF$
  (\ie{} those maps which are bijections away from the basepoint).
\item Let $x\in \OO$ be an object with $p(x)=\pset{n}$ and let $\rho^i_{x}\colon x\rt x_i$ be a $p$-cocartesian lift of the unique inert map $\rho^i\colon \pset{n}\rt \pset{1}$ which sends $i$ to $1$. For every $f\colon \pset{m}\rt \pset{n}$ and $y\in \OO_{\pset m}$, postcomposition with the $\rho^i_{x}$ induce an equivalence
$$\begin{tikzcd}
\Map_{\OO}^f\big(y, x\big)\arrow[r, "\sim"] & \prod_i \Map_{\OO}^{\rho^i\circ f}(y, x_i).
\end{tikzcd}$$
\item For every tuple $(x_1, \dots, x_n)$ of objects in $\OO_{\pset{1}}$, there exists an $x\in \OO_{\pset{n}}$ together with $p$-cocartesian lifts $\rho^i_{x}\colon x\rt x_i$. 
\end{enumerate}
A morphism in $\OO$ is called \emph{inert} if it is the cocartesian
lift of an inert map in $\pF$.

If $\OO$ is an \iopd{}, then an \emph{$\OO$-monoid} in an \icat{} $X$
with finite products is a functor $M \colon \OO \to X$ satisfying the \emph{Segal condition}: for
every $x\in \OO_{\pset{n}}$ with inert maps $\rho^i_{x} \colon x\rt x_i$,
the morphism $M(x) \to \prod_{i} M(x_{i})$ is an equivalence. An
\emph{$\OO$-monoidal \icat{}} is a cocartesian fibration corresponding
to an $\OO$-monoid in $\Cat$ (or equivalently an \iopd{} with a map to
$\OO$ that is a cocartesian fibration).

\begin{definition}
The $(\infty, 2)$-category $\OMonCat^{\lax}$ of $\OO$-monoidal $\infty$-categories and \emph{lax $\OO$-monoidal functors} between them is given by the $1$-full sub-$2$-category of $\mathbf{Cocart}^{\lax}(\OO)$ whose:
\begin{enumerate}
\item objects are $\OO$-monoidal \icats{},
\item morphisms are functors $\CC^{\otimes}\rt \DD^{\otimes}$ over $\OO$ that preserve the cocartesian morphisms lying over inert morphisms in $\OO$.
\end{enumerate}
\end{definition}
By definition, the underlying \icat{} of
$\OMonCat^{\lax}$ is the full subcategory of $\infty$-operads over
$\OO$ whose objects are the $\OO$-monoidal \icats{}. A \emph{strong}
$\OO$-monoidal functor corresponds to a morphism $\CC^{\otimes}\rt
\DD^{\otimes}$ that preserves all cocartesian edges.

\begin{example}
Let us explicitly mention the special case $\OO = \pF$, where $\OMonCat^{\lax}$ has objects symmetric monoidal $\infty$-categories, $1$-morphisms lax symmetric monoidal functors, and $2$-morphisms symmetric monoidal transformations. In particular, Theorem A from the introduction is a statement about morphism categories therein (and in the oplax analogue defined below). 
\end{example}

\begin{definition}
The $(\infty, 2)$-category $\OMonCat^{\oplax}$ of $\OO$-monoidal $\infty$-categories and \emph{oplax $\OO$-monoidal functors} between them is the sub-2-category of $\mathbf{Cart}^{\oplax}(\OO^\op)$ whose:
\begin{enumerate}
\item objects are cartesian fibrations $\CC_{\otimes} \to \OO^{\op}$ corresponding to $\OO$-monoids,
\item morphisms are functors $\CC_{\otimes}\rt \DD_{\otimes}$ that preserve cartesian morphisms lying over inert morphisms in $\OO^\op$.
\end{enumerate}
\end{definition}

Note in particular that the objects in $\OMonCat^{\oplax}$ are a
priori not $\OO$-monoidal \icats{}: one has to take the cocartesian
fibration over $\OO$ dual to a cartesian fibration over $\OO^\op$ to get an $\OO$-monoidal $\infty$-category in the usual sense. The following lemma thus simply asserts that essentially \emph{by definition}, an oplax $\OO$-monoidal functor is a lax $\OO$-monoidal functor between the opposite categories. 

\begin{lemma}
Taking opposite categories defines an equivalence of $(\infty, 2)$-categories
$$
(-)^{\op}\colon \OMonCat^{\oplax}\rt \big(\OMonCat^{\lax}\big)^{2\mm{-op}}.
$$
\end{lemma}
\begin{proof}
It suffices to verify that the equivalence of Remark \ref{rem:opposite of lax is oplax} identifies the relevant sub-2-categories. Given a cartesian fibration $\CC_\otimes\rt \OO^{\op}$, let us write ${\CC}^{\mathrm{op},\otimes}\rt \OO$ for the opposite cocartesian fibration. The Segal map 
$$
(\rho_i^*)_{i} \colon \CC_{\otimes}(x)\rt \prod_i \CC_{\otimes}(x_i)
$$
is then the \emph{opposite} of the Segal map 
$(\rho_{i,!})_{i}\colon {\CC}^{\mathrm{op},\otimes}(x)\rt \prod_i {\CC}^{\mathrm{op},\otimes}(x_i)$, so that one is an equivalence if and only if the other is. Finally, a functor preserving cartesian lifts of inert morphisms is sent to the functor between opposite categories, which preserves cocartesian lifts of inert morphisms.
\end{proof}

For example, for $\OO = (\pF)_\mathrm{inert}$, the trivial operad, $\OMonCat^{\oplax}$ is a $2$-fold Segal space model for the $(\infty, 2)$-category $\tCat$ (by the discussion in Section \ref{subsec:scunstrGray}). In this case we find, unsurprisingly, that taking opposite categories defines an equivalence $\tCat \simeq \tCat^{2\mm{-op}}$.

\begin{lemma}\label{lem:monoidal adjoint}
Let $g\colon \CC^{\otimes}\rt \DD^{\otimes}$ be a lax $\OO$-monoidal functor, i.e.\ a morphism in $\OMonCat^{\lax}$. Then the following two conditions are equivalent:
\begin{enumerate}
\item For every $x\in \OO_{\pset{1}}$, the induced map on fibres $g\colon \CC^{\otimes}(x)\rt \DD^{\otimes}(x)$ is a right adjoint.
\item For every $x\in \OO$, the induced map on fibres $g\colon \CC^{\otimes}(x)\rt \DD^{\otimes}(x)$ is a right adjoint.
\end{enumerate}
\end{lemma}
\begin{proof}
This follows from the fact that for each $x\in \OO$, there is a commuting square
$$\begin{tikzcd}
\CC^{\otimes}(x)\arrow[r, "g_x"]\arrow[d, "\sim"{right}, "(\rho^i_!)_{i}"{left}] & \DD^{\otimes}(x)\arrow[d, "\sim"{left}, "(\rho^i_!)_{i}"{right}]\\
\prod_i \CC^{\otimes}(x_i)\arrow[r, "(g_{x_i})"{below}] & \prod_i\DD^{\otimes}(x_i)
\end{tikzcd}$$
where $\rho^i\colon x\rt x_i$ are the canonical inert maps decomposing $x$ into its components $x_i \in \OO_{\pset{1}}$.
\end{proof}
\begin{definition}
A lax $\OO$-monoidal functor $g\colon \CC^{\otimes}\rt \DD^{\otimes}$ is a \emph{lax $\OO$-monoidal right adjoint} if it satisfies the equivalent conditions of Lemma \ref{lem:monoidal adjoint}. 

Likewise, an oplax $\OO$-monoidal functor $f\colon \CC_{\otimes}\rt \DD_{\otimes}$ is called an \emph{oplax $\OO$-monoidal left adjoint} if it induces left adjoint functors between the fibres over each $x\in \OO^{\op}$ (equivalently, all $x\in \OO^{\op}_{\pset{1}}$).
\end{definition}
\begin{theorem}\label{thm:monoidal adjunctions}
For each $\infty$-operad $\OO$, there is a natural equivalence of $(\infty, 2)$-categories
$$\begin{tikzcd}
\Adj\colon \OMonCat^{\lax, \mm{R}}\arrow[r, "\sim"] & \Big(\OMonCat^{\oplax, \mm{L}}\Big)^{(1, 2)\mm{-op}}
\end{tikzcd}$$
between the $1$-full sub-$2$-categories whose morphisms are lax $\OO$-monoidal right adjoints and oplax $\OO$-monoidal left adjoints. 
\end{theorem}

\begin{proof}
It suffices to show that the equivalence of Theorem \ref{thm:dualizing lax natural transformations} identifies the two relevant sub-2-categories
$$\begin{tikzcd}
\OMonCat^{\lax, \mm{R}}\arrow[d, hookrightarrow]\arrow[r, dashrightarrow, "\sim"] & \Big(\OMonCat^{\oplax, \mm{L}}\Big)^{(1, 2)\mm{-op}}\arrow[d, hookrightarrow]\\
\mathbf{Cocart}^{\lax, \mm{R}}(\OO)\arrow[r, "\Adj"{above}, "\sim"{below}] & \Big(\mathbf{Cart}^{\oplax, \mm{L}}(\OO^{\op})\Big)^{(1, 2)\mm{-op}}.
\end{tikzcd}$$
At the level of objects, note that the functor $\Adj$ sends a cocartesian fibration over $\OO$ to the cartesian fibration over $\OO^{\op}$ classifying the same functor $\OO\rt \Cat$. In particular, a cocartesian fibration over $\OO$ satisfies the Segal conditions if and only if its image under $\Adj$ does.

It remains to verify that the functor $\Adj$ sends a map $g\colon
\CC^\otimes \rt \DD^{\otimes}$ that preserves cocartesian lifts of
inert maps to a functor $F\colon \DD_\otimes\rt \CC_\otimes$ of cartesian
fibrations over $\OO^{\op}$ that preserves cartesian lifts of
inert maps (for the reverse implication, reverse the roles of $f$ and
$g$ in the next argument). By \cref{beck-chevalley}, this comes down to the following assertion: for any inert map $\beta\colon x\rt x'$ in $\OO$, the lax $\OO$-monoidal functor $g$ defines the commuting left square
$$\begin{tikzcd}
\CC^{\otimes}(x)\arrow[d, "\beta_!"{left}]\arrow[r, "g_x"] & \DD^{\otimes}(x)\arrow[d, "\beta_!"] & & \CC^{\otimes}(x)\arrow[d, "\beta_!"{left}] & \DD^{\otimes}(x)\arrow[d, "\beta_!"]\arrow[l, "f_x"{above}]\\
\CC^{\otimes}(x')\arrow[r, "g_{x'}"{below}] & \DD^{\otimes}(x') & & \CC^{\otimes}(x') & \DD^{\otimes}(x')\arrow[l, "f_{ x'}"{below}]\arrow[lu, Rightarrow, shorten=1em]
\end{tikzcd}$$
and we have to verify that the associated Beck-Chevalley transformation on the right is an equivalence. Using the Segal condition, these squares can be identified with
$$\begin{tikzcd}
\prod\limits_{i\in I}\CC^{\otimes}(x_i)\arrow[d, "\mm{pr}"{left}]\arrow[r, "(g_{x_i})"] & \prod\limits_{i\in I}\DD^{\otimes}(x_i)\arrow[d, "\mm{pr}"] & & \prod\limits_{i\in I}\CC^{\otimes}(x_i)\arrow[d, "\mm{pr}"{left}] & \prod\limits_{i\in I}\DD^{\otimes}(x_i)\arrow[d, "\mm{pr}"]\arrow[l, "(f_{x_i})"{above}]\\
\prod\limits_{j\in J}\CC^{\otimes}(x_j)\arrow[r, "(g_{x_j})"{below}] & \prod\limits_{j\in J}\DD^{\otimes}(x_j) & & \prod\limits_{j\in J}\CC^{\otimes}(x_j) & \prod\limits_{j\in J}\DD^{\otimes}(x_j)\arrow[l, "(f_{x_j})"{below}]\arrow[lu, Rightarrow, shorten=0.8em, xshift=-0.6em, yshift=0.2em]
\end{tikzcd}$$
where the vertical functors are projections associated to an inclusion of finite sets $J\subseteq I$. But for such projections, the Beck--Chevalley transformation is always an equivalence (since the unit and counit maps can be computed in each factor). 
\end{proof}

By considering morphism categories in the statement of Theorem \ref{thm:monoidal adjunctions}, we find:

\begin{corollary}\label{propaintext}
Given an $\infty$-operad $\OO$ and two $\OO$-monoidal
$\infty$-categories $\C$ and $\D$, taking adjoints
gives a canonical equivalence between the \icat{} of oplax $\OO$-monoidal left adjoint functors $\C \rightarrow \D$ and the opposite of the \icat{} of lax $\OO$-monoidal right adjoint functors.\qed
\end{corollary}

As another application of our machinery we have the following result. Recall the $\infty$-operad $\mathbb{RM}\mathrm{od}$, defined in \cite[Section 4.2.1]{HA}, which encodes the data of an $\mathbb{E}_1$-algebra equipped with a right module.

\begin{proposition}\label{prop:monoidalinternalhom}
Let $\OO$ be an $\infty$-operad and let $\C$ be an $(\mathbb E_1 \otimes \OO)$-monoidal $\infty$-category, and $D$ a right $\OO$-monoidal module over it, i.e. the pair $(\C,D)$ is equipped with the structure of an $\mathbb{RM}\mathrm{od} \otimes \OO$-algebra. Suppose further that the action of $\C$ on $D$ is colourwise closed, that is the action $- \otimes c \colon D(x)  \rightarrow D(x)$ admits a right adjoint  $[c,-] \colon D(x) \rightarrow D(x)$ for every colour $x\in O_{\langle 1\rangle}$ and $c \in \C(x)$. 

Then the mapping object functors $[-,-] \colon D(x)^\op\times D(x) \rightarrow C(x)$  admit a canonical lax $\OO$-monoidal refinement.
\end{proposition}

Observing that every $\mathbb E_1\otimes \OO$-monoidal $\infty$-category acts on itself $\OO$-monoidally (via the map $\mathbb{RM}\mathrm{od} \rightarrow \mathbb E_1$ from \cite[Remark 4.2.1.5]{HA}) and that $\mathbb E_1 \otimes \mathbb E_n \simeq \mathbb E_{n+1}$ for $1 \leq n \leq \infty$, see \cite[Theorem 5.1.2.2]{HA}, we obtain:

\begin{corollary}\label{cor::monoidalinternalhom}
Let $\OO$ be an $\infty$-operad and $\C$ an $(\mathbb E_1 \otimes \OO)$-monoidal $\infty$-category such that the $\mathbb{E}_1$-monoidal categories $\C(x)$ are right closed for every colour $x\in O_{\langle 1\rangle}$.  Then the mapping object functors $[-,-] \colon \C(x)^{\op} \times \C(x) \rightarrow \C(x)$ admit a canonical lax $\OO$-monoidal refinement. In particular if $\C$ is a closed $\mathbb E_{n+1}$-monoidal \icat{} for some $1 \leq n\leq \infty$, then $[-,-] \colon \C^\op \times \C \rightarrow \C$ carries a canonical lax $\mathbb E_n$-monoidal refinement.\qed
\end{corollary}

For $n = \infty$ such a lax symmetric monoidal refinement was first established by Hinich in \cite[Section A.5]{Hinich} using different means; his construction most certainly agrees with ours, but let us refrain from attempting a formal comparison in this paper.

\begin{proof}[Proof of Proposition \ref{prop:monoidalinternalhom}]
Let $\C^\otimes,D^\otimes,(D^\op)^\otimes \rightarrow \OO$ be the cocartesian fibrations of operads witnessing the $\OO$-monoidal structures of $\C, D$ and $D^\op$ respectively. 

We follow the same strategy as in Example \ref{ex:monoidalclosed}, and so wish to construct the morphism
\[\begin{tikzcd}
(D^\op)^\otimes \times_{\OO} \C^\otimes \ar[rd,"\pr_1"'] && \ar[ll,"{(}\pr_1{,}{[}{-}{,}{-}{]}{)}"'] (D^\op)^\otimes \times_{\OO} D^\otimes \ar[ld,"\pr_1"] \\
& (D^\op)^\otimes &
\end{tikzcd}\]
in $\MonCat_{(D^{\op})^\otimes}^{\lax, \mm{R}}$ from its counterpart in $\MonCat_{D_{\otimes}}^{\oplax, \mm{L}}$ using Theorem \ref{thm:monoidal adjunctions}. By definition of the tensor product of operads we can regard $(\C,D)$ as an $\mathbb{RM}\mathrm{od}$-algebra in $\OMonCat$, i.e. an $\mathbb E_1$-algebra in $\OMonCat$ equipped with a right module. So in particular the action
\[\otimes \colon D \times C \longrightarrow D.\]
is itself a (strongly) $\OO$-monoidal functor. Applying cartesian unstraightening, we obtain a map
\[\mu \colon \D_\otimes \times_{\OO} \C_\otimes \longrightarrow \D_\otimes, \]
with which we form
\[\begin{tikzcd}
\D_\otimes \times_{\OO} \C_\otimes \ar[rd,"\pr_1"'] \ar[rr,"(\pr_1{,}\mu)"] &&  D_\otimes \times_{\OO} D_\otimes \ar[ld,"\pr_1"] \\
& D_\otimes. &
\end{tikzcd}\]
By Lemma \ref{lem:monoidal adjoint} this indeed defines a morphism in $\MonCat_{D_\otimes}^{\oplax, \mm{L}}$, which dualises as desired by \cref{cor:identifyingadjoints}.
\end{proof}

\section{Parametrised units and counits}\label{sec:paraunit}

Consider two symmetric monoidal \icats{} $\CC^\otimes$ and $\DD^\otimes$ and  a lax symmetric monoidal right adjoint $g\colon\CC^{\otimes} \rightarrow \DD^{\otimes}$, with left adjoint $f$. Given any finite collection of objects $\{y_i\}$ in $\DD$, we have a canonical comparison map $\bigotimes y_i \rightarrow g\big(\bigotimes f(y_i)\big)$ given by
\[\begin{tikzcd}\bigotimes y_i \arrow[r, "\eta"] & gf\big(\bigotimes y_i\big) \arrow[r, "g(\mu)"] & g\big(\bigotimes f(y_i)\big),\end{tikzcd}\] 
where $\eta$ is the unit of the adjunction $f\dashv g$ and $\mu$ is given by the oplax monoidal structure of $f$. This is the prototypical example of the parametrised unit morphism that we will consider in this section. 

The goal of \S\ref{subsec:paracounit1} is to make explicit the functoriality of these maps, and in \S\ref{subsec:paraunits2} we similarly produce a functor extracting adjoint morphisms in a parametrised adjunction.

\subsection{Parametrised (co)units}\label{subsec:paracounit1}

Let us consider a parametrised left adjoint $f$ over $B$ with parametrised right adjoint $g$ from \cref{thm:dualizing lax natural transformations}
\[
  \begin{tikzcd}
    D \arrow{dr} \arrow{rr}{f} & & C
    \arrow{dl} \\
    & B
  \end{tikzcd}
  \qquad\qquad
  \begin{tikzcd}
C^\vee \arrow{dr} \arrow{rr}{g} & & D^\vee
\arrow{dl} \\
& B^{\op};
\end{tikzcd}
\]
here and in the remainder of this section we will again use $(-)^\vee$ to denote the cocartesian fibration dual to a cartesian fibration to ease notation.

Given any edge $\beta\colon b\rightarrow b'$ in $B$ and any $y\in D_{b'}$, the natural map $\lambda_\beta(y)\colon f_b\beta^*(y)\rt \beta^*f_{b'}(y)$ (dual to Construction \ref{con:lax natural}) is adjoint to a map
$$\begin{tikzcd}[column sep=3pc]
\eta_\beta(y)\colon \beta^*(y)\arrow[r] & g_bf_b\beta^*(y)\arrow[r, "g_b\lambda_\beta(y)"] & g_b\beta^*f_{b'}(y).
\end{tikzcd}$$
One can also obtain $\eta_{\beta}(y)$ from the Beck--Chevalley transformation (\cref{def:BC}) as
$$\begin{tikzcd}[column sep=3pc]
\eta_{\beta}(y)\colon \beta^*(y)\arrow[r] & \beta^*g_{b'}f_{b'}(y)\arrow[r, "\rho_\beta(f_{b'}(y))"] & g_b\beta^*f_{b'}(y).
\end{tikzcd}$$
The goal of this section is to describe the fuctoriality of this unit morphism $\eta_\beta(y)$ in $\beta$ and $y$. To motivate the functoriality in $\beta$, let us consider the following:

\begin{example}\label{ex:unit conjugation}
Let $\beta\colon b\rt b'$, $\gamma\colon b'\rt b''$ and $y\in D_{b''}$. Then we claim that 
$$
\eta_{\gamma\beta}(y)\simeq g_b\beta^*(\lambda_\gamma(y))\circ \eta_\beta(\gamma^*(y))
$$
and dually that
$$\eta_{\beta\alpha}(z)\simeq \rho_\alpha(\beta^*f_{b''}(z))\circ \alpha^*\eta_{\beta}(z)$$ for any $\alpha \colon a\rightarrow b$ in $B$ and $z\in D_{b'}$.
To see the first identification consider the diagram
$$\begin{tikzcd}[column sep=2.5pc]
f_b\beta^*\gamma^*(y)\arrow[r, "{\lambda_{\beta}(\gamma^*(y))}"] & \beta^*f_{b'}\gamma^*(y)\arrow[r, two heads]\arrow[d, "\beta^*\lambda_{\gamma}(y)"{swap}] & f_{b'}\gamma^*(y)\arrow[d, "{\lambda_\gamma(y)}"]\\
& \beta^*\gamma^*f_{b''}(y)\arrow[r, two heads] & \gamma^*f_{b''}(y)\arrow[r, two heads] & f_{b''}(y)
\end{tikzcd}$$
in $C$. The top row factors the image under $f$ of the cartesian arrow $\beta^* \gamma^*(y)\rt \gamma^*(y)$ into $\lambda_{\beta}(\gamma^*(y))$, followed by a cartesian morphism (cf.\ Construction \ref{con:lax natural}). Notice that the total composite along the top is the image under $f$ of the cartesian arrow $\beta \gamma^*(y)\rt y$, and that following the bottom gives a factorisation of this into a fibrewise followed by a cartesian morphism. Therefore we conclude that $\lambda_{\gamma\beta}(y)\simeq \beta^*\lambda_{\gamma}(y)\circ \lambda_\beta(\gamma^*(y))$. Applying $g_b$ and precomposing with the unit of the adjoint pair $(f_b, g_b)$ gives the claim.

The second identification arises from a dual analysis using the description of $\eta_\alpha$ in terms of the mate $\rho_\alpha$.
\end{example}

Example \ref{ex:unit conjugation} indicates how the arrow $\eta_\beta$ depends on $\beta$ via both pre- and postcomposition. 
Our goal will now be to make this precise by proving the following:
\begin{theorem}\label{thm:parametrised unit}
Let $f\colon D\rt C$ be a parametrised left adjoint over $B$ with parametrised right adjoint $g$. Then there are canonical diagrams
$$\begin{tikzcd}
D \times_{B} \TwL(B)  \arrow[r,"\eta"] \ar[d] & 
    \Ar(D^{\vee}) \ar[d] && C^{\vee}
    \times_{B^{\op}} \TwR(B^{\op}) \arrow[r,"\epsilon"] \ar[d] & \Ar(C) \ar[d]\\
    B^\op \ar[r,"\mathrm{const}"] & \Ar(B^\op)  && B \ar[r,"\mathrm{const}"] & \Ar(B)
\end{tikzcd}$$
whose restrictions to $D\times_B\{\id_b\}$ and $C^\vee \times_{B^{\op}} \{\id_b\}$ for some $b$ in $B$ are equivalent to the unit and counit of the adjoint pair $(f_b, g_b)$ respectively.
\end{theorem}

After some preliminaries we will produce the functors $\eta$ and $\epsilon$ in \cref{const:paraunit} below. We will refer to these functors as the \emph{parametrised unit} and \emph{counit}, respectively.

To prepare the construction, let us write 
$$
p=(p_1, p_2) \colon X \rt B\times [1], \qquad\qquad q=(q_1, q_2)\colon X^\vee\rt [1]\times B^{\op}
$$
for the \lortho{s} classified by $f\colon D\rt C$ and $g\colon C^\vee\rt D^\vee$ respectively. Recall from \cref{thm:dualizing lax natural transformations} that $\ol{q}=(q_2, q_1)$ is the Gray fibration dual to the \lortho{} $p$. In particular, 
\[
X_0\simeq D, \qquad\qquad X_1\simeq C, \qquad\qquad X^\vee_0\simeq D^\vee \qquad \text{and} \qquad X^\vee_1\simeq C^\vee.
\] 
Naively, one could try to imitate Construction \ref{con:unit-counit} and construct the unit map using a cocartesian transport of the fibre inclusion $X_0\hookrightarrow X$, followed by a cartesian transport. More precisely, since $p_2$ is a cocartesian fibration, we can form the cocartesian transport (\cref{con:cocartesian push})
\[\begin{tikzcd} i_{0,\coc} \colon X_{0} \times [1] \arrow[r] & X\end{tikzcd}\] 
along $p_2$ of the fibre inclusion $i_{0}\colon X_{0} \hookrightarrow X$; this takes $y \in X_{0, b}\simeq D_b$ to the cocartesian morphism $y \rt f_{b}(y)$. 
Dually, $q_1$ is a cartesian fibration and we can form the cartesian transport
\[\begin{tikzcd} j_{1,\cac} \colon X_{1}^{\vee} \times [1] \arrow[r] & X^\vee\end{tikzcd}\] 
along $q_1$ of the fibre inclusion $j_{1} \colon  X^\vee_{1} \hookrightarrow X^\vee$; this takes $x \in X_{b,1}^\vee\simeq C_b^\vee$ to the cartesian morphism $g_{b}(x) \rt x$.

To construct the unit as in \cref{con:unit-counit}, we would now like take the cartesian transport of $i_{0, \coc}$ (and dually the cocartesian transport of $j_{1, \cac}$ for the counit). This can be done \emph{fibrewise} over $b\in B$, but for a global construction we will need to replace $p\colon X\rt B\times [1]$ by its dual $q\colon X^\vee\rt [1]\times B$, which is a cartesian fibration over $[1]$. Here we run into a problem, however: $i_{0, \coc}$ is a functor between \lortho{s} which generally does \emph{not} preserve cartesian arrows in the $B$-direction and hence does not induce a map between the dual fibrations. Indeed, for $\beta\colon b\rt b'$ and a cartesian morphism $\tilde{\beta}\colon \beta^*y\rt y$ in $X_0$, the image of the cartesian arrow $(\tilde{\beta}, 1)$ in $X_0\times 1$ under $i_{0, \coc}$ is
$$
f_{b}(\beta^*y)\rt f_{b'}(y).
$$
This is cartesian for all $\beta$ and $y$ if and only if $f\colon C\rt D$ preserves cartesian morphisms over $B$.

To deal with this issue (and the dual issue for the counit map), we will first extend the functor $i_{0, \coc}$ to the \emph{free} cartesian fibration on $X_0\times [1]$ and then dualise over $B$. Let us therefore briefly recall the description of free fibrations from \cite[Section 4]{GHN}.

\begin{notation}
Given a functor $\phi \colon E \to B$ we write
  \[ \begin{tikzcd}\Fcoc{B}(E) := E \times_{B}
   \Ar(B)\arrow[r] & B; \quad (e, \phi(e) \to b')
    \arrow[r, mapsto] & b',\end{tikzcd}\] where the pullback is formed along evaluation at $0$
  and the map to $B$ is given by evaluation at $1$. Dually, define
  \[ \begin{tikzcd}\Fcart{B}(E) := E \times_{B}
    \Ar(B) \arrow[r] &  B; \quad (e, b \to \phi(e))
    \arrow[r, mapsto] & b,\end{tikzcd}\] 
  where the pullback is formed along evaluation at $1$
  and the map to $B$ is given by evaluation at $0$.
\end{notation}

We will need the following result from \cite[Theorem 4.5]{GHN}:

\begin{theorem}
  The natural maps
  \[\begin{tikzcd} E \arrow[r] & \Fcart{B}(E), &  E\ar[r] & \Fcoc{B}(E)\end{tikzcd}\]
  over $B$ induced by the constant diagram functor
  $B \rt \Ar(B)$, exhibit
  $\Fcart{B}(E)$ and $\Fcoc{B}(E)$ as the free
  cartesian and cocartesian fibrations on $\phi \colon E \rt
  B$, respectively. In other words, the functors
  \[\begin{tikzcd}\Fcart{B} \colon \cat{Cat}_{/ B} \arrow[r] &  \cart(B), & \Fcoc{B} \colon \cat{Cat}_{/ B} \arrow[r] & 
    \cocart(B)\end{tikzcd}\]
  are left adjoint to the forgetful functors
  $\cart(B)\rt \Cat_{/B}$ and $\cocart(B) \rt
  \Cat_{/ B}$. \qed
\end{theorem}

\begin{remark}\label{free:func}
  Consider a commutative triangle
  \[
    \begin{tikzcd}
      E \arrow{rr}{f} \arrow{dr}[swap]{\phi} & & E'
      \arrow{dl}{p} \\
       & B,
    \end{tikzcd}
  \]
  where $p$ is a cartesian fibration. We can extend this uniquely to a
  diagram
  \[
    \begin{tikzcd}
      E \arrow{r} \arrow[bend left=40]{rr}{f} \arrow{dr}[swap]{\phi}
      & \Fcart{B}(E) \arrow{d} \arrow{r}{\overline{f}} & E'
      \arrow{dl}{p} \\
       & B,
    \end{tikzcd}
  \]
  where $\overline{f}$ preserves cartesian morphisms. Informally, the functor
  $\overline{f}$ is given by
  \[ \big(e, b \stackrel{\beta}{\rt} \phi(e)\big) \longmapsto \beta^{*}f(e),\]
  where $\beta^{*}f(e) \to f(e)$ is a cartesian morphism in $E'$ over $\beta$.
\end{remark}
 
Viewing $X_0 \times [1]$ as an $\infty$-category over $B$ via the functor 
$$\begin{tikzcd}
X_0\times [1] \arrow[r] & X_0 \arrow[r] & B,
\end{tikzcd}$$
we can then extend $i_{0,\coc}$ to the free cartesian fibration over $B$ as in Remark \ref{free:func}. Similarly we can extend $j_{1,\cac}$ to the free cocartesian fibration over $B^{\op}$, giving
  \[\begin{tikzcd}[row sep=0.4pc, column sep=1.4pc] \overline{\imath}_{0,\coc} \colon \Fcart{B}\big(X_{0} \times
    [1]\big) \simeq \Fcart{B}(X_{0}) \times
    [1] \arrow[r] & X,\\
    \overline{\jmath}_{1,\cac} \colon
    \Fcoc{B^{\op}}\big(X^\vee_{1} \times [1]\big) \simeq
    \Fcoc{B^{\op}}(X^\vee_{1}) \times [1] \arrow[r] & X^\vee.\end{tikzcd}\]
These can also be viewed as functors into arrow \icats{}, informally given by
\[\begin{tikzcd}[row sep=0.1pc]
\Fcart{B}(X_0)\arrow[r] & \Ar(X); \qquad \big(y \in X_{b',0}, b
    \xto{\beta} b'\big) \arrow[r, mapsto] &  \big(\beta^{*}y \to \beta^{*}f_{b'}y\big),\\
   \Fcoc{B^{\op}}(X^\vee_1)\arrow[r] & \Ar(X^\vee);\qquad \big(x \in
    X^{\vee}_{b,1}, b \xto{\beta} b'\big) \arrow[r, mapsto] & \big(\beta^{\op}_!g_{b}x \to
    \beta^{\op}_!x\big).
\end{tikzcd}\]

By construction the functors $\overline{\imath}_{0,\coc}$ and $\overline{\jmath}_{1,\cac}$ preserve cartesian morphisms over $B$ and cocartesian morphisms over $B^{\op}$, respectively.  Therefore they induce functors between the dual (co)cartesian fibrations, and we obtain functors
\[\begin{tikzcd}[row sep=0.4pc]
 \Dualco(\overline{\imath}_{0,\coc}) \colon \Dualco\big(\Fcart{B}(X_{0}) \times
    [1]\big)\simeq \Dualco\big(\Fcart{B}(X_{0})\big) \times
    [1]\arrow[r] & 
    X^\vee,\\
    \Dualcart(\overline{\jmath}_{1,\cac}) \colon \Dualcart\big(\Fcoc{B^{\op}}(X^\vee_{1}) \times [1]\big)\simeq \Dualcart\big(\Fcoc{B^{\op}}(X^\vee_{1})\big) \times [1] \arrow[r] & X.
 \end{tikzcd}\]
The first equivalence in the two lines follows from the fact that the dual of a constant cocartesian fibration is a constant cartesian fibration. We now note that the domains of the functors $\Dualco(\overline{\imath}_{0,\coc})$ and $\Dualcart(\overline{\jmath}_{1,\cac})$ admit further simplification, by means of the following extension of the duality between arrow and twisted arrow $\infty$-categories from \cite[Lemma 3.1.3]{symmseq}:

\begin{lemma}\label{lem:freedual}
  For $\phi \colon E \rt B$, the duals of the free
  fibrations on $\phi$ can be identified as
  \[ \Dualco\big(\Fcart{B}(E)\big) \simeq E
    \times_{B} \TwL(B) \to B^{\op},\]
  \[ \Dualcart\big(\Fcoc{B}(E)\big) \simeq E
    \times_{B} \TwR(B) \to B^{\op}\]
    naturally in $\phi$. \qed
\end{lemma}
We now have all the ingredients to construct the parametrised unit and counit from \cref{thm:parametrised unit}:
\begin{construction}\label{const:paraunit}
As before, let $p=(p_1, p_2)\colon X\rt B\times [1]$ and $q=(q_1, q_2)\colon X^\vee\rt [1]\times B^{\op}$ be the orthofibrations classified by $f$ and $g$. Lemma \ref{lem:freedual} now implies that the functors $\overline{\imath}_{0,\coc}$ and $\overline{\jmath}_{1,\cac}$ have duals
\[ \begin{tikzcd}[row sep=0.3pc]
\Dualco(\overline{\imath}_{0,\coc}) \colon X_{0}
    \times_{B} \TwL(B) \times [1]\arrow[r] & X^\vee,\\
\Dualcart(\overline{\jmath}_{1,\cac}) \colon X^\vee_{1}
  \times_{B^{\op}} \TwR(B^{\op}) \times [1] \arrow[r] & X\end{tikzcd}\]
that preserve cocartesian and cartesian morphisms over $B$, respectively. Now we can form the cartesian transport of $\Dualco(\overline{\imath}_{0,\coc})$ via $q_1$ and the cocartesian transport of $\Dualcart(\overline{\jmath}_{1,\cac})$ via $p_2$, respectively. This gives functors
\[\begin{tikzcd}[row sep=0.3pc]
(\Dualco(\overline{\imath}_{0,\coc}))_{\cac} \colon X_{0}
    \times_{B} \TwL(B) \times [1] \times
    [1] \arrow[r] & 
    X^\vee,\\
    (\Dualcart(\overline{\jmath}_{1,\cac}))_{\coc} \colon X_{1}^{\vee}
    \times_{B^{\op}} \TwR(B^{\op}) \times
    [1] \times [1] \arrow[r] & X.
    \end{tikzcd}\]

    We can informally describe these functors as follows: the value of
    $(\Dualco(\overline{\imath}_{0,\coc}))_{\cac}$ at an object
    $(y \in X_{b',0}, \beta \colon b \to b')$ is the square
    \[
      \begin{tikzcd}
        \beta^{\op}_!(y) \arrow{r} \arrow[equals]{d} & g_{b}\beta^{\op}_!f_{b'}y
        \arrow{d} \\
        \beta^{\op}_!(y) \arrow{r} & \beta^{\op}_!f_{b'}y
      \end{tikzcd}
    \]
    in $X^{\vee}$. Note that the top horizontal arrow takes values in $X^\vee_0\simeq D^\vee$. Dually, the value of
    $(\Dualcart(\overline{\jmath}_{1,\cac}))_{\coc}$ at $(x \in
    X_{b',1}^{\vee}, \beta \colon b \to b')$ is the square
    \[
      \begin{tikzcd}
        \beta_!^{*}g_{b'}x \arrow{r} \arrow{d} & \beta^{*}x \arrow[equals]{d} \\
        f_{b}\beta^{*}g_{b'}x \arrow{r} & \beta^{*}x,
      \end{tikzcd}
    \]
  whose bottom arrow is contained in $X_1\simeq C$. We then obtain the desired parametrised unit and counit maps as the restrictions
 \[\begin{tikzcd}[row sep=0.3pc]
 \eta :=
      (\Dualco(\overline{\imath}_{0,\coc}))_{\cac}|_{X_{0}
    \times_{B} \TwL(B) \times [1] \times \{0\}} \colon
      X_{0}
    \times_{B} \TwL(B) \times [1]  \arrow[r] & 
    X^{\vee}_{0}\\
    \epsilon := (\Dualcart(\overline{\jmath}_{1,\cac}))_{\coc}|_{X_{1}^{\vee}
    \times_{B^{\op}} \TwR(B^{\op}) \times
    [1] \times \{1\}}
    \colon X_{1}^{\vee}
    \times_{B^{\op}} \TwR(B^{\op}) \times
    [1] \arrow[r] & X_{1}.
    \end{tikzcd}\]
 Note that this construction is natural in $X$ (and hence in $f\colon C\rt D$) and is compatible with base change along $B'\rt B$. In the case where $B=\ast$ is a point, the free fibration and dualisation functors are naturally equivalent to the identity \cite{Toen, BGN} and the above construction reduces to the construction of the (co)unit from \cref{con:unit-counit}.
 \end{construction}

\subsection{Passing to adjoint morphisms}\label{subsec:paraunits2}

Next we consider the functoriality of passing to the adjoint of a morphism in the parametrised setting. We first sketch a construction in the non-parametrised case, which will have the benefit of generalising readily. Given an adjunction
\[ f \colon D \rightleftarrows C \cocolon g,\]
the unit transformation $\eta$ fits in a
commutative square
\[
  \begin{tikzcd}
    D \arrow[d, "f"{swap}] \arrow{r}{\eta} &
    \Ar(D) \arrow{d}{\ev_{1}} \\
    \cat{C} \arrow{r}{g} & \cat{D}.
  \end{tikzcd}
\]
Here $\ev_{1}$ is a cocartesian fibration, so we can extend $\eta$ to
the free cocartesian fibration on $f$, giving a commuting square
\[
  \begin{tikzcd}
    \Fcoc{C}(D) \arrow{d} \arrow{r}{\overline{\eta}} &
    \Ar(D) \arrow{d}{\ev_{1}} \\
    C \arrow{r}{g} & D.
  \end{tikzcd}
\]
Unwinding definitions, we find that $\overline{\eta}$ takes $(d, f(d) \xto{\phi} c)$ to the
composite $d \to gf(d) \stackrel{g(\phi)}{\rt} g(c)$, i.e.\ to the morphism adjoint to
$\phi$. We now give a parametrised version of this construction:
\begin{construction}\label{const:passadj}
We keep the notation of \cref{thm:parametrised unit} and let $f\colon D\rt C$ be a parametrised left adjoint over $B$, with right adjoint $g\colon C^\vee\rt D^\vee$. The parametrised unit
  $\eta$ fits in a commutative square
  \[
    \begin{tikzcd}
      D \times_{B} \TwL(B)
      \arrow{r}{\eta} \arrow[d, "\Dualco(\overline{f})"{swap}] & \Ar(D^{\vee})
      \arrow{d}{\ev_{1}} \\
      C^{\vee} \arrow{r}{g} & D^{\vee},
    \end{tikzcd}
  \]
  where $\Dualco(\overline{f})$ is obtained by first extending $f \colon
  D \to C$ to $\overline{f} \colon
  \Fcart{B}(D) \to C$ and then
  dualising over $B$. At the level of objects, $\Dualco(\ol{f})$ is therefore given by
  \[ \big(y \in D_{b'}, \beta \colon b \to b')
    \longmapsto \beta^{*}f_{b'}(y). \]
Now we can extend $\eta$ over the free cocartesian fibration on $\Dualco(\overline{f})$, giving a commutative square
  \begin{equation}
    \label{eq:parpassadj}
    \begin{tikzcd}
      \Fcoc{C^{\vee}}(D \times_{B} \TwL(B))
      \arrow{r}{\overline{\eta}} \arrow{d} &
      \Ar(D^{\vee})
      \arrow{d}{\ev_{1}} \\
      C^{\vee} \arrow{r}{g} & D^{\vee}.
    \end{tikzcd}    
  \end{equation}
  Here $\overline{\eta}$ is given by the assignment
  \[ \big(y \in D_{b'}, b \stackrel{\beta}{\rt} b', \beta^{*}f_{b'}(y) \stackrel{\phi}{\rt} x\big)
    \longmapsto \big(\beta^{*}y \stackrel{\eta_\beta(y)}{\rt} g_{b}\beta^{*}f_{b'}y \stackrel{g(\phi)}{\rt} g_{b''}x\big)\]
    where $x\in C^\vee_{b''}$.
  We can also pass to the dual cartesian fibrations, which gives a
  commutative square
  \begin{equation}
    \label{eq:parpassadjdual}
    \begin{tikzcd}
      (D \times_{B} \TwL(B))
      \times_{C^{\vee}} \TwR(C^{\vee})
      \arrow{r}{\overline{\eta}^\vee} \arrow{d} &
      \TwR(D^{\vee})
      \arrow{d} \\
       (C^{\op})^\vee \arrow{r}{g^{\op}} & (D^{\op})^\vee.
    \end{tikzcd}    
  \end{equation}
\end{construction}

\cref{const:passadj} encodes the functoriality of passing to the adjoint morphism in the generic case of a parametrised adjunction. However if the parametrised adjunction has a particularly simple form, then the functoriality can be improved significantly:

\begin{example}\label{ex:counitftr} 
Recall from \cref{ex:monoidalclosed} that given a functor $f:D\times B \rightarrow C$ such that for each $b\in B$, $f(-,b)\colon D \rightarrow C$ is a left adjoint, the diagram 
\begin{center}
\begin{tikzcd}
D\times B \arrow[rd, "p_2"'] \arrow[rr, "{(f,\id_B)}"] &    & C\times B \arrow[ld, "p_2"] \\
                                                          & B &                              
\end{tikzcd}
\end{center} is an example of a parametrised left adjoint. In this case the parametrised unit from \cref{thm:parametrised unit} is a functor
  \[ \eta \colon D \times \TwL(B) \times
    [1] \to
    D.\] To an object $(y, b \xto{\beta} b')$ this
  assigns the map
  $y \to g(f(y,b'),b)$ adjoint to $f(y,b) \to f(y,b')$.
  To a morphism
  \[ \left(
      y_0\to y_1,
      \begin{tikzcd}
        b_0  \arrow{d} & \arrow{l} b_1 \arrow{d} \\
        b'_0 \arrow{r} & b'_1 
      \end{tikzcd}
    \right) \]
  it assigns the square
  \[\nlcsquare{y_0}{y_1}{g(f(y_0,b'_0),b_0)}{g(f(y,b'_1),b_1).}\]
  Now we consider the commutative square \cref{eq:parpassadjdual}
  from \cref{const:passadj}; in our special case this simplifies to
        \[
    \begin{tikzcd}
      (D \times \TwL(B))
      \times_{C \times B^{\op}}
      \TwR(C \times B^{\op})
      \arrow{r} \arrow{d}  & \TwR(D)
      \arrow{d} \\
      C^{\op}\times B \arrow{r}{g^{\op}} &
      D^{\op}.
    \end{tikzcd}
  \]
  An object of $ (D \times \TwL(B))
      \times_{C \times B^{\op}}
      \TwR(C \times B^{\op})$ can be
  described as a list
  \[(y, b \to b', f(y,b') \to x, b'' \to b),\] and
  the top horizontal functor takes this to the composite
  \[ y \to g\big(f(y,b'),b\big) \to g(x,b) \to g(x,b'')\]
  in $\TwR(D)$. For a morphism
  \[\left(
      \begin{tikzcd}
        y_0 \arrow{d} & b_0 \arrow{r} & b'_0 \arrow{d} & f(y_0,b'_0) \arrow{d}
        \arrow{r} & x_0 & b_0'' \arrow{d}
        \arrow{r} & b_0\\
        y_1, & b_1 \arrow{u} \arrow{r} & b'_1, & f(y_1,b'_1) \arrow{r} & x_1,
        \arrow{u} & b_1'' \arrow{r} 
        & b_1 \arrow{u}
      \end{tikzcd}
    \right)
  \]
  we get in $\TwR(D)$ a morphism
  \[
    \begin{tikzcd}
       y_0  \arrow{d} \arrow{r} & y_1 \arrow{d} \\
       g(x_0,b_0'')  &   g(x_1,b_1'') \arrow{l}.
    \end{tikzcd}
  \]
  We note that this is an equivalence if the maps $y_0 \to y_1$, $x_1 \to
 x_0$, and $b_0'' \to b_1''$ are equivalences. This means
  our functor factors through the localisation of the $\infty$-category \[(D \times \TwL(B))
  \times_{C \times B^{\op}}
  \TwR(C \times B^{\op})\] at these morphisms. Our final goal
  in this section is to identify this localisation, for which we first
  recall a result of Hinich:
\end{example}

\begin{proposition}[Hinich]\label{propn:fibloc}
  Let $p \colon E \to B$ be a cocartesian
  fibration. Suppose for all $b \in B$ we have a collection
  $W_{b}$ of morphisms in
  $E_{b}$ such that for $\beta \colon b \to b'$ in $B$
  the functor $\beta_{!} \colon E_{b} \to E_{b'}$ induced by $p$
  takes $W_{b}$ into $W_{b'}$. Then we can form the cocartesian
  fibration $E' \to B$ corresponding to the
  functor $b \mapsto \cat{E}_{b}[W_{b}^{-1}]$. The canonical
  morphism of cocartesian fibrations $E \to E'$
  exhibits $E'$ as the localisation of $E$ at the
  collection of morphisms $x \xto{\phi} x'$ such that $p(\phi)$ is an
  equivalence and $p(\phi)_{!}x \to x'$ is in $W_{p(x')}$.
\end{proposition}

\begin{proof}
  This is a special case of \cite[Proposition 2.1.4] {HinichLoc} (or
  more precisely, of the stronger result that is actually proved in
  \cite[Section 2.2]{HinichLoc}). See also \cite[Proposition A.14]{NS} for a generalisation, as well as a more invariant proof.
\end{proof}

This allows us to prove the following:

\begin{corollary}\label{cor:Twrdualloc}
  Suppose $p \colon E \to B$ is a cocartesian fibration; then
  the identity map of $E$ induces (via
  the free cocartesian fibration) a morphism of cocartesian fibrations
  $\Fcoc{B}(E) = E \times_{B} \Ar(B) \to
  E$; passing to the dual cartesian fibrations we get a morphism of
  cartesian fibrations
  \[ E \times_{B} \TwR(B) \xto{\Phi} E^{\vee}
  \]
  over $B^{\op}$. For any functor $A \to
  B^\op$ the induced morphism of cartesian fibrations
  \[ E \times_{B} \TwR(B)
    \times_{B^{\op}} A \xto{\Phi'}
    E^{\vee} \times_{B^{\op}} A
  \]
  exhibits $E^{\vee} \times_{B^{\op}} A$ 
  as a localisation.
\end{corollary}

\begin{proof}
Suppose first that $A\rightarrow B$ is the identity. At the fibre over $b \in B^{\op}$ we get the functor
  \[ E \times_{B} B_{/b} \to
    E_{b}\]
  taking $(x \in E_{b'}, b' \xto{\beta} b)$ to $\beta_{!}x$. This
  has a fully faithful right adjoint (taking $x \in E_{b}$
  to $(x, \id_{b})$), hence it is the localisation at the class
  $W_{b}$ of
  morphisms $(x \xto{\phi} y, b' \xto{\gamma} b'' \xto{\beta} b)$ such that
  $\beta_{!}\gamma_{!}x \to \beta_{!}y$ is an equivalence. For $\beta \colon b \to
  b' \in B$, the functor induced by the cartesian fibration over $B^{\op}$, $E \times_{B}
   B_{/b} \to E \times_{B}
  B_{/b'}$ is given by composition with $\beta$, and hence takes
  $W_{b}$ to $W_{b'}$. The result then follows from (the dual of) Proposition~\ref{propn:fibloc}. Finally note that the conclusion of Proposition~\ref{propn:fibloc} is preserved under base change along any functor $A \to B$ and therefore the general result follows.
\end{proof}

Taking $p$ to be the identity of $B$, we obtain the
following special case:
\begin{corollary}\label{cor:Twrloc}
  For any \icat{} $B$, the projection
  \[ \TwR(B) \to B^{\op} \]
  is a localisation, as is the functor
  \[ A\times_{B^{\op}} \TwR(B) \to
    A\]
  for any functor $A\to B^{\op}$. \qed
\end{corollary}

Returning to the $B$-indexed family of left adjoints $f\colon D\times B\rt C$ from Example \ref{ex:counitftr}, we see that the
functor
\[ (D \times \TwL(B)) \times_{C
    \times B^{\op}} \TwR(C \times
  B^{\op}) \to \TwR(D) \] obtained from the
parametrised unit factors through
$(D \times B)\times_{C}
\TwR(C)$. We have thus proved:
\begin{corollary}\label{cor:prodparunit}
  Let $f\colon D\times B\rt C$ be a functor such that each $f_b\colon D\rt C$ is a left adjoint. Then there is a
  functor
  \[ (D \times B)\times_{C}
    \TwR(C) \to \TwR(D),\] which takes
  $(y,b,f(y,b) \to x)$ to the adjoint map $y \to g(x,b)$.\qed
\end{corollary}

Restricting to the fibre over $x \in
C$, we see in particular:
\begin{corollary}\label{cor:prodparunitfib}
  In the situation of \cref{cor:prodparunit}, for every $x \in
  C$ there is a natural map
  \[ (D \times B) \times_{C}
    C_{/x} \to \TwR(D) \]
  sending $(y, b, f(y,b) \to x)$ to the adjoint map $y \to g(b,x)$.\qed
\end{corollary}

\begin{example}\label{ex:dualparadj}
  Let $\cat{C}$ be a closed symmetric monoidal \icat{}, with the
  tensor product viewed as a $\cat{C}$-parametrised left adjoint as in
  Example~\ref{ex:monoidalclosed}. From Corollary \ref{cor:prodparunit} we
  obtain a functor
  \[ (\cat{C} \times \cat{C}) \times_{\cat{C}}
    \TwR(\cat{C}) \to \TwR(\cat{C}),\]
  taking $(x,y, x\otimes y \to z)$ to the adjoint map $x \to [y,z]$.
  Fixing $z \in \cat{C}$, this specialises as in Corollary \ref{cor:prodparunitfib} to a natural functor
  \[ (\cat{C} \times \cat{C}) \times_{\cat{C}}
    \cat{C}_{/z} \to \TwR(\cat{C}) \] that sends $(x, y, x
  \otimes y \to z)$ to the adjoint morphism $x \to [y,z]$.
\end{example}

\section{Lax natural transformations and the calculus of mates}\label{sec:laxgray}
The goal of this final section is to prove \cref{laxstr}, i.e.\ to produce straightening equivalences $$
\mathbf{Cocart}^{\lax}(B)\simeq \tcat{Fun}^{\lax}(B, \tcat{Cat}) \qquad \text{and} \qquad \mathbf{Cart}^{\oplax}(B)\simeq \tcat{Fun}^{\oplax}(B^{\op}, \tcat{Cat}),
$$ 
where the right hand $(\infty,2)$-categories consist of functors $B \rightarrow \tcat{Cat}$ as objects, (op)lax natural transformations as morphisms and modifications between these as $2$-morphisms. Following \cite{HaugsengLax} we define these $\infty$-categories as right adjoints to the (oplax) Gray tensor product for $(\infty,2)$-categories constructed by Gagna, Lanari and Harpaz in \cite{GHL-Gray}, so that there are equivalences
\begin{align*}
\Map_{\Cat_2}(\tcat{A}\Gtimes B, \tcat{Cat}) &\simeq \Map_{\Cat_2}\big(\tcat{A}, \tFun^\lax(B, \tCat)\big)\\
\Map_{\Cat_2}(B\Gtimes \tcat{A}, \tcat{Cat}) &\simeq \Map_{\Cat_2}\big(\tcat{A}, \tFun^\oplax(B, \tCat)\big).
\end{align*}
Their Gray tensor product is defined using Lurie's scaled simplicial sets from \cite{LurieGoo} as a model for $(\infty,2)$-categories, and so we begin with a short review of these in \S\ref{subsec:scaled}. In \S\ref{subsec:scunstrGray} we then show that Lurie's straightening equivalence for locally cocartesian fibrations restricts to an equivalence
\[\Fun(A\Gtimes B, \tcat{Cat})\simeq \Gray(A, B),\]
from which we then deduce \cref{laxstr} and \cref{core} in \S\ref{subsec:laxnat}.

\subsection{Scaled simplicial sets as a model for $(\infty,2)$-categories}\label{subsec:scaled}

We start by recalling a few definitions:

\begin{definition}\label{def:marked simplicial sets}
A \emph{marked simplicial set} is a pair $(X, T)$ with $X$ a simplicial set and $T\subseteq X_1$ a set of 1-simplices that contains the degenerate ones. Let $\markSet$ denote the category of
marked simplicial sets. 
\end{definition}

By \cite[Theorem 3.1.5.1]{HTT} the category $\markSet$ has a model structure Quillen equivalent to the Joyal model structure on $\sSet$, whose fibrant objects are precisely quasicategories marked by their equivalences. We also write $\markCat$ for the
category of \emph{marked simplicial categories}, i.e.\ categories enriched in marked simplicial sets.

\begin{definition}\label{def:scaled simplicial sets}
A \emph{scaled simplicial set} is a pair $(X, S)$ with $X$ a simplicial set and $S\subseteq X_2$ a set of 2-simplices that contains the degenerate ones. As usual, we will write $X^\sharp=(X, X_2)$ for $X$ with the maximal scaling. Let $\scSet$ denote the category
of scaled simplicial sets, with the morphisms being maps of simplicial sets that preserve the scalings. 

We write $N^{\scale}\colon \markCat\rt \scSet$ for the scaled nerve, which takes a marked simplicial category $\mathbf{C}$ to the coherent nerve $N\mathbf{C}$ of its underlying simplicial category, scaled by the set of 2-simplices $\Delta^2 \rt N\mathbf{C}$ corresponding to functors of simplicial categories $F\colon \mathfrak{C}(\Delta^2)\rt\mathbf{C}$ such that the edge $\Delta^1= \mathfrak{C}(\Delta^2)(0, 2) \rt \mathbf{C}(F(0), F(2))$ is
marked; here $\mathfrak{C}$ denotes the path category functor, left adjoint to the coherent nerve $N$. Its upgrade to a left adjoint of $\nerve^\scale$ we denote $\mathfrak{C}^{\scale}$.
\end{definition}

The following is \cite[Theorem 4.2.7]{LurieGoo}:

\begin{theorem}[Lurie]\label{thm:scmodelstr}
There is a model structure on $\scSet$ where the cofibrations are the monomorphisms and the weak equivalences are the maps $f$ such that $\mathfrak{C}^{\scale}(f)$ is a Dwyer-Kan equivalence of marked simplicial categories. Moreover, the adjunction $\mathfrak{C}^{\scale} \vdash N^{\scale}$ is a Quillen equivalence where $\markCat$ carries the marked Bergner model structure.\qed
\end{theorem}

\begin{remark}
An explicit description of the fibrant objects in $\scSet$ in terms of lifting properties has been obtained by Gagna, Harpaz and Lanari in \cite{GHL19}.
\end{remark}

It is then a consequence of the main results of \cite{LurieGoo} that the underlying $\infty$-category of $\scSet$ is equivalent to $\Cat_2$. Given this, a particularly simple description of the equivalence follows from work of Barwick and Schommer-Pries \cite{BSP}: Let us write $\Theta_2$ for the full subcategory of the (ordinary) category of strict $2$-categories spanned by the strict $2$-categories
\[
[m]\big([n_1], \dots, [n_m]\big)=
\begin{tikzcd}
0 \arrow[r, bend right=60, "n_1"{swap}, ""{name=t1}] \arrow[r, dash, dotted, ""{name=m1}, ""{swap, name=n1}] \arrow[r, bend left=60, "0", ""{swap, name=s1}] & 1\arrow[Rightarrow, from=s1, to=m1, shorten=-2pt]\arrow[Rightarrow, from=n1, to=t1, shorten=-2pt]\arrow[r, bend right=60, "n_2"{swap}, ""{name=t2}] \arrow[r, dash, dotted, ""{name=m2}, ""{swap, name=n2}] \arrow[r, bend left=60, "0", ""{swap, name=s2}] &
\phantom{3}\makebox[0pt]{\hspace{-4.5pt}$\cdots$}\arrow[Rightarrow, from=s2, to=m2, shorten=-2pt]\arrow[Rightarrow, from=n2, to=t2, shorten=-2pt] \arrow[r, bend right=60, "n_m"{below}, ""{name=tm}] \arrow[r, dash, dotted, ""{name=mm}, ""{below, name=nm}] \arrow[r, bend left=60, "0", ""{below, name=sm}] &
\phantom{3}\makebox[0pt]{$m.$}\arrow[Rightarrow, from=sm, to=mm, shorten=-2pt]\arrow[Rightarrow, from=nm, to=tm, shorten=-2pt]
\end{tikzcd}
\]
Since these have only identities as invertible $k$-morphisms, we obtain a full subcategory inclusion $\Theta_2\hookrightarrow \scSet$ by viewing these $2$-categories as marked simplicial categories with only degenerate edges marked and then applying the scaled nerve. Now consider the functor
\begin{equation}\label{diag:diagonal}\begin{tikzcd}
\delta_2\colon \Delta\times \Delta\ar[r] & \Theta_2\hookrightarrow \scSet \quad {\big([m], [n]\big)}\arrow[r, mapsto] & {[m]([n], \dots, [n])}.
\end{tikzcd}\end{equation}
It now follows from the main results of \cite{BSP} that the derived mapping $\infty$-groupoids
$$
\Map^h_{\scSet}\big(\delta_2(-, -), (X, S)\big) \colon \Delta^\op \times \Delta^\op \longrightarrow \Gpd.
$$
form a complete 2-fold Segal $\infty$-groupoid, and that this
assignment induces an equivalence between the $\infty$-category
associated to the model structure on $\scSet$ from
\cref{thm:scmodelstr} and the \icat{} of complete two-fold Segal
spaces. In other words: scaled simplicial sets are a model for
$(\infty,2)$-categories.

\begin{definition}
We write $\tcat{Cat}^\scale$ for the large scaled simplicial set $N^{\scale}(\cat{sSet}^{+, \circ})$, where the category $\cat{sSet}^{+, \circ}$ of fibrant marked simplicial sets is regarded as enriched in itself via its internal Hom. 
\end{definition}

In this section we will use the $(\infty, 2)$-category associated to the scaled simplicial set $\tcat{Cat}^\scale$ as our preferred model for the $(\infty, 2)$-category $\tcat{Cat}$ of $\infty$-categories.

We now recall the definition of the (oplax) Gray tensor product in terms of scaled simplicial
sets, as given in \cite{GHL-Gray}:
\begin{definition}\label{def:scaled gray}
If $(X, S)$ and $(Y, T)$ are scaled simplicial sets, we define their \emph{oplax Gray tensor product}
$$
(X,S) \Gtimes (Y,T) = \big(X\times Y, S\Gtimes T\big)
$$
to be the scaled simplicial set with underlying simplicial set $X\times Y$, with scaling $S \Gtimes T$ consisting of the 2-simplices of the forms:
\begin{itemize}
\item $(s_1\alpha ,\tau)$ with $\alpha\in X_1, \tau\in T$,
\item $(\sigma, s_0\beta)$ with $\sigma\in S$, $\beta\in Y_1$.
\end{itemize}
For simplicial sets $X$ and $Y$ we will abbreviate $X^\sharp\Gtimes Y^\sharp$ to just $X\Gtimes Y$.
\end{definition}

From \cite[Theorem 2.14]{GHL-Gray} we quote:

\begin{theorem}[Gagna, Harpaz, Lanari]\label{thm:gray tensor}
The oplax Gray tensor product \[ \Gtimes \colon \scSet\times \scSet
  \rt \scSet\]
is a left Quillen bifunctor. \qed
\end{theorem}
It follows that the oplax Gray tensor product induces a functor on the
level of $\infty$-categories
$$
- \Gtimes -\colon \Cat_2\times\Cat_2\rt \Cat_2,
$$
which preserves colimits in each variable. As the name suggests, this is supposed to be thought of as a homotopy-coherent refinement of the standard oplax Gray tensor product for strict $2$-categories \cite{Gray74}. This is supported by the following:
\begin{proposition}\label{prop:gray of simplex}
For any $m, n\geq 0$, there is a natural isomorphism between the oplax Gray tensor product $[m]\Gtimes [n]$ from Theorem \ref{thm:gray tensor} and the standard oplax Gray tensor product $[m]\Gtimes_{\mm{st}} [n]$ of $[m]$ and $[n]$, computed in strict $2$-categories and depicted informally as
\begin{equation}\label{diag:grid}\begin{tikzcd}[column sep=1.5pc, row sep=1.5pc]
00\arrow[r]\arrow[d] & 10 \arrow[d]\arrow[r]\arrow[Rightarrow, ld, start anchor={[xshift=-1ex, yshift=-1ex]}, end anchor={[xshift=1.2ex, yshift=1.2ex]}] & 20\arrow[r, dashed]\arrow[d]\arrow[Rightarrow, ld, start anchor={[xshift=-1ex, yshift=-1ex]}, end anchor={[xshift=1.2ex, yshift=1.2ex]}] & m0\arrow[d]\arrow[Rightarrow, ld, start anchor={[xshift=-1ex, yshift=-1ex]}, end anchor={[xshift=1.2ex, yshift=1.2ex]}]\\
01\arrow[r]\arrow[d, dashed] & 11 \arrow[r]\arrow[d, dashed]\arrow[Rightarrow, ld, start anchor={[xshift=-1ex, yshift=-1ex]}, end anchor={[xshift=1.2ex, yshift=1.2ex]}] & 21\arrow[r, dashed]\arrow[d, dashed] \arrow[Rightarrow, ld, start anchor={[xshift=-1ex, yshift=-1ex]}, end anchor={[xshift=1.2ex, yshift=1.2ex]}]&  m1\arrow[d, dashed]\arrow[Rightarrow, ld, start anchor={[xshift=-1ex, yshift=-1ex]}, end anchor={[xshift=1.2ex, yshift=1.2ex]}]\\
0n\arrow[r] & 1n \arrow[r] & 2n\arrow[r, dashed] & mn.
\end{tikzcd}\end{equation}
\end{proposition}
\begin{proof}
Note that $\Delta[n]^\sharp$ is a scaled simplicial set model for $[n]$, viewed as an $(\infty, 2)$-category. The oplax Gray tensor product from Theorem \ref{thm:gray tensor} can then be modeled by the marked simplicial category $\mathfrak{C}^{\scale}(\Delta[m]^\sharp\Gtimes \Delta[m]^\sharp)$. Forgetting the marking, this simplicial category is the Boardman--Vogt resolution of $[m]\times [n]$ (cf.\ \cite[Proposition 6.3.3]{MoerdijkToen}). Consequently, it can be identified with the simplicial category whose objects are tuples $x=(x^0, x^1)$ with $0\leq x^0\leq m$ and $0\leq x^1\leq n$, and where 
$$
\Map_{\mathfrak{C}^{\scale}(\Delta[m]^\sharp\Gtimes \Delta[n]^\sharp)}(x, y)=\nerve\big(\mm{Ch}_{x, y}\big)
$$
is the nerve of the poset $\mm{Ch}_{x, y}$ of nondegenerate chains $\sigma=[x=x_0< x_1<\dots < x_t=y]$ in $[m]\times [n]$ starting at $x$ and ending at $y$, ordered by subchain inclusions. Composition is given by concatenation of chains. Furthermore, a subchain inclusion $\sigma'\subseteq \sigma$ is marked if it is obtained by removing one $x_i$ from $\sigma$, such that either $x_{i}^0=x_{i+1}^0$ or $x_{i-1}^1=x_i^1$.

On the other hand, $[m]\Gtimes_{\mm{st}} [n]$ can be described as the following strict $2$-category \cite{HaugsengLax}: its objects are tuples $x=(x^0, x^1)$ with $0\leq x^0\leq m$ and $0\leq x^1\leq n$ and
$$
\Map_{[m]\Gtimes_{\mm{st}} [n]}(x, y)=\mm{MaxCh}_{x, y}
$$
is the poset whose objects are maximal nondegenerate chains from $x$ to $y$, with order generated by
$$
\scalebox{0.5}{\begin{tikzcd}[ampersand replacement=\&, cramped] {}\arrow[d] \& {}\\ {}\arrow[r] \& {} \end{tikzcd}} \leq  \scalebox{0.5}{\begin{tikzcd}[ampersand replacement=\&, cramped] {}\arrow[r] \& {}\arrow[d] \\ \& {} \end{tikzcd}}
$$
in the picture \eqref{diag:grid}.
Composition is concatenation of such chains. For each tuple $x$ and $y$, we will specify a map of posets $\mm{max}\colon \mm{Ch}_{x, y}\rt \mm{MaxCh}_{x, y}$ as follows: for any chain $\sigma$ from $x$ to $y$ in the grid \eqref{diag:grid}, let $\sigma\subseteq \max(\sigma)$ be the unique maximal chain extending $\sigma$ that is maximal with respect to the partial ordering on $\mm{MaxCh}$: this means that every arrow in the chain $\sigma$ going $r$ steps right and $d$ steps down is replaced by the maximal chain first going $r$ steps right and then $d$ steps down. One easily verifies that $\mm{max}$ is a map of posets, which sends every marked arrow in $\mm{Ch}_{x, y}$ to the identity. Furthermore, it is compatible with concatenation of chains. We therefore obtain a natural map
$$\begin{tikzcd}
\phi\colon \mathfrak{C}^{\scale}(\Delta[m]^\sharp\Gtimes \Delta[m]^\sharp)\arrow[r] & {[m]\Gtimes_{\mm{st}} [n]}
\end{tikzcd}$$
where we view $[m]\Gtimes_{\mm{st}} [n]$ as a marked simplicial category by taking nerves of mapping categories and marking equivalences (which in this case are just identities). To see that this is an equivalence, it remains to verify that $\mm{max}\colon \mm{Ch}_{x, y}\rt \mm{MaxCh}_{x, y}$ exhibits $\mm{MaxCh}_{x, y}$ as the localisation of $\mm{Ch}_{x, y}$ at the marked arrows. To see this, observe that the functor $\mm{max}$ is a cocartesian fibration. For each maximal chain $\tau$, the inverse image of $\tau$ has a maximal element ($\tau$ itself) and for every other $\sigma$ in the inverse image, the inclusion of chains $\sigma\subseteq \tau$ is a composite of marked arrows. It follows that the fibres of $\mm{max}$ have contractible realisation, so $\phi$ is an equivalence as desired.
\end{proof}

 We will need the following observation about the functor $\delta_2$:

\begin{lemma}\label{lem:localization of Gray tensor product}
For each $[m], [n]\in \Delta$, there is a natural map of $(\infty, 2)$-categories
$$
[m]\Gtimes [n]\rt \delta_2\big([m], [n]\big) = {[m]([n], \dots, [n])}
$$
which exhibits the codomain as the localisation of $[m]\Gtimes [n]$ at all $1$-morphisms contained in some $\{i\}\Gtimes [n]$.
\end{lemma}
\begin{proof}
Since $[m]\Gtimes [n]$ and $\delta_2([m], [n])=[m]([n], \dots, [n])$ are both gaunt 2-categories (i.e.\ the only invertible $2$-morphisms are the identities), the desired natural map $[n]\Gtimes [m]\rt [m]([n], \dots, [n])$ is simply the evident map of strict 2-categories that collapses all $\{i\}\Gtimes [n]$ to the $i$-th vertex in $[m]([n], \dots, [n])$. For instance, for $[m]=[2]$ and $[n]=[1]$, it is given pictorially by the map collapsing the vertical $1$-morphisms
$$\begin{tikzcd}
{} \arrow[r]\arrow[d] & {}\arrow[r]\arrow[d]\arrow[Rightarrow, ld, start anchor={[xshift=-1ex, yshift=-1ex]}, end anchor={[xshift=1.2ex, yshift=1.2ex]}] & \arrow[d]\arrow[Rightarrow, ld, start anchor={[xshift=-1ex, yshift=-1ex]}, end anchor={[xshift=1.2ex, yshift=1.2ex]}]\\
{}\arrow[r] & {}\arrow[r] & {}
\end{tikzcd}
\rt
\begin{tikzcd}[column sep=2.7pc]
{} \arrow[r, bend left=45, start anchor=north east, end anchor=north west, ""{name=U1, below}]\arrow[r, bend left=-45, start anchor=south east, end anchor=south west, ""{name=D1, above}] & {} \arrow[r, bend left=45, start anchor=north east, end anchor=north west, ""{name=U2, below}]\arrow[r, bend left=-45, start anchor=south east, end anchor=south west, ""{name=D2, above}] & {} \arrow[Rightarrow, from=U1, to=D1]\arrow[Rightarrow, from=U2, to=D2]\end{tikzcd}$$
To see that this is a localisation, note that both the domain and codomain are functors $\Delta\times \Delta\rt \Cat_2$ satisfying the co-Segal conditions; it therefore suffices to show this when $[n]$ and $[m]$ are $0$ or $1$, where the result is easily verified.
\end{proof}

\subsection{Scaled unstraightening of Gray fibrations}\label{subsec:scunstrGray}
Let us now recall Lurie's straightening theorem for locally cocartesian fibrations over scaled simplicial sets.

\begin{proposition}\label{prop:exp enriched cat}
If $\tcat{C}$ is a marked simplicial category, then the marked simplicial category $\Fun^+(\tcat{C}, \markSet)^\circ$ of fibrant-cofibrant objects in the projective model structure on the enriched functor category $\Fun^+(\tcat{C}, \markSet)$ is weakly equivalent to $\tFun^\scale(N^{\scale}(\tcat{C}), \tcat{Cat}^\scale)$, where $\tFun^\scale(-, -)$ denotes the internal Hom in scaled simplicial sets. In other words, the projective model structure on $\Fun^+(\tcat{C}, \markSet)$ describes the $(\infty, 2)$-category of functors from $\tcat{C}$ to $\tcat{Cat}$.
\end{proposition}
\begin{proof}
This follows from \cite[Proposition A.3.4.13]{HTT} since $\markSet$ is an excellent model category by \cite[Example A.3.2.22]{HTT}.
\end{proof}
\begin{definition}
If $(X, S)$ is a scaled simplicial set and $p\colon E\rt X$ is a locally cocartesian inner fibration, then we say that $p$ is cocartesian over $S$ if for every $\sigma\colon [2]\rt X$ in $S$, the base change $\sigma^*E \rt [2]$ is a cocartesian inner fibration.
\end{definition}

\begin{theorem}[Lurie]
Let $(X, S)$ be a scaled simplicial set. Then there is a left proper combinatorial marked simplicial model structure on the slice category $\markSet/X^\sharp$ (where $X^\sharp$ denotes $X$ with all 1-simplices marked) such that the cofibrations are the monomorphisms, and
an object $(E, T)\rto{p} X^\sharp$ is fibrant if and only if
\begin{enumerate}
\item the underlying map of simplicial sets $p\colon E\rt X$ is a locally cocartesian inner fibration,
\item $T$ is precisely the set of locally $p$-cocartesian edges in $E$,
\item the locally cocartesian inner fibration $p$ is cocartesian over $S$.
\end{enumerate}
We write $\markSet_{(X,S)}$ for $\markSet/X^\sharp$ equipped with this model structure.
\end{theorem}

\begin{proof}
As a simplicial model category this is a special case of \cite[Theorem 3.2.6]{LurieGoo}, applied to the categorical pattern $(X, X_1, S, \emptyset)$. The marked simplicial enrichment follows from \cite[Remark 3.2.26]{LurieGoo}.
\end{proof}

\begin{theorem}[Lurie]\label{thm:scaled_unstr}
If $(X, S)$ is a scaled simplicial set, then there is a marked simplicial Quillen equivalence
$$\begin{tikzcd}
\Strsc_{(X,S)}\colon \markSet_{(X,S)}\arrow[r, yshift=0.5ex] & \Fun^+(\mathfrak{C}^{\scale}(X,S), \markSet) \colon \Unsc_{(X,S)}\arrow[l, yshift=-0.5ex]
\end{tikzcd}$$
where $\Fun^+(\mathfrak{C}^{\scale}(X,S), \markSet)$ is equipped with the projective model structure.
\end{theorem}

\begin{proof}
As an (unenriched) Quillen equivalence this follows from \cite[Theorem
3.8.1]{LurieGoo}. The compatibility with the simplicial enrichment is discussed in \cite[Remark 3.8.2]{LurieGoo}, and the same argument clearly extends to show that this is a marked simplicial adjunction. 
\end{proof}

This marked simplicial Quillen equivalence induces a weak equivalence between the underlying (fibrant) marked simplicial categories of fibrant-cofibrant objects, i.e.\ an equivalence of $(\infty, 2)$-categories. Combining this with Proposition \ref{prop:exp enriched cat}, we get:

\begin{corollary}\label{cor:scaled straightening homotopy invariant version}
Given any scaled simplicial set $(X,S)$, there is an equivalence of fibrant scaled simplicial sets
\[\pushQED{\qed}\tFun^\scale((X,S), \tcat{Cat}^\scale)\simeq \nerve^{\scale}\big((\markSet_{(X,S)})^\circ\big).\qedhere \popQED\]
\end{corollary}

\begin{remark}
The categories $\markSet_{(X,S)}$ only depend pseudonaturally on $(X,S)$, and therefore the equivalence in \cref{cor:scaled straightening homotopy invariant version} is not literally natural at the point-set level, but this can be dealt with in the same way as in the proof of the analogous statement for the usual unstraightening equivalence in \cite[Corollary A.32]{GHN}.
\end{remark}
Specialising to the case of a Gray tensor product of two scaled simplicial sets, we obtain the following:
\begin{proposition}\label{prop:locally cocart of gray tensor}
Straightening for locally cocartesian fibrations gives an equivalence between maps of scaled simplicial sets $(X,S)\Gtimes (Y,T) \rt \tcat{Cat}^\scale$ and locally cocartesian inner fibrations $E\rt X\times Y$ such that
\begin{enumerate}
\item for $x\in X$, the restriction $E_x\rt Y$ is cocartesian over $T$,
\item for $y\in Y$, the restriction $E_y\rt X$ is cocartesian over $S$,
\item for 1-simplices $\alpha\colon x\rt x'$ in $X$, $\beta\colon y\rt y'$ in $Y$, $p$ is cocartesian over the $2$-simplex $(s_1\alpha, s_0\beta)$.
\end{enumerate}
\end{proposition}
\begin{remark}
Condition (3) can be rephrased as follows: for any $e \in E_{x, y}$, if $e\rt (\alpha, \id_y)_!e$ is a locally cocartesian morphism over $(\alpha, \id_y)$, and $(\beta, \id_y)_!e\rt (\id_{x'},\beta)_!(\alpha, \id_y)_!e$ is a locally cocartesian morphism over $(\id_{x'},\beta)$, then the composite $e\rt (\id_{x'},\beta)_!(\alpha, \id_y)_!e$ is locally cocartesian over $(\alpha, \beta)$.
\end{remark}
\begin{proof}[Proof of \cref{prop:locally cocart of gray tensor}]
Corollary \ref{cor:scaled straightening homotopy invariant version} implies that there is a natural equivalence of $(\infty, 2)$-categories 
$$
\tFun^\scale((X,S)\Gtimes (Y,T), \tcat{Cat}^\scale)\simeq \nerve^{\scale}\big((\markSet)^\circ_{(X,S)\Gtimes (Y,T)}\big).
$$
Next one can apply the equivalence $X\Gtimes Y \simeq (X\times Y, T_{-})$ of  \cite[Proposition 2.10]{GHL-Gray} to obtain an equivalence $(\markSet)_{(X,S)\Gtimes (Y,T)}^\circ \simeq (\markSet)_{(X\times Y,T_{-})}^\circ$. The objects of the right hand side are exactly those of the proposition. However one can also show directly that the fibrant objects of $(\markSet)_{(X,S)\Gtimes (Y,T)}$ are precisely the locally cocartesian fibrations satisfying conditions (1)-(3) above. We will do this to keep the treatment self contained. By definition, the fibrant objects are locally cocartesian fibrations $p\colon E\rt X\times Y$ such that for $(\sigma, \tau)\in (S\times T)_{\oplax}$, the pullback $(\sigma, \tau)^*E\rt [2]$ is a cocartesian fibration. On the other hand, conditions (1)--(3) assert that $E$ is cocartesian over the subset of $2$-simplices $(S\times T)'\subseteq (S\times T)_{\oplax}$ given as follows:
\begin{itemize}
\item $(s_0^2x, \tau)$ with $x\in X_0$ and $\tau\in T$.
\item $(\sigma, s_0^2y)$ with $\sigma\in S$ and $y\in Y_0$.
\item $(s_1\alpha, s_0\beta)$ with $\alpha\in X_1, \beta\in Y_1$.
\end{itemize}
We claim that this already implies that $p$ is cocartesian over every $2$-simplex in $(S\times T)_{\oplax}$. Indeed, let us show that $p$ is cocartesian over $(\sigma, s_0\beta)$, for $\sigma\in S$ of the form
$$\begin{tikzcd}
x\arrow[r, "\kappa"]\arrow[rd, "\mu"{swap}] & x'\arrow[d, "\lambda"]\\
& x''
\end{tikzcd}$$
and $\beta\colon y\rt y'$ in $Y_1$; the case of a $2$-simplex $(s_1\alpha, \tau)$ will follow from a similar argument. Consider the $3$-simplex $\xi=(s_2\sigma, s_0^2\beta)$, which may be depicted as:
$$
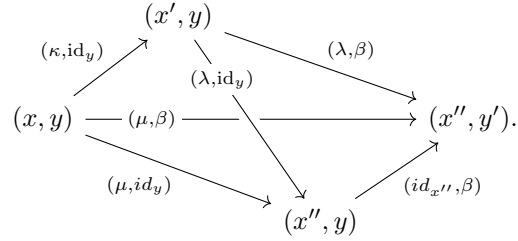
\begin{tikzcd}[column sep=1.5pc, row sep=1.8pc]
& {(x', y)}  \arrow[rrd, "{(\lambda, \beta)}"] & & \\
{(x, y)} \arrow[ru, "{(\kappa, \id_y)}"] \arrow[rrd, "{(\mu, id_y)}"']\arrow[rrr, "{(\mu, \beta)}"{description, pos=0.2}] &  &  & {(x'', y').} \\
&  & {(x'', y)}\arrow[luu, leftarrow, "{(\lambda, \id_y)}"{description, pos=0.75}, crossing over] \arrow[ru, "{(id_{x''}, \beta)}"'] &            
\end{tikzcd}$$
Note that $d_2\xi=(\sigma, s_0\beta)$, while the faces $d_0\xi=(s_1\lambda, s_0\beta),  d_1\xi=(s_1\mu, s_0\beta)$ and $d_3\xi=(\sigma, s_0^2y)$ are all in $(S\times T)'$. If $p\colon E\rt X\times Y$ is cocartesian over $(S\times T)'$, a locally cocartesian arrow $e\rt (\mu, \beta)_!e$ can therefore be identified in turn with the following composites of locally cocartesian arrows:
\begin{itemize}
\item  $e\to (\mu, \id_y)_!e\to (\id_{x''}, \beta)_!(\mu, \id_y)_!e$, since $p$ is cocartesian over $d_1\xi$,
\item $e\to (\kappa, \id_y)_!e\to (\lambda, \id_y)_!(\kappa, \id_y)_!e\to (\id_{x''}, \beta)_!(c, \id_y)_!e$, since $p$ is cocartesian over $d_3\xi$,
\item $e\to (\kappa, \id_y)_!e\to (\lambda, \beta)_!(\kappa, \id_y)_!e$, since $p$ is cocartesian over $d_0\xi$.
\end{itemize}
The last assertion means precisely that $p$ is cocartesian over $d_2\xi=(\sigma, s_0\beta)$, as desired.
\end{proof}

\begin{remark}
	The previous result also follows by combining \cite[Proposition 2.10]{GHL-Gray} with Lurie's scaled unstraightening (\Cref{thm:scaled_unstr}).
\end{remark}

Specialising to Gray tensor products of $\infty$-categories, we obtain the following:

\begin{corollary}\label{cor:gray fib is loc cocart of gray}
Let $A$ and $B$ be \icats{}. Then there is a natural equivalence of $\infty$-categories
$$
\Fun(A\Gtimes B, \tcat{Cat})\simeq \Gray(A, B).
$$
\end{corollary}
\begin{proof}
Combine Lemma \ref{lem:gray reform} and Proposition \ref{prop:locally cocart of gray tensor}.
\end{proof}

\begin{remark}
Similarly to \ref{item:annoyingsquareisannoying} of \cref{remark:annoyingsquare}, it is not a priori clear to us that the equivalence constructed in \cref{cor:gray fib is loc cocart of gray} restricts to the usual straightening equivalence
\[\Fun(A \times B, \Cat) \simeq \Cocart(A \times B):\] 
Besides this direct equivalence one can use the two inclusions \[\Cocart(A \times B) \subseteq \RCocart(A,B), \LCocart(A,B)\]
	and apply straightening in one factor after the other to obtain two more equivalences
\[\Cocart(A \times B) \simeq \Fun(B,\Cocart(A)) \simeq \Fun(A \times B,\Cat)\]
and
\[\Cocart(A \times B) \simeq \Fun(A,\Cocart(B)) \simeq \Fun(A \times B,\Cat),\]
and by construction the equivalence from \cref{cor:gray fib is loc cocart of gray} restricts to the latter of these.
Again it will follow from the uniqueness results of \cite{part2} that these three equivalences agree. 
\end{remark}

\subsection{Unstraightening of lax natural transformations and the calculus of mates}\label{subsec:laxnat}
As an application of the scaled unstraightening for Gray fibrations provided by Corollary \ref{cor:gray fib is loc cocart of gray}, we will now prove the main theorem of this section:

\begin{theorem}\label{thm:cocart vs functor lax transformations}
There are equivalences of $(\infty, 2)$-categories
\[\mathbf{Cocart}^{\lax}(B)\simeq \tcat{Fun}^{\lax}(B, \tcat{Cat}) \qquad \text{and} \qquad \mathbf{Cart}^{\oplax}(B)\simeq \tFun^{\oplax}(B^{\op}, \tcat{Cat}),\] which are natural in $B$.
\end{theorem}
In particular, this implies that the $(\infty, 2)$-categories $\mathbf{Cocart}^{\lax}(\ast)$ and $\mathbf{Cart}^{\oplax}(\ast)$ are both equivalent to $\tcat{Cat}$, as mentioned already after \cref{def:2-cat of cart oplax}.
\begin{remark}
In fact, it follows from \cite{part2} that the natural equivalence of Theorem \ref{thm:cocart vs functor lax transformations} is essentially unique.
\end{remark}
\begin{proof}
Let us start with the lax case, the oplax case being similar.

\subsubsection*{Lax case} Let $B$ be an \icat{} and consider the \itcat{} $\tFun^{\lax}(B, \tcat{Cat})$ determined by the natural equivalence of \igpds{}
$$
\Map_{\Cat_2}\big(\tcat{A}, \tFun^{\lax}(B, \tcat{Cat})\big)\simeq \Map_{\Cat_2}\big(\tcat{A}\Gtimes B, \tcat{Cat}\big).
$$
In terms of Segal \igpds{}, $\tFun^{\lax}(B, \tcat{Cat})$ is then described by the bisimplicial \igpd{}
$$\begin{tikzcd}[column sep=1.3pc]
N_{\Delta\times \Delta}\big(\tFun^{\lax}(B, \tcat{Cat})\big)\colon \Delta^{\op} \times \Delta^{\op} \arrow[r] & \Gpd ; \quad \big([m], [n]\big)\arrow[r, mapsto] & \Map_{\Cat_2}\Big(\delta_2([m], [n])\Gtimes B, \tcat{Cat}\Big),
\end{tikzcd}$$
where $\delta_2$ is the functor \eqref{diag:diagonal}. On the other hand, the \itcat{} $\tcat{Cocart}^{\lax}(B)$ was defined as the complete Segal \igpd{} whose value on $([m], [n])$ is given by the \igpd{} of functors (of \icats{}) $[m]\rt \cocart^{\lax}_{[n]}(B\times [n])$. Our goal will be to prove that there is a natural equivalence between these two bisimplicial \igpds{}.

To see this, let us first consider the following two natural subgroupoid inclusions:
$$\begin{tikzcd}[row sep=0.3pc, column sep=1.5pc]
\Map_{\Cat_2}\Big(\delta_2([m], [n])\Gtimes B, \tcat{Cat}\Big)\arrow[r, hook] & {\Map_{\Cat_2}\Big(([m]\Gtimes [n])\Gtimes B, \tcat{Cat}\Big)}\arrow[r, "\Unsc", hook] & \core\big(\Cat/[m]\times [n]\times B\big)\\
\Map_{\Cat}\big([m], \cocart^{\lax}_{[n]}(B\times [n])\big)\arrow[r, hook, "\Unco"] & \iota\Gray([m], B\times [n])\arrow[r, hook] & \core\big(\Cat/[m]\times B\times [n]\big).
\end{tikzcd}$$
The first inclusion uses the localisation $[m]\Gtimes [n]\rt \delta_2([m], [n])$ from Lemma \ref{lem:localization of Gray tensor product}. It suffices to verify that upon reversing the factors of $B$ and $[n]$, both of these inclusions determine the same subgroupoid of functors $p\colon E\rt B\times [m]\times [n]$. Unraveling the definitions, the image of the first map is the subgroupoid of functors $p$ with the following properties:
\begin{enumerate}[label=(\alph*)]
\item for each $j\in [n], b\in B$, the restriction $E_{j, b}\rt [m]$ is a cocartesian fibration.

\item for each $i\in [m], b\in B$, the restriction $E_{i, b}\rt [n]$ is a cocartesian fibration. 

\item for each $i\in [m], j\in [n]$, the restriction $E_{i, j}\rt B$ is a cocartesian fibration.

\item for arrows $\kappa\colon i\rt i'$, $\phi\colon j\rt j'$, $\beta\colon b\rt b'$ in $[m], [n]$ and $B$ respectively, $p$ is cocartesian over the $2$-simplices $(s_1\kappa, s_1\phi, s_0\beta)$ and $(s_1\kappa, s_0\phi, s_0^2\beta)$.

\item for any $i\in [m]$, the restriction $E_i\rt [n]\times B$ arises as the base change of a Gray fibration over $[0]\times B$; in particular, it is a cocartesian fibration, so this already implies (b) and (c).
\end{enumerate}
The first four conditions describe the image of $\Unsc$, by a twofold application of Proposition \ref{prop:locally cocart of gray tensor}. The fifth condition follows from Corollary \ref{cor:gray fib is loc cocart of gray} and Lemma \ref{lem:localization of Gray tensor product}, together with the fact that the Gray tensor product preserves colimits in each variable, so that
$$
\delta_2([m], [n])\Gtimes A\simeq \Big([m]\Gtimes [n]\coprod_{\core[m]\Gtimes [n]} \core[m]\Gtimes [0] \Big)\Gtimes A\simeq \big([m]\Gtimes [n]\big)\Gtimes A\coprod_{(\core[m]\Gtimes [n])\Gtimes A}\big(\core[m]\Gtimes [0]\big)\Gtimes A.
$$
On the other hand, unraveling Definition \ref{def:2-cat of cart oplax} shows that the image of the second map consists, after permuting $B$ and $[n]$, of functors $p$ with the following properties:
\begin{enumerate}[label=(\alph*')]
\item for each $j\in [n]$ and $b\in B$, the restriction $E_{j, b}\rt [m]$ is a cocartesian fibration.
\item for each $i\in [m]$, the restriction $E_i\rt [n]\times B$ is a cocartesian fibration.
\item for each $\kappa\colon i\rt i'$, $\phi\colon j\rt j'$ and $\beta\colon b\rt b'$ in $[m], [n]$ and $B$ respectively, $p$ is cocartesian over $(s_1\kappa, s_0\phi, s_0\beta)$.
\item for each $i\in [m]$, the restriction $E_i\rt [n]\times B$ is a cocartesian fibration which arises as the base change of a cocartesian fibration over $[0]\times B$ (in particular, this implies (b')).
\end{enumerate}
Indeed, by Lemma \ref{lem:gray reform} the first three conditions are equivalent to $p$ being an object of $\Gray([m], B\times [n])$. Condition (d') is then equivalent to the straightening of this Gray fibration over $[m]$ taking values in the full subcategory $\cocart^{\lax}_{[n]}(B\times [n])\subseteq \cocart^{\lax}(B\times [n])$.

Evidently, condition (e) is equivalent to (d') and conditions (a) and (a') coincide. It therefore remains to show that (d) and (c') are equivalent. Let us fix $\kappa\colon i\rt i'$, $\phi\colon j\rt j'$ and $\beta\colon b\rt b'$ as above and consider the $3$-simplex $\xi=\big(s_2s_1(\kappa), s_1s_0(\phi), s_0^2(\beta)\big)$ given by
$$\begin{tikzcd}[column sep=0.8pc, row sep=1.6pc]
& {(i', j, b)}  \arrow[rrd] & & \\
{(i, j, b)} \arrow[ru] \arrow[rrd]\arrow[rrr] &  &  & {(i', j', b').} \\
&  & {(i', j', b)}\arrow[luu, leftarrow, crossing over] \arrow[ru] &            
\end{tikzcd}$$
Note that $p$ is cocartesian over $d_0\xi$ by condition (e) (or (d')). Assuming condition (d), we furthermore have that $p$ is cocartesian over $d_3\xi$ and $d_1\xi$. The argument from Proposition \ref{prop:locally cocart of gray tensor} shows that $p$ is cocartesian over $d_2\xi$, which is precisely condition (c'). Conversely, condition (c') implies that $p$ is cocartesian over $d_2\xi$ and $d_3\xi$. An argument similar to that of Proposition \ref{prop:locally cocart of gray tensor} then shows that $p$ is cocartesian over $d_1\xi$ as well, so that (d) follows.

\subsubsection*{Oplax case}
Likewise, the $(\infty, 2)$-category $\tcat{Fun}^{\oplax}(B^{\op}, \tcat{Cat})$ corresponds to the $2$-fold Segal \igpd{}
$$\begin{tikzcd}
\big([m], [n]\big)\arrow[r, mapsto] & \Map_{\Cat_2}\Big(B^{\op}\Gtimes \delta_2([m], [n]), \tcat{Cat}\Big).
\end{tikzcd}$$
Lemma \ref{lem:localization of Gray tensor product} identifies this mapping \igpd{} with the \igpd{} of maps $B^{\op}\Gtimes [m]\Gtimes [n]\rt \tcat{Cat}$ whose restriction to each $B^{\op}\Gtimes \{i\}\Gtimes [n]$ arises from $B^{\op}\rt \tcat{Cat}$.

Under scaled unstraightening, this mapping \igpd{} is identified with a certain \igpd{} of locally cocartesian fibrations $p\colon E\rt B^{\op}\times [m]\times [n]$. Unraveling the definitions as in the lax case, one sees that this is the \igpd{} of those functors $p$ such that:
\begin{itemize}
\item denoting by $\mm{pr}_B$ the projection onto $B^{\op}$, we have that $p$ defines a map $\mm{pr}_B\circ p\rt \mm{pr}_B$ in $\cocart(B^{\op})$.

\item for each $b\in B^{\op}$, the map $E_b\rt [m]\times [n]$ is a Gray fibration.

\item for each $i\in [m]$, the Gray fibration $E_{i}\rt B^{\op}\times [n]$ arises as the base change of a cocartesian fibration over $B^{\op}\times [0]$.
\end{itemize}
Dualising over $B^{\op}$, i.e.\ applying cocartesian unstraightening and cartesian straightening over $B$, this is identified with the \igpd{} of maps $q\colon F\rt B\times [m]\times [n]$ such that
\begin{itemize}
\item $q$ defines a map $\mm{pr}_B\circ q\rt \mm{pr}_B$ in $\cart(B)$.

\item for each $b\in B$, the restriction $F_b\rt [m]\times [n]$ is a Gray fibration.
\item for each $i\in [m]$, the \lortho{} $F_{i}\rt B\times [n]$ arises as the base change of a cartesian fibration over $B\times [0]$.
\end{itemize}
Permuting $[m]$ and $[n]$, this is equivalent to $q$ defining an element in $\core\Ortholax(B\times [n], [m])$ such that for all $i\in [m]$, the map $F_i\rt B\times [n]$ arises as the base change of a cartesian fibration over $B\times [0]$.
Under straightening over $[m]$, this is precisely the \igpd{} of $(m, n)$-simplices of $\tcat{Cart}^{\oplax}(B)$ (see Definition \ref{def:2-cat of cart oplax}).
\end{proof}

\begin{observation}\label{obs:lax colim}
  Since the equivalence
  \[   \tcat{Fun}^{\lax}(B,
    \tcat{Cat}) \stackrel{\sim}{\rt} \mathbf{Cocart}^{\lax}(B)\]
  of \cref{thm:cocart vs functor lax transformations}
  is by construction given on objects by the unstraightening functor,
  for functors $F,G \colon B \to \tcat{Cat}$ we obtain an equivalence
  \[ \mathrm{Nat}^{\lax}(F,G) \simeq \Fun_{/B}(\Unco(F), \Unco(G)),\]
  depending 2-functorially on $F, G\in \tcat{Fun}^{\lax}(B, \tcat{Cat})$, where the left-hand side denotes the mapping \icat{} in $\tcat{Fun}^{\lax}(B,
  \tcat{Cat})$. Taking $G$ to be the constant functor with value  $X\in\tcat{Cat}$, we get an equivalence
  \[ \mathrm{Nat}^{\lax}(F,\mathrm{const}_{X}) \simeq \Fun(\Unco(F),
    X),\]
  depending $1$-functorially on $X\in\Cat$, since we have a natural equivalence $\Unco(\mathrm{const}_{X}) \simeq X \times B$. Let $\theta\colon F\to \mathrm{const}_{\Unco(F)}$ be the lax natural transformation corresponding to the identity map under the above natural equivalence and observe that this induces a natural transformation
  \[\begin{tikzcd}[column sep=2pc]
  \Fun(\Unco(F), -)\arrow[r, "\mm{const}"] & \mathrm{Nat}^{\lax}(\mathrm{const}_{\Unco(F)},\mathrm{const}_{(-)})\arrow[r, "\theta^*"] & \mathrm{Nat}^{\lax}(F,\mathrm{const}_{(-)})
  \end{tikzcd}\]
   between functors of $(\infty, 2)$-categories $\tcat{Cat}\rt \tcat{Cat}$. This natural transformation is a natural equivalence, since the underlying natural transformation between functors of $(\infty, 1)$-categories $\Cat\rt \Cat$ is an equivalence by definition of $\theta$. 
     
  In other
  words, $\Unco(F)$ has the universal property of the lax colimit
  of $F$: it corepresents the functor
  \[ \mathrm{Nat}^{\lax}(F,\mathrm{const}_{(-)}) \colon \tCat \to
    \tCat.\]
  Similarly, the cartesian unstraightening of $F \colon B \to \Cat$ is the oplax colimit:
  it satisfies
  \[ \mathrm{Nat}^{\oplax}(F,\mathrm{const}_{X}) \simeq \Fun(\Uncart(F),
    X).\]
Such a characterisation of the unstraightening was first established in \cite{GHN}, where the authors defined lax (co)limits for functors $F\colon B \rightarrow \Cat$ as certain weighted (co)limits. As an applicaton of \cref{thm:cocart vs functor lax transformations} we can therefore deduce that their lax colimits really have the desired universal property expressed above.
\end{observation}

\begin{definition}
For an \icat{} $B$, let us write $\tFun^{\lax, \mm{R}}(B, \tcat{Cat})\subseteq \tFun^{\lax}(B, \tcat{Cat})$ for the $1$-full sub-$2$-category spanned by the lax natural transformations sending each object in $B$ to a right adjoint. Likewise, let $\tFun^{\oplax, \mm{L}}(B, \tcat{Cat})\subseteq \tFun^{\oplax}(B, \tcat{Cat})$ for the $1$-full sub-$2$-category on oplax natural transformations $B\rt \Cat$ with values in left adjoints.
\end{definition}
Combining Theorem \ref{thm:dualizing lax natural transformations} and Theorem \ref{thm:cocart vs functor lax transformations}, we obtain the following:
\begin{theorem}\label{thm:mates}
Let $B$ be an \icat{}. Then there is an equivalence
$$
\Adj\colon \tFun^{\lax, \mm{R}}(B, \tcat{Cat})\rto{\sim} \big(\tFun^{\oplax, \mm{L}}(B, \tcat{Cat})\big)^{(1, 2)-\op}
$$
sending each lax natural transformation $F \Rightarrow G$ with values in right adjoints to the corresponding oplax natural transformation $G \Rightarrow F$ with values in left adjoints.
\end{theorem}
\begin{proof}
Unravelling the proof of Theorem \ref{thm:cocart vs functor lax transformations}, one sees that the equivalence $\tFun^{\lax}(B, \tcat{Cat})\simeq \tcat{Cocart}^{\lax}(B)$ is given at the level of objects by the usual unstraightening from \cite{HTT}, which in this special agrees with the locally cocartesian unstraightening from \cite{LurieGoo}. In particular, this equivalence identifies the $1$-full sub-$2$-category $\tFun^{\lax, \mm{R}}(B, \tcat{Cat})$ with $\tcat{Cocart}^{\lax, \mm{R}}(B)$, and similarly for the oplax case. The result then follows from Theorem \ref{thm:cocart vs functor lax transformations}.
\end{proof}

Denoting by $\mathbf{LFun}^{\mathrm{lax}}(B,\mathbf{Cat})$ the full sub-2-category of $\mathbf{Fun}^{\mathrm{lax}}(B,\mathbf{Cat})$ spanned by all functors taking values in left adjoints and similarly for $\mathbf{RFun}^{\mathrm{opl}}(B,\mathbf{Cat})$ we also find the following.
\begin{theorem}
Let $B$ be an $\infty$-category. Then there is an equivalence
\[\mathbf{RFun}^{\mathrm{opl}}(B,\mathbf{Cat}) \xrightarrow{\sim}  \mathbf{LFun}^{\mathrm{lax}}(B^{\mathrm{op}},\mathbf{Cat})\]
sending each diagram $B \rightarrow \mathrm{Cat}$ with values in right adjoints to the corresponding diagram $B^\mathrm{op} \rightarrow \mathrm{Cat}$ of left adjoints.
\end{theorem}
\begin{proof}
Simply observe that both sides are equivalent to $\mathbf{Bicart}^{\mathrm{(op)lax}}(B)$ via the unstraightening equivalence from Theorem \ref{thm:cocart vs functor lax transformations}.
\end{proof}

\bibliographystyle{amsalpha} 

\end{document}